\documentclass[11pt]{amsart}
\usepackage{amsmath,amssymb,amsfonts,amsthm}
\usepackage{thmtools}
\declaretheorem[parent=section]{theorem}
\declaretheorem[sibling=theorem]{proposition}
\declaretheorem[sibling=theorem]{lemma}
\declaretheorem[sibling=theorem]{corollary}

\declaretheorem[sibling=theorem,style=remark]{remark}

\declaretheorem[sibling=theorem,style=definition]{definition}
\declaretheorem[numbered=no,name=Theorem]{theorem*}
\declaretheorem[numbered=no,style=definition,name=Definition]{definition*}
\declaretheorem[numbered=no,name=Proposition]{proposition*}

\usepackage{hyperref}
\hypersetup{backref,colorlinks=true,allcolors=black}
\usepackage{amssymb}
\usepackage[foot]{amsaddr}
\usepackage{enumerate}
\usepackage{xcolor}
\usepackage{mathtools}
\usepackage{booktabs}
\usepackage{dsfont}
\usepackage{mathrsfs}
\usepackage{tikz}
\usetikzlibrary{matrix,arrows}
\usepackage{thmtools}
\usepackage{mathrsfs}
\usepackage[T1]{fontenc}
\usepackage[utf8]{inputenc}

\usepackage{parskip}

\usepackage[
activate={true,nocompatibility}
,final
,tracking=true
,kerning=true
,spacing=true
,factor=1100
,stretch=10
,shrink=10
]{microtype}

\usepackage[a4paper,left=1.05in,right=1.05in,top=1.1in,bottom=1.1in]{geometry}

\begin{document}

\bibliographystyle{amsalpha}

\renewcommand{\d}[1]{\,\ensuremath{\mathrm{d}} #1}
\renewcommand{\L}{\operatorname{L}}
\renewcommand{\P}{\mathbb{P}}
\newcommand{\D}{\operatorname{D}}
\newcommand{\Id}{\operatorname{Id}}
\newcommand{\I}{\operatorname{I}}
\newcommand{\var}{\operatorname{Var}}
\newcommand{\cov}{\operatorname{Cov}}
\newcommand{\mi}{\ensuremath{\mathrm{i}}}
\newcommand{\me}{\ensuremath{\mathrm{e}}}
\newcommand{\N}{\mathbb{N}}
\newcommand{\Z}{\mathbb{Z}}
\newcommand{\Q}{\mathbb{Q}}
\newcommand{\R}{\mathbb{R}}
\newcommand{\C}{\mathbb{C}}
\newcommand{\on}[1]{\operatorname{#1}}
\newcommand{\todo}[1]{{\textcolor{blue}{\textbf{TODO: }#1}}}
\newcommand{\tr}{\on{tr}}
\newcommand{\norm}[1]{\left\lVert #1 \right\rVert}
\newcommand{\Ex}[1]{\operatorname{E}\! \left[ #1 \right]}
\newcommand{\abs}[1]{\left\vert #1 \right\rvert}
\newcommand{\Ker}[1]{\operatorname{ker}\left(#1\right)}
\newcommand{\proj}{\widehat\otimes_\pi}
\newcommand{\gams}{\gamma^s}
\newcommand{\gam}{\gamma}
\newcommand{\E}{\mathbb{E}}
\newcommand{\Sym}{\mathrm{Sym}}
\newcommand{\ip}[2]{\left\langle #1,#2 \right\rangle}
\newcommand{\1}{\mathbf{1}}
\newcommand{\mysquarebrackets}[1]{[#1]}

\author{Solesne Bourguin$^1$}
\address{$^1$Université Gustave Eiffel, Université Paris Est
  Créteil, CNRS, LAMA UMR8050, F-77447 Marne-la-Vallée, France}
\author{Simon Campese$^2$}
\address{$^2$Institute of Mathematics, Hamburg University of Technology,
  Am Schwarzenberg-Campus 3, D-21073 Hamburg, Germany}
\title[Banach-valued carr\'{e} du champ and Gaussian approximation]{Banach-Valued Carr\'{e} du Champ Operators and Functional Gaussian Approximation}
\begin{abstract}
We develop a projective-tensor-valued carré du champ for Malliavin-differentiable Banach-valued random elements and establish an associated integration-by-parts formula. The resulting dimension-free comparison principle bounds Gaussian approximation against twice Fréchet-differentiable test functions by the expected projective norm of a random covariance discrepancy. For processes on Skorokhod space equipped with the uniform norm, a smoothing argument converts this comparison into a quantitative bound in the bounded-Lipschitz metric, subject to explicit regularization errors. On Banach-valued Wiener chaos we express the carré du champ through contractions and obtain quantitative estimates in terms of contracted kernels. As a concrete application, we establish a quantitative functional Breuer--Major bound for the canonical polygonally interpolated partial-sum process, together with a corollary for the original Skorokhod step process. The estimate applies to the entire path in the uniform topology under explicit covariance conditions. We also record the trace-class specialization of the contraction bounds in Hilbert spaces.
\end{abstract}
\maketitle

\section{Introduction}
\begingroup
The Malliavin--Stein method combines the Gaussian integration-by-parts formula with Stein's method to obtain quantitative normal approximation for nonlinear functionals of Gaussian fields. Its starting point is a comparison between the covariance of a
target Gaussian vector and the so-called Malliavin--Stein matrix. For a
centered Malliavin-differentiable random vector
$F=\left(F_1,\ldots,F_n\right)$, the entries of this matrix are
\begin{equation*}
\Gamma\left(F_i,-L^{-1}F_j\right)
=\left\langle DF_i,-DL^{-1}F_j\right\rangle_{\mathfrak{H}},
\qquad 1\leq i,j\leq n.
\end{equation*}
Beginning with the foundational one-dimensional
work~\cite{nourdin-peccati:2009:steins-method-wiener} of Nourdin and Peccati, this principle has led to multivariate Gaussian bounds and estimates in terms of contractions of multiple Wiener--It\^{o} kernels. We refer to~\cite{nourdin-peccati-reveillac:2010:multivariate-normal-approximation,nourdin-peccati:2012:normal-approximations-malliavin} for the multivariate developments and a systematic account.

Extending such estimates to stochastic processes poses two distinct problems. First, bounds for a growing number of finite-dimensional marginals need not be uniform in the dimension, while convergence of these marginals alone gives no quantitative control of the laws of entire trajectories. Second, bounds for sufficiently smooth functions on an infinite-dimensional space do not automatically yield bounds in a metric that characterizes weak convergence. Functional forms of Stein's method go back to~\cite{barbour:1990:steins-method-diffusion}. In~\cite{shih:2011:stein-abstract-wiener}, an infinite-dimensional Stein equation was developed in the setting of abstract Wiener spaces, yielding quantitative Gaussian approximation using exchangeable pairs, including bounds for sums of independent Banach-valued random variables.

A related program by Coutin and Decreusefond uses Stein's method for quantitative approximations of processes. In~\cite{coutin-decreusefond:2013:stein-brownian}, Brownian approximation is studied, including Poisson-to-Brownian and Donsker-type limits. Subsequent works establish rates for enhanced random walks in fractional Sobolev rough-path topologies~\cite{coutin-decreusefond:2020:stein-rough-paths} and quantitative Donsker bounds in Wasserstein-1 distance~\cite{coutin-decreusefond:2020:donsker-wasserstein}. In~\cite{besancon-coutin-decreusefond-moyal:2024:diffusive-limits}, the stability of Wasserstein-type bounds under Lipschitz transformations is used to derive diffusive approximations for Poisson-driven systems, including Markov processes and Hawkes processes. Together, these works illustrate how the choice of path topology and probability distance shapes the available quantitative functional limit theorems.

A complementary, Malliavin-based line of research approaches infinite-dimensional Gaussian approximation through operator-valued covariance discrepancies. In~\cite{bourguin-campese:2020:approximation-hilbert-valued-gaussians}, we developed Hilbert-valued $\Gamma$-calculus and dimension-free estimates for smooth Gaussian comparison, including a contraction analysis on Wiener chaos. Related developments include the Hilbert--Poisson framework in~\cite{bourguin-campese-dang:2024:hilbert-poisson} and the Hilbert-space fourth-moment results in~\cite{duker-zoubouloglou:2025:fourth-moment-theorem-hilbert}. The complementary analysis of smooth probability distances in~\cite{bassetti-bourguin-campese-ea:2025:caveat-metrizing-convergence}, building on questions already present in~\cite{gine-leon:1980:central-limit-theorem}, makes precise how the norm imposed on second derivatives interacts with the topology of convergence in infinite dimensions. In particular, the trace-class norm, rather than the Hilbert--Schmidt norm, is the natural Hilbert-space counterpart of the covariance control pursued here.

A different route to path-space metrics consists of deriving estimates for smooth functionals and then regularizing less regular tests. Such a smoothing principle for c\`adl\`ag processes was established in~\cite{barbour-ross-zheng:2024:steins-method-smoothing}, permitting comparisons in the bounded-Lipschitz distance
associated with the uniform topology. In~\cite{balasubramanian-goldstein-ross-salim:2024:gaussian-random-fields}, Gaussian smoothing was further developed for random fields on more general index sets, with applications to wide random neural networks. Recent functional Stein bounds for dynamical systems~\cite{leppanen-nakajima-nakano:2025:brownian-smooth} also illustrate the continuing role of smooth test-function metrics. Our purpose is to combine a Banach-valued Malliavin covariance calculus with a path-space regularization that leads to a convergence-determining metric.

\noindent\textbf{Main contributions.} The mathematical core of this paper is a projective-tensor-valued carr\'{e} du champ for Banach-valued Wiener functionals, an associated integration-by-parts identity, and a dimension-free Gaussian comparison principle. The principal functional approximation theorem, stated explicitly below, combines this comparison with regularization to control the bounded-Lipschitz distance between laws of entire trajectories. A further general result provides contraction identities and projective-norm bounds for Banach-valued multiple Wiener--It\^{o} integrals. The final functional Breuer--Major theorem illustrates the path-space comparison by a direct projective-tensor calculation tailored to stationary Gaussian sequences.

For a Banach space $E$, covariance need not be represented by an operator on $E$ itself. Instead, a square-integrable $E$-valued random element has a covariance tensor in $E\widehat\otimes_{\pi}E$, with the projective norm providing a natural replacement for the trace-class norm. Using the Banach-valued Malliavin calculus based on $\gamma$-radonifying operators~\cite{maas-neerven:2008:clark-ocone-formula-umd,maas:2010:malliavin-calculus-decoupling,pronk-veraar:2014:tools-malliavin-calculus}, we construct a random tensor $\Gamma_{\pi}\left(X,-L^{-1}X\right)$ whose expectation equals the covariance tensor of $X$. The essential new ingredient is an integration-by-parts formula with values in this projective tensor product. Coupled with the smart-path interpolation, it yields the dimension-free comparison
\begin{equation}
\label{eq:42}
\abs{\E\left(f\left(X\right)\right)-\E\left(f\left(Z\right)\right)}
\leq \frac12\norm{f''}_{\infty}
\E\left(\norm{\Gamma_{\pi}\left(X,-L^{-1}X\right)-Q_Z}_{\pi}\right),
\end{equation}
where $X\in\mathbb{D}^{1,2}\left(E\right)$ is centered, $Z$ is a centered Radon Gaussian random element with covariance tensor $Q_Z$, and $f$ is an admissible twice Fr\'{e}chet-differentiable function (see Proposition~\ref{smartpathproposition}). This formulation puts the covariance and the Malliavin discrepancy in the same Banach tensor space. It also makes transparent the reduction to the trace-class framework when $E$ is Hilbert.

The projective formulation also has a useful consequence for Banach-valued Wiener chaos. For multiple Wiener--It\^{o} integrals we obtain expansions of $\Gamma_{\pi}$ in terms of vector-valued kernel contractions and quantitative bounds for their projective norms (Section~\ref{sec:chaos}). The argument uses the Banach-valued chaos and decoupling theory developed in~\cite{maas:2010:malliavin-calculus-decoupling}. Unlike the finite-dimensional theory, the norm of a contracted kernel must be related to the $\gamma$-radonifying structure of the corresponding vector-valued multiple integral. These bounds complement the functional comparison theorem by describing its projective discrepancy in terms of Wiener-chaos kernels.

For processes taking values in the Skorokhod space equipped with the uniform norm, we combine the smooth comparison~\eqref{eq:42} with the regularization of~\cite{barbour-ross-zheng:2024:steins-method-smoothing} to bound the bounded-Lipschitz distance, with explicit approximation errors. The resulting path-space comparison takes the following explicit form.
\begin{samepage}
\begin{theorem}[Functional Gaussian comparison]
\label{thm:intro-main}
Let $T>0$, $d\in\N$, and let $\mathcal{D}=D\left(\left[0,T\right];\R^d\right)$ carry the uniform norm. Let $X$ and $Z$ be centered $\mathcal{D}$-valued Radon random elements, with $X\in\mathbb{D}^{1,2}\left(\mathcal{D}\right)$ and $Z$ a Radon Gaussian random element with almost surely continuous paths. Put
\begin{equation*}
\Delta\left(X,Z\right)=\E\left(\norm{\Gamma_{\pi}\left(X,-L^{-1}X\right)-Q_Z}_{\mathcal{D}\widehat\otimes_{\pi}\mathcal{D}}\right).
\end{equation*}
Then, for every $\varepsilon>0$,
\begin{equation*}
d_{\mathrm{BL}}\left(X,Z\right)
\leq\E\left(\norm{X_{\varepsilon}-X}_{\infty}\right)
+\E\left(\norm{Z_{\varepsilon}-Z}_{\infty}\right)
+C_{T,d}\varepsilon^{-2/3}\Delta\left(X,Z\right)^{1/3}.
\end{equation*}
If, in addition, both expected regularization errors are bounded by $M\varepsilon^{\alpha}$ for all $\varepsilon>0$, for some $M,\alpha>0$, then
\begin{equation*}
d_{\mathrm{BL}}\left(X,Z\right)
\leq C_{T,d,M,\alpha}\Delta\left(X,Z\right)^{\alpha/\left(3\alpha+2\right)}.
\end{equation*}
\end{theorem}
\end{samepage}
This is a convenient formulation of the results established in Section~\ref{sec:Skorokhod}, where the regularized processes $X_{\varepsilon}$ and $Z_{\varepsilon}$ are defined. The bounded-Lipschitz distance, in contrast with the intermediate smooth-test bound, characterizes weak convergence on the continuous-path subspace endowed with the uniform norm. The regularization errors quantify the cost of moving from the differentiable test-function class to this stronger conclusion.

To illustrate the use of the projective-tensor discrepancy for a classical dependent process, Section~\ref{applicationSection} considers the genuine polygonally interpolated Hermite partial sums in the Breuer--Major setting. The classical Breuer--Major central limit theorem (see~\cite{breuer-major:1983:central-limit-theorems}) and its functional extension~\cite{nourdin-nualart:2020:functional-breuer-major} establish convergence under covariance summability conditions, while quantitative scalar bounds are available, notably, from~\cite{nourdin-peccati-podolskij:2011:quantitative-breuer-major}. Related Hilbert-valued versions are studied in~\cite{duker-zoubouloglou:2024:breuer-major-hilbert}. Here the emphasis is different. We obtain a bound for the entire interpolated path in the bounded-Lipschitz distance for the uniform topology. More precisely, for a stationary standard Gaussian sequence with covariance $\rho$, Hermite order $p\geq2$, and positive limiting variance, the conditions $\rho\in\ell^1\left(\Z\right)$ and $\sum_{k\in\Z}\abs{k}\abs{\rho\left(k\right)}^p<\infty$ yield, for every $\eta>0$,
\begin{equation*}
d_{\mathrm{BL}}\left(F_n,B\right)=O\left(n^{-1/14+\eta}\right).
\end{equation*}
We also obtain explicit rates for fractional Gaussian noise with $H<3/4$ and a corollary for the original c\`adl\`ag partial-sum process. This application is intended to demonstrate the projective carré du champ method on a natural stochastic process. Establishing the minimal covariance assumptions or the optimal rate for this particular model is not its purpose.

The rest of the paper is organized as follows. We begin in
Section~\ref{sec:prelim} by recalling the necessary preliminaries
on tensor products, Gaussian measures and Malliavin calculus.
The projective carré du champ is then introduced in
Section~\ref{sec:carre-du-champ}, where we also establish the
corresponding integration-by-parts formula and Gaussian comparison
results. In Section~\ref{sec:Skorokhod}, we develop the regularization
procedure leading to quantitative functional Gaussian approximation
in the bounded-Lipschitz distance. Section~\ref{sec:chaos} is devoted
to contraction estimates for Banach-valued Wiener chaos, including
their specialization to Hilbert spaces. Finally,
Section~\ref{applicationSection} illustrates our approach through
a quantitative functional Breuer--Major theorem.
\endgroup

\section{Preliminaries}
\label{sec:prelim}
\subsection{Generalities}
Let $E$ be a real Banach space. If not mentioned otherwise, all occuring random elements are defined on a common probability space $(\Omega,\mathcal F,\mathbb P)$, and $\mathbb{E}$ denotes mathematical expectation with respect to $\mathbb{P}$.

All vector spaces in this paper are assumed to be real.
For a Banach space $E$, let $U_{E}$ be its closed unit ball and $E^{\ast}$ its topological dual. If $F$ is another Banach space, their topological direct sum is denoted by $E \oplus F$ and equipped with the norm $\left\lVert u \oplus v \right\rVert = \left\lVert u \right\rVert_{E} + \left\lVert v \right\rVert_{F}$ to turn it into a Banach space. If a closed subspace $U$ of $E$ is complemented by another closed subspace $V$, then $E$ itself is isomorphic to the topological direct sum $U \oplus V$. Note that in particular, the natural projections onto $U$ and $V$ are both continuous.

For vector spaces $E_1,E_2,\dots,E_k,F$, we denote by $L(E_1,E_2,\dots,E_k;F)$ the space of $k$-linear $F$-valued mappings on $E_1 \times E_2 \times \dots \times E_{k}$ and by $\mathcal{F}(E_1,E_2,\dots,E_k;F)$ the subspace of finite-rank mappings. If the vector spaces are normed, then $\mathcal{L}(E_1,E_2,\dots,E_k;F)$ denotes those $k$-linear mappings which are bounded (i.e. continuous with respect to the norm  $\left\lVert \varphi \right\rVert = \sup_{x_{j} \in U_{j}} \left\lVert \varphi(x_1,\dots,x_k) \right\rVert_F$), so that in particular $\mathcal{F}(E_1,\dots,E_k;F) \subseteq \mathcal{L}(E_1,\dots,E_k;F)$. We write $L_k(E,F) = L(\underbrace{E,E,\dots,E}_{\text{$k$ times}};F)$ and $L(E,F) = L_1(E,F)$, with analogous shorthands for the respective bounded or finite-rank subspaces.

We write $\mathcal{C}^k(E,F)$ for the space of $k$-times continuously Fréchet differentiable functions, and $\mathcal{C}_b^{k}(E,F)$ for the subspace of those $f \in C^k(E,F)$ such that $f$ and its first $k$ derivatives are bounded. The Fréchet derivative of $f$ is denoted by $f'$.

\subsection{Topological tensor products}
\label{sec:topol-tens-prod}

We denote by $E \otimes F$ the algebraic tensor product of $E$ and $F$, by $E \otimes_{\alpha} F$ the corresponding normed space, where $\alpha$ is some norm on $E \otimes F$, and by $E \widehat{\otimes}_{\alpha} F$ its completion. A norm $\alpha$ on the class of normed spaces is called a tensor norm (also called uniform cross norm), if $\alpha$ is reasonable, i.e. $\varepsilon \leq \alpha \leq \pi$, where $\varepsilon$ and $\pi$ are the injective and projective norms, respectively, and the metric mapping property is satisfied: for normed spaces $E_i,F_i$ and bounded operators $T_i \in \mathcal{L}(E_i,F_i)$, $i=1,2$, the tensor operator $T_1 \otimes T_2 \colon E_1 \otimes_{\alpha} E_2 \to F_1 \otimes_{\alpha} F_{2}$ is bounded and satisfies
\begin{equation*}
  \left\lVert T_1 \otimes T_2 \right\rVert \le \left\lVert T_1 \right\rVert \left\lVert T_2 \right\rVert,
\end{equation*}
where all norms are the usual operator norms. In particular, for simple tensors $x \otimes y  \in E \otimes F$, one has $\left\lVert x \otimes y \right\rVert_{\alpha} = \left\lVert x \right\rVert_E \left\lVert y \right\rVert_{F}$. The Hilbert cross norm $\sigma$ (obtained from the inner product induced by defining $\left\langle x_1 \otimes y_1, x_2 \otimes y_2  \right\rangle_{\sigma} = \left\langle x_1,x_2 \right\rangle_E \left\langle y_1,y_2 \right\rangle_F$) is a tensor norm in the above sense as well, albeit of course only on pairs of Hilbert spaces. With the above notation, the algebraic tensor product of two Hilbert spaces $\mathfrak{H}$ and $\mathfrak{K}$ equipped with the tensor norm $\sigma$ is denoted by $\mathfrak{H} \otimes_{\sigma} \mathfrak{K}$, the Hilbert space obtained upon completing by $\mathfrak{H} \widehat{\otimes}_{\sigma} \mathfrak{K}$. As is well known, this Hilbert space is isometrically isomorphic to the Hilbert-Schmidt operators.

The symmetric tensor product of $E$ and $F$ is denoted by $E \odot F$. For a positive integer $p$ and an associative tensor norm $\alpha$ (such as the Hilbert cross norm $\sigma$), the notations $E^{\otimes p}$, $E^{\otimes_{\alpha} p}$ and $E^{\widehat{\otimes}_{\sigma} p}$ as well as analogous versions for the symmetric tensor powers have their obvious meaning.

We will mainly be concerned with the projective tensor product $E \widehat{\otimes}_{\pi} F$, which is obtained from the projective tensor norm $\norm{\cdot}_{\pi}$, given for $u \in E \otimes F$ by
\begin{equation*}
\norm{u}_{\pi} = \inf \left\{ \sum_{k=1}^m
  \norm{x_k}_{E}\norm{y_k}_{F}\colon u=\sum_{k=1}^m x_k\otimes y_k \text{ in } E \otimes F \right\},
\end{equation*}
where the infimum is taken over all finite representations of $u$ as a
sum of simple tensors. Any element $u \in E \widehat{\otimes}_{\pi} F$ has representations of the form $u = \sum_{k=1}^{\infty} x_k \otimes y_k$ with $x_k \in E$ and $y_k \in F$ such that $\sum_{k=1}^{\infty} \left\lVert x_k \right\rVert_{E} \left\lVert y_k \right\rVert_{F} < \infty$. The dual of the projective tensor product can be identified with the space of bounded bilinear forms $\mathcal{L}(E,F;\mathbb{R})$, where the duality pairing $(\varphi,u)_{\pi}$ between $\varphi  \in \mathcal{L}(E,F;\mathbb{R})$ and a tensor $u = \sum_{k=1}^m x_{k} \otimes y_k  \in E \otimes F$ is defined as
\begin{equation}
 \label{eq:9}
  (\varphi, u)_{\pi} = \sum_{k=1}^m \varphi(x_k,y_k).
\end{equation}

As a tensor norm $\alpha$ satisfies $\alpha \leq \pi$, the dual of $E \widehat{\otimes}_{\alpha} F$ can be identified with the subspace of those bounded bilinear forms $\mathcal{L}(E,F;\mathbb{R})$ which are also continuous with respect to the dual norm defined by $\alpha$. Furthermore, by the dual representation of a norm, it holds that
\begin{equation}
 \label{eq:23}
  \left| (\varphi,u)_{\alpha}  \right| \leq \left\lVert \varphi \right\rVert_{\left( E \widehat{\otimes}_{\alpha} F \right)^{\ast}} \, \left\lVert u \right\rVert_{\alpha}.
\end{equation}
For $u \in E \otimes F$ and $\varphi  \in \left( E \widehat{\otimes}_{\alpha} F \right)^{\ast}$, the pairings $(\varphi,u)_{\alpha}$ and $(\varphi,u)_{\pi}$ coincide in value and can be computed via~\eqref{eq:9}.

For the projective norm, the dual representation reads
\begin{equation*}
  \left\lVert u \right\rVert_{\pi} = \sup_{\varphi \in U_{{\mathcal{L}(E,F;\mathbb{R})}}} \left| \left( \varphi, u \right)_{\pi} \right|
\end{equation*}
and we can write~\eqref{eq:23} more explicitly as
\begin{equation*}
 \left| (\varphi,u)_{\pi} \right| \leq \left\lVert \varphi \right\rVert_{\mathcal{L}(E,F;\mathbb{R})} \left\lVert u \right\rVert_{\pi}
\end{equation*}
for all $u \in E \widehat{\otimes}_{\pi} F$ and $\varphi  \in \mathcal{L}(E,F;\mathbb{R})$.

A bounded linear operator $T \colon E \to F$, with $E$ and $F$ Banach spaces, is called nuclear, if there exist $f_k \in E^{\ast}$ and $y_k \in F$ such that
\begin{equation*}
  Tx = \sum_{k=1}^{\infty} \left( f_k,x \right)_{E^{\ast},E} y_k
\end{equation*}
for all $x \in E$ and
\begin{equation*}
  \sum_{k=1}^{\infty} \left\lVert f_k \right\rVert_{E^{\ast}} \, \norm{y_k}_F < \infty.
\end{equation*}
Such a representation is called nuclear. The nuclear norm $\left\lVert T \right\rVert_{\nu}$ of $T$ is defined as
\begin{equation*}
  \left\lVert T \right\rVert_{\nu} = \inf \sum_{k=1}^{\infty} \left\lVert f_k \right\rVert \left\lVert y_k \right\rVert,
\end{equation*}
where the infimum is taken over all nuclear representations $T = \sum_{k=1}^{\infty} f_k \otimes y_k$. The space of all nuclear operators is denoted by $\mathfrak{N}(E,F)$. The finite rank operators are densely contained.

By identifying a simple tensor $x \otimes y  \in E^{\ast} \otimes F$ with the rank one operator $T_{x \otimes y}  \in \mathcal{L}(E,F)$ defined by $T_{f \otimes y} x = f(x) y$, then extending by linearity and continuity, we obtain a metric surjection $J \colon E^{\ast} \widehat{\otimes}_{\pi} F \to \mathfrak{N}(E,F)$.
It is straightforward to see, that
\begin{equation}
  \label{eq:21}
  \left\lVert J(u) \right\rVert_{\nu} \leq \left\lVert u \right\rVert_{\pi}.
\end{equation}
Indeed, the infimum in the projective norm is taken over finite representations only, while for the nuclear norm also infinite representations are allowed. As is well known, $J$ might fail to be injective, so that it is in general not possible to identify $E^{\ast} \widehat{\otimes}_{\pi} F$ with $\mathfrak{N}(E,F)$. Although $J$ is injective on the algebraic tensor product, its nuclear norm need not equal the projective norm there: the nuclear norm allows representations that need not represent the same tensor. If $F$ has the approximation property, $J$ is an isometric isomorphism onto $\mathfrak{N}(E,F)$, and equality holds in~\eqref{eq:21}. We refer to~\cite[Chapters~5 and~29]{defant-floret:1993:tensor-norms-operator} for the approximation-property and trace results.

For the trace on the algebraic tensor product $E^{\ast} \otimes E$ and the finite-rank operators $\mathcal{F}(E,E)$, we will use the symbol $\operatorname{tr}$ (or $\operatorname{tr}_E$ if the Banach space should be emphasized). The trace on the projective tensor product $E^{\ast} \widehat{\otimes}_{\pi} E$ will be denoted by $\operatorname{tr}_{\pi}$ or $\operatorname{tr}_{E,\pi}$. If $J$ is injective, $\operatorname{tr}_{\pi}$ induces a trace on $\mathfrak{N}(E)=\mathfrak{N}(E,E)$, for which we will again use $\operatorname{tr}$ or $\operatorname{tr}_E$. Recall that the trace on $E^{\ast} \otimes E$ is the unique linearization of the evaluation map on $E^{\ast} \times E$, i.e. the bilinear form defined by $(f^{\ast},x) \mapsto \left( f^{\ast},x \right)_{E^{\ast},E}$. For an algebraic tensor
\begin{equation*}
  u = \sum_{k=1}^n f_k^{\ast} \otimes x_k  \in E^{\ast} \otimes E,
\end{equation*}
we hence get
\begin{equation*}
  \operatorname{tr} \left( u \right) = \sum_{k=1}^n \left( f_k^{\ast},x_k \right)_{E^{\ast},E}.
\end{equation*}

A distinct canonical map sends $E\widehat{\otimes}_{\pi}F$ continuously into $\mathfrak{N}\left(F^{\ast},E\right)$, assigning to $x\otimes y$ the operator $\varphi\mapsto\varphi\left(y\right)x$. Its norm is at most one, but it need not be surjective, even when $E$ and $F$ have the approximation property. The operator-valued covariance and contraction tensors used below are always obtained through this map. In particular, we do not identify $E\widehat{\otimes}_{\pi}F$ with the whole space $\mathfrak{N}\left(F^{\ast},E\right)$.

The relevant condition for injectivity of $J$ is the approximation property, which is in particular shared by all Hilbert spaces. If $K$ is a Hilbert space, we can thus identify $K \widehat{\otimes}_{\pi} K$ with $\mathfrak{N}(K)$, and furthermore the nuclear norm coincides with the trace norm in this case, so that $\mathfrak{N}(K)$ is the space of trace class operators. The trace itself can then of course also be computed using an orthonormal basis.

\subsection{Banach valued random elements and Bochner spaces}
\label{sec:banach-valued-random}

Let $(E,\mathcal{S})$ be a Banach space equipped with a $\sigma$-field $\mathcal{S}$ and $(\Omega,\mathcal{F},\mu)$ a probability space. A random element $X$ on $E$ is a measurable mapping $X \colon \Omega \to E$. Only the cylindrical $\sigma$-field $\mathcal{E}(E)$ (the minimal $\sigma$-field with respect to which all continuous linear functionals on $E$ are measurable) and the Borel $\sigma$-field $\mathcal{B}(E)$ on $E$ will be considered for $\mathcal{S}$ in this article. While these two $\sigma$-fields coincide if $E$ is separable, in general one has $\mathcal{E}(E) \subseteq \mathcal{B}(E)$.

A Borel random element $X$ is called Radon, if its law is Radon, i.e. a Borel probability measure on $E$ which is tight, in the sense that for any $\varepsilon>0$ there exists a compact set $K_{\varepsilon} \subseteq E$ such that $\mu \left( X \in K_{\varepsilon} \right) \geq 1- \varepsilon$. Equivalently, $X$ almost surely takes its values in a closed, separable subspace of $E$. The law of a Radon random element is uniquely determined by its values on $\mathcal{E}(E)$. As usual, we identify two random elements, if they agree almost everywhere.

For $p \in [1,\infty]$, the Bochner spaces of (equivalence classes of) $E$-valued Borel random elements of strong order $p$ will be denoted by $L^p(\Omega;E)$ or $L^p(\mu;E)$ if we want to emphasize the measure. If $E=\mathbb{R}$, we simply write $L^p(\Omega)$ or $L^p(\mu)$. As elements of $L^p(\Omega;E)$ are almost surely separably valued, they are Radon. Note that $\mathcal{E}(E)$-measurable random elements of strong order $p$ have unique Radon extensions to $\mathcal{B}(E)$ (see~ \cite[Prop. II.1.1]{vakhania-tarieladze-chobanyan:1987:probability-distributions-banach}).

The algebraic tensor product $L^p(\Omega) \otimes E$ is embedded in the Bochner space $L^p(\Omega;E)$ via $[f]_{\mu} \otimes x \mapsto [f \otimes x]_{\mu}$, where $[ f ]_{\mu}$ and $[ f \otimes x]_{\mu}$ denote the respective equivalence classes. The norm on $L^p(\Omega;E)$ thus induces a norm $\Delta_p$ on $L^p(\Omega) \otimes E$ and it is well known (see for example~\cite[ch. 7]{defant-floret:1993:tensor-norms-operator}), that $L^p(\Omega) \widehat{\otimes}_{\Delta_p} E \cong L^p(\Omega;E)$, that $\varepsilon \leq \Delta_p \leq \pi$ and that $\Delta_p$ is not a tensor norm, as it fails the metric mapping property: if $S \in \mathcal{L}(L^{p}(\mu), L^q(\nu))$ and $T \in \mathcal{L}(E,F)$, where $\nu$ is another finite measure, $F$ a Banach space and $q \in [1,\infty)$, then the corresponding tensor operator $S \otimes T \colon L^p(\mu) \otimes_{\Delta_p} E \to L^q(\nu) \otimes_{\Delta_q} F$ is not necessarily continuous, i.e. it is not guaranteed that
$\left\lVert S \otimes T \right\rVert \leq \left\lVert S \right\rVert \, \left\lVert T \right\rVert$.

However, there are several situations where continuity follows (see~\cite[ch. 7]{defant-floret:1993:tensor-norms-operator} or the survey~ \cite{defant-floret:1992:continuity-tensor-product}), one of which will be used in this paper:

\begin{theorem}
  \label{thm:1}
  Let $E$, $F$ be Banach spaces, $\mu$, $\nu$ be two measures, $p$, $q  \in [1,\infty]$ and $T \in \mathcal{L}(E,F)$. If $S \in \mathcal{L}(L^p(\mu),L^q(\nu))$ is positive (mapping a.e. positive functions to a.e. positive functions), then $S \otimes T \colon L^p(\mu) \otimes_{\Delta_p} E \to L^q(\nu) \otimes_{\Delta_q} F$ is continuous and one has $\left\lVert S \otimes T \right\rVert \leq \left\lVert S \right\rVert \, \left\lVert T \right\rVert$.
\end{theorem}

A proof can be found in~\cite[p.80]{defant-floret:1993:tensor-norms-operator}.

The cross-covariance operator (see~\cite[Ch. III.2.4]{vakhania-tarieladze-chobanyan:1987:probability-distributions-banach}) $R_{X,Y} \colon F^{\ast} \to E$ of an $E$-valued random element $X$ and an $F$-valued random element $Y$, both of strong second order, is defined by
\begin{equation*}
  R_{X,Y} = \mathbb{E} \left( (Y - \mathbb{E}(Y), \cdot)_{F,F^{\ast}} \left( X - \mathbb{E}(X) \right) \right)
   = \mathbb{E} \left( (Y, \cdot)_{F,F^{\ast}} \left( X - \mathbb{E}(X) \right) \right).
 \end{equation*}
 As is well known (see for example~\cite[Ch. III.2.3]{vakhania-tarieladze-chobanyan:1987:probability-distributions-banach}), the strong second order of $X$ and $Y$ implies that $R_{X,Y}$ is nuclear.

 In the same way, we define the cross covariance tensor $Q_{X,Y}\in E\widehat{\otimes}_{\pi}F$ by
\begin{equation*}
  Q_{X,Y}=\E\left(\left(X-\E\left(X\right)\right)\otimes Y\right).
\end{equation*}
If $X=Y$, the positive and symmetric covariance operator (resp. tensor) of $X$ is defined by $R_X = R_{X,Y}$ (resp. $Q_{X} = Q_{X,X}$).

\subsection{Radon Gaussian measures}
\label{sec:radon-gauss-meas}
Let $E$ be a Banach space. A measure $\gamma$  on $\mathcal{E}(E)$ is called Gaussian, if the measure $\gamma \circ f^{-1}$ is Gaussian for any $f \in E^{\ast}$. A Radon measure $\gamma$ on $\mathcal{B}(E)$ is called a Radon Gaussian measure, if its restriction to $\mathcal{E}(E)$ is a Gaussian measure. A Gaussian measure is called centered (resp. (non)-degenerate), if the Gaussian laws induced on the dual elements are centered (resp. (non)-degenerate).

An $E$-valued random element is called (Radon) Gaussian, centered and/or (non)-degenerate, if its induced law has the respective properties. Two Gaussian random elements $X$ and $Y$ taking values in Banach spaces $E$ and $F$, respectively, are jointly Gaussian, if the random element $X \oplus Y \in E \oplus F$, with $E \oplus F$ denoting the Banach direct sum of $E$ and $F$, is Gaussian, i.e. if for all $e^{\ast} \in E^{\ast}$ and $f^{\ast} \in F^{\ast}$, the distribution of $e^{\ast}(X) + f^{\ast}(Y)$ is Gaussian.

Let $\gamma$ be a centered Radon Gaussian measure and let $R_{\gamma} \colon E^{\ast} \to E$ be its covariance operator, defined in analogy to the covariance operator of a random element of strong order two as the Bochner integral
\begin{equation}
  \label{eq:22}
  R_{\gamma} f = \int_{E}^{} f(x) x \gamma(dx).
\end{equation}
Then, one has
\begin{equation*}
  \left\lVert R_{\gamma} \right\rVert_{\nu} \leq \int_E^{} \left\lVert x \right\rVert^2 \gamma(dx).
\end{equation*}
Denote by $E_{\gamma}^{\ast}$ the $L^2(\gamma)$-completion of the canonical embedding of $E^{\ast}$ into $L^2(\gamma)$. As $\norm{f}_{L^2\left(\gamma\right)}\leq\left(\int_E\norm{x}_E^2\gamma\left(dx\right)\right)^{1/2}\norm{f}_{E^{\ast}}$ for all $f \in E^{\ast}$, the operator $R_{\gamma}$ can also be defined (via~\eqref{eq:22}) on $E^{\ast}_{\gamma}$ and retains its nuclearity. The Cameron-Martin space $H_{\gamma}\subseteq E$ is defined as $R_{\gamma}(E_{\gamma}^{\ast})$. Equipped with the inner product inherited from $L^2(\gamma)$, i.e.
\begin{equation*}
  (R_{\gamma}f,R_{\gamma}g)_{H(\gamma)} = (f,g)_{L^2(\gamma)} = f \left( R_{\gamma}g \right),
\end{equation*}
it becomes a separable Hilbert space which is isometrically isomorphic to $E_{\gamma}^{\ast}$. Moreover, we have the embeddings
\begin{equation*}
  E^{\ast} \subseteq H^{\ast} \simeq H \subseteq E.
\end{equation*}
As shown in~\cite{kwapien-szymanski:1980:remarks-gaussian-measures} (see \cite{tarieladze:1980:nuclear-covariance-operators} for a different proof),  one can always find an orthonormal basis $\left\{ e_k \colon k \in \mathbb{N} \right\}$ of $H_{\gamma}$ such that
\begin{equation*}
  R_{\gamma} = \sum_{k=1}^{\infty} \left( e_k, \cdot \right)_{E,E^{\ast}} \otimes e_k
\end{equation*}
is a nuclear representation of $R_{\gamma}$, i.e.
\begin{equation*}
  \sum_{k=1}^{\infty} \left\lVert e_k \right\rVert^2_{E} < \infty.
\end{equation*}

\subsection{$\gamma$-radonifying operators}
\label{sec:gamma-radon-oper}

This section and the two following it essentially summarize most of the content of the first five sections of~\cite{maas:2010:malliavin-calculus-decoupling}, with some additions suitable for our purposes.

An isonormal Gaussian process $W$ on a real separable Hilbert space $\mathfrak{H}$ is a linear isometry $W \colon \mathfrak{H} \to L^2(\Omega, P)$, where $(\Omega,\mathcal{F},P)$ is some probability space, such that $W(h)$ is a centered Gaussian random variable for all $h \in \mathfrak{H}$.

The isonormal Gaussian process $W$ induces a Gaussian cylinder set measure on $\mathfrak{H}$, which has no $\sigma$-finite extension unless $\mathfrak{H}$ is finite-dimensional. However, one can always achieve such an extension on a larger Banach space $E$, into which $\mathfrak{H}$ is continuously embedded. Such an embedding is a $\gamma$-radonifying operator.

\begin{definition}
  \label{def:1}
  Let $W$ be an isonormal Gaussian process on a real, separable Hilbert space $\mathfrak{H}$ and let $E$ be a Banach space. A bounded linear operator $T \colon \mathfrak{H} \to E$ is called $\gamma$-radonifying, if there exists a strongly measurable Gaussian random variable $Z \in L^2(\Omega;E)$ such that
\begin{equation*}
  W(T^{\ast} f) =  f(Z)
\end{equation*}
for all $f  \in E^{\ast}$, where $T^{\ast}$ denotes the adjoint of $T$ and we have identified $\mathfrak{H}$ with its dual.
\end{definition}

A $\gamma$-radonifying operator $T$ as above therefore transports the Gaussian cylindrical distribution on $\mathfrak{H}$ induced by $W$ to a Gaussian cylindrical distribution on $E$, the latter of which can then be extended to a bona fide Radon Gaussian measure $\gamma_{T}$ on $E$.

To connect with the previous section, one can verify directly that the covariance operator $R_{\gamma_{T}}$ is given by $R_{\gamma_T} = T T^{\ast}$, where we have identified $\mathfrak{H}$ with its dual. Indeed, for $f,g \in E^{\ast}$ one has
\begin{multline*}
  g(R_{\gamma_T}f)
  =
  \mathbb{E} \left( f(Z) g(Z) \right)
  =
  \mathbb{E} \left( W(T^{\ast}f) W(T^{\ast}g) \right)
  \\ = \left( T^{\ast}f,T^{\ast}g \right)_{\mathfrak{H}^{\ast}}
  = \left( TT^{\ast}f,g \right)_{E,E^{\ast}}
  = g(TT^{\ast}f).
\end{multline*}

It can be shown that, given a $\gamma$-radonifying operator $T$, the Gaussian series $\sum_{k=1}^{\infty} W(\mathfrak{h}_k) T \mathfrak{h}_k$ converges in $L^2(\Omega;E)$ for some orthonormal basis $\left\{ \mathfrak{h}_k \colon k \in \mathbb{N} \right\}$ of $\mathfrak{H}$, and that the convergence and value of the sum does not depend on the choice of the isonormal Gaussian process or the orthonormal basis, so that
\begin{equation*}
  \left\lVert T \right\rVert_{\gamma(\mathfrak{H},E)} = \mathbb{E} \left(
    \left\lVert \sum_{n} W(\mathfrak{h}_{n}) T \mathfrak{h}_n \right\rVert_E^2
  \right)^{\frac{1}{2}}
\end{equation*}
defines a characterizing norm on the space $\gamma(\mathfrak{H},E)$ of all $\gamma$-radonifying operators, giving a Banach space that embeds continuously into $\mathcal{L}\left(\mathfrak{H},E\right)$, but is not generally closed for the operator norm. By approximating the above series by its partial sums, it is seen that the finite rank operators $\mathcal{F}(\mathfrak{H},E)$ are densely contained in $\gamma(\mathfrak{H},E)$ and that a $\gamma$-radonifying operator is hence compact. Moreover, $\gamma(\mathfrak{H},E)$ forms an operator ideal, i.e. if $T \in \gamma(\mathfrak{H},E)$, $S \in \mathcal{L}(\mathfrak{G},\mathfrak{H})$ and $R \in \mathcal{L}(E,F)$, where $\mathfrak{G}$ is a Hilbert and $F$ a Banach space, then $RTS \in \gamma(\mathfrak{G},F)$ and $\left\lVert RTS \right\rVert_{\gamma(\mathfrak{G},F)} \leq \left\lVert R \right\rVert_{\mathcal{L}(E,F)} \, \left\lVert T \right\rVert_{\gamma(\mathfrak{H},E)} \, \left\lVert S \right\rVert_{\mathcal{L}(\mathfrak{G},\mathfrak{H})}$.

Inductively, one defines the spaces $\gamma^{p}(\mathfrak{H},E) = \gamma \left(\mathfrak{H},\gamma^{p-1}(\mathfrak{H},E)\right)$ for $p \in \mathbb{N}$, where we set $\gamma^0(\mathfrak{H},E) = E$, so that $\gamma^1(\mathfrak{H},E) = \gamma(\mathfrak{H},E)$.

The density of the finite rank operators  $\mathcal{F}(\mathfrak{H},E)$ in $\gamma(\mathfrak{H},E)$ transfers to the density of $\mathcal{F}^p(\mathfrak{H},E)$ in $\gamma(\mathfrak{H},E)$, where $\mathcal{F}^p(\mathfrak{H},E)$ is obtained from $E$ analogously to $\gamma^p(\mathfrak{H},E)$.

Alternatively, one can view $\gamma(\mathfrak{H},E)$ as the completion of the algebraic tensor product $\mathfrak{H} \otimes_{\gamma} E$ with respect to the tensor norm $\gamma$ (on pairs of Hilbert and Banach spaces) inferred from the identification of $\mathfrak{H} \otimes E$ with $\mathcal{F}(\mathfrak{H},E)$. Then $\gamma(\mathfrak{H},E) \simeq \mathfrak{H} \widehat{\otimes}_{\gamma} E$ and, more generally,
\begin{equation*}
  \gamma^p(\mathfrak{H},E) \simeq \mathfrak{H} \, \widehat{\otimes}_{\gamma} \left( \dots \left( \mathfrak{H} \, \widehat{\otimes}_{\gamma} \left( \mathfrak{H} \, \widehat{\otimes}_{\gamma} E \right)   \right) \dots \right).
\end{equation*}
For simplicity, given $T \in \gamma^{p}(\mathfrak{H},E)$ and $h=(h_1,\dots,h_p) \in \mathfrak{H}^p$, we write $T(h) = T(h_1\dots,h_p)$ instead of the more cumbersome
\begin{equation*}
  \left(
    \dots
  \left(
    \left( T h_1
  \right) \, h_2 \right) \dots
\right) h_p.
\end{equation*}

Note that any $T \in \mathcal{F}^p(\mathfrak{H},E)$ can be written in the form
\begin{equation}
 \label{eq:33}
  T = \sum_{\mathbf{i} \in [n]^{p}} \mathfrak{h}_{\mathbf{i}} \otimes x_{\mathbf{i}},
\end{equation}
where $[n] = \left\{ 1,2,\dots,n \right\}$, $\left\{ \mathfrak{h}_k \colon k \in \mathbb{N} \right\}$ is an orthonormal basis of $\mathfrak{H}$ and, for a multi index $\mathbf{i}=(i_1,i_2,\dots,i_p) \in [n]^{p}$, $x_{\mathbf{i}} \in E$ and the symbol
\begin{equation*}
   \mathfrak{h}_{\mathbf{i}} \otimes x_{\mathbf{i}}
\end{equation*}
denotes the rank-one operator in $\mathcal{F}^p(\mathfrak{H},E)$ defined for $h=(h_1,\dots,h_p) \in \mathfrak{H}^{p}$ by
\begin{equation*}
  \left( \mathfrak{h}_{\mathbf{i}} \otimes x_{\mathbf{i}} \right) h
=
\prod_{j=1}^p \left\langle \mathfrak{h}_{i_j}, h_j \right\rangle_{\mathfrak{H}} x_{\mathbf{i}}.
\end{equation*}
The norm for such $T$ is then given by
\begin{equation}
 \label{eq:34}
 \left\lVert T \right\rVert_{\gamma^{p}(\mathfrak{H},E)}
 =
 \mathbb{E} \left(
   \left\lVert \sum_{\mathbf{i}  \in [n]^{p}}^{} W_{\mathbf{i}} \, x_{\mathbf{i}} \right\rVert_E^2
\right)^{\frac{1}{2}},
\end{equation}
where $W \colon \mathfrak{H} \to L^2(\Omega)$ is an isonormal Gaussian process and
\begin{equation*}
  W_{\mathbf{i}} = W^{(1)} \left( \mathfrak{h}_{i_1} \right) W^{(2)} \left( \mathfrak{h}_{i_2} \right) \cdots W^{(p)} \left( \mathfrak{h}_{i_p} \right),
\end{equation*}
with the $W^{(j)}$ being mutually independent copies of $W$, so that $W_{\mathbf{i}}$ is a product of independent standard Gaussians.

Note that in~\eqref{eq:34} one can in general not replace the products $W_{\mathbf{i}}$ by $\widetilde{W}(\mathfrak{h}_{\mathbf{i}})$, using a single isonormal Gaussian process $\widetilde{W}$ on $\mathfrak{H}^{\otimes p}$. As shown in~ \cite{kalton-weis:2014:h-functional-calculus-square} (see also~ \cite{neerven-weis:2008:stochastic-integration-operator-valued}), $\gamma^p(\mathfrak{H},E)$ is isomorphic to $\gamma(\mathfrak{H}^{\widehat{\otimes}_{\sigma} p},E)$ for all $p \in \mathbb{N}$ precisely when $E$ has Pisier's property $(\alpha)$, which was introduced in~ \cite{pisier:1978:results-banach-spaces}. This property is shared in particular by all Hilbert spaces, in which case $\gamma^p(\mathfrak{H},E)$ is isometrically isomorphic to the Hilbert tensor product $\mathfrak{H}^{\widehat{\otimes}_{\sigma} p} \widehat{\otimes}_{\sigma} E$ corresponding to multilinear Hilbert-Schmidt mappings.

The symmetrization operator $\operatorname{Sym} \colon \gamma^p(\mathfrak{H},E) \to \gamma^{p}(\mathfrak{H},E)$ is defined for $T \in \gamma^p(\mathfrak{H},E)$ and $h \in \mathfrak{H}^{p}$ by
\begin{equation*}
 \operatorname{Sym}(T) h = \frac{1}{p!}
 \sum_{\sigma \in S_{p}}^{} T h_{\sigma},
\end{equation*}
where $S_p$ is the permutation group on $[p]$ and $h_{\sigma} = \left( h_{\sigma(1)}, \dots,h_{\sigma(p)} \right)$.
It is clear that $\operatorname{Sym}$ is a projection on $\gamma^{p}(\mathfrak{H},E)$ and we denote its closed range by $\gamma^p_{sym}(\mathfrak{H},E)$, calling the operators in this subspace symmetric. Analogously, we define $\mathcal{F}^p_{sym} (\mathfrak{H},E) = \operatorname{Sym} \left( \mathcal{F}^p(\mathfrak{H},E) \right)$.

For more details on $\gamma$-radonifying operators, see~ \cite{neerven:2010:gamma-radonifying-operators-survey} and the references therein.

\subsection{Malliavin calculus in Banach spaces}
\label{sec:mall-calc-banach}

In~\cite{maas:2010:malliavin-calculus-decoupling} and \cite{maas-neerven:2008:clark-ocone-formula-umd} (see also~ \cite{pronk-veraar:2014:tools-malliavin-calculus}), $\gamma$-radonifying operators emerged as natural objects for developing Malliavin calculus for Banach-valued random elements. We will outline the construction of the Malliavin derivative, Ornstein-Uhlenbeck semigroup and its generator from these papers, skipping the divergence operator, which is not needed to prove our results. For the subsequent covariance identities we require the pseudo-inverse on Malliavin Sobolev spaces. We give a construction on this domain that does not impose geometric conditions on the Banach space.

We use standard notation for the induced operators of Malliavin calculus, i.e. $D$ (Malliavin derivative), $\delta$ (its adjoint), $(P_t)_{t \geq 0}$ (Ornstein-Uhlenbeck semigroup), $L$ (its generator) and $L^{-1}$ (the pseudoinverse of the generator) and refer to~\cite{nualart:2006:malliavin-calculus-related} or~\cite{nourdin-peccati:2012:normal-approximations-malliavin} for details.

For clarity, given a Banach space $E$ different from $\mathbb{R}$, we will indicate throughout this section with a subscript when the Malliavin operators act on $E$-valued random elements, writing $D_{E}$, $L_{E}$ etc. Once all operators have been extended, we will drop the subscript again, which, as explained in the last paragraph of this section, can be done without causing confusion.

Let $W \colon \mathfrak{H} \to L^2(\Omega)$ be an isonormal Gaussian process as above and $\mathcal{S}$ be the class of smooth random variables of the form $f(W(h_1),\dots, W(h_n))$, where $f \in C^{\infty}_b(\mathbb{R}^{n})$ is an infinitely differentiable function on $\mathbb{R}^{n}$ such that $f$ and all of its partial derivatives are bounded. For a Banach space $E$, we define $\mathcal{S}(E) = \mathcal{S} \otimes E$ and lift the Malliavin derivative to $\mathcal{S}(E)$ by defining $D_E = D \otimes \operatorname{Id}_{E}$. Explicitly, given $F \in \mathcal{S}(E)$ of the form $F = f(W(h_1),\dots, W(h_n)) \otimes x$, with $x \in E$, $D_EF$ is the random element
\begin{equation*}
  D_{E}F = \left( D f(W(h_1),\dots,W(h_n)) \right) \otimes x = \sum_{k=1}^n \partial_k f(W(h_1),\dots, W(h_n))  h_k \otimes x,
\end{equation*}
where $h_n \otimes x$ denotes the rank-one operator $T \colon \mathfrak{H} \to E$ given by $Th = \left\langle h_n,h \right\rangle_{\mathfrak{H}} x$. This procedure yields closable operators $D_E$ mapping from $\mathcal{S} \otimes_{\Delta_p} E \subseteq L^{p}(\Omega;E)$ to $\mathcal{S} \otimes_{\Delta_p } ( \mathfrak{H} \otimes_{\gamma} E) \subseteq L^p(\Omega,\gamma(\mathfrak{H},E))$ for any $p \in [1,\infty)$. Setting $D^1_{E}=D_E$ and, for convenience, $D_E^{0}=\operatorname{Id}_{E}$, we can therefore inductively define Malliavin derivatives $D_{E}^k \colon L^p(\Omega;E) \to L^p(\Omega,\gamma^k(\mathfrak{H},E))$ for any $k \in \mathbb{N}_0$ by $D^{k+1}_EF=D_E(D^k_EF)$. The domains of their closures will be denoted by $\mathbb{D}^{k,p}(E)$, together with the corresponding norm
\begin{equation*}
  \left\lVert F \right\rVert_{k,p} = \sum_{j=0}^k \left\lVert D^j_EF \right\rVert_{L^p(\Omega;\gamma^j(\mathfrak{H},E))},
\end{equation*}
and as usual we will continue to use the symbol $D^{k}_E$ for the closure. As mentioned in the previous section, if $E$ is a Hilbert space, the space $\gamma^{k}(\mathfrak{H},E)$, coincides with the $k$-multilinear $E$-valued Hilbert-Schmidt mappings. In particular, for $E=\mathbb{R}$ we have $\gamma^{k}(\mathfrak{H},\mathbb{R}) \simeq \mathfrak{H}^{\widehat{\otimes}_{\sigma} k}$ and obtain the classical Malliavin derivatives treated for example in \cite{nualart:2006:malliavin-calculus-related} or \cite{nourdin-peccati:2012:normal-approximations-malliavin}.

Pronk and Veraar showed in~\cite[Prop. 3.8]{pronk-veraar:2014:tools-malliavin-calculus} that the chain rule has an extension to the vector-valued setting, when the codomain of the smooth transformation is UMD (unconditional for martingale differences). A Banach space $E$ is called UMD, if for some (or equivalently all) $p \in (1,\infty)$ there exists a constant $C_{p}$ such that for all $E$-valued martingale difference sequences $(d_k)_{k \in \mathbb{N}}$ in $L^p(\Omega;E)$ (with respect to some filtered probability space), all binary sequences $(r_k)_{k \in \mathbb{N}} \subseteq \left\{ -1,1 \right\}^{\mathbb{N}}$ and all $n \in \mathbb{N}$ one has
\begin{equation*}
  \left\lVert \sum_{k=1}^n r_k d_k \right\rVert_p \leq C_p
  \left\lVert \sum_{k=1}^n d_k \right\rVert_p.
\end{equation*}
Examples of UMD-spaces are Hilbert spaces and the $L^p(\Omega,\mu)$-spaces for $p \in (1,\infty)$ and some $\sigma$-finite measure $\mu$ on a measurable space $(\Omega,\mathcal{F})$. The chain rule reads as follows (see~\cite[Prop. 3.8]{pronk-veraar:2014:tools-malliavin-calculus}).

\begin{theorem}
  \label{thm:4}
  Let $E_{1}$ be a Banach space, $E_2$ be a UMD Banach space, $p \in (1,\infty)$ and $\varphi \in \mathcal{C}_b^{1}(E_1,E_2)$. If $F \in \mathbb{D}^{1,p}(E_1)$, then $\varphi(F) \in \mathbb{D}^{1,p}(E_2)$ and
\begin{equation*}
  D_{E_2} \varphi(F) = \varphi'(F) D_{E_1}F.
\end{equation*}
In the case where $E_2=\mathbb{R}$, one may also take $p=1$.
\end{theorem}

As the positive and strongly continuous Ornstein-Uhlenbeck semigroup $(P_t)_{t \geq 0}$ is defined on $L^p(\Omega)$ for any $p \in [1,\infty)$, the tensor operators
\begin{equation*}
  P_{E,t} = P_t \otimes \operatorname{Id}_E
\end{equation*}
extend to a strongly continuous semigroup $(P_{E,t})_{t \geq 0}$ on the Lebesgue-Bochner spaces $L^p(\Omega;E)$ for any Banach space $E$ and $p \in [1,\infty)$ via Theorem~\ref{thm:1}. Let $L_{E}$ be its infinitesimal generator with $L^p(\Omega;E)$-domain $\operatorname{dom}_p(L_{E})$.

\begingroup

We shall use the classical Banach-valued Gaussian Poincar\'{e}
inequality in the following form. For every Banach space $E$ and
$p\in\left(1,\infty\right)$, there exists a constant $C_p>0$,
depending only on $p$, such that every
$F\in\mathbb{D}^{1,p}\left(E\right)$ satisfies
\begin{equation}
\label{eq:vector-poincare}
\norm{F-\E\left(F\right)}_{L^p\left(\Omega;E\right)}
\leq C_p\norm{D_EF}_{L^p\left(\Omega;\gamma\left(\mathfrak{H},E\right)\right)}.
\end{equation}
This inequality follows from the results of Maurey and Pisier
(see~\cite[Corollary~2.4]{pisier:1986:probabilistic-methods-banach})
and their $L^p$ extension
(see~\cite[Section~2.1]{cordero-eskenazis:2024:discrete-log-sobolev}).
The formulation in terms of $\gamma$-radonifying operators follows
from the Kahane--Khintchine inequalities, and the extension to
$\mathbb{D}^{1,p}\left(E\right)$ follows by density.

For $p\in\left(1,\infty\right)$, the positive Ornstein--Uhlenbeck semigroup extends to a strongly continuous contraction semigroup on both $L^p\left(\Omega;E\right)$ and $L^p\left(\Omega;\gamma\left(\mathfrak{H},E\right)\right)$. The derivative commutation identity, initially for smooth cylindrical functionals and then by closure, is
\begin{equation}
\label{eq:semigroup-derivative-commutation}
D_EP_{E,t}F=e^{-t}P_{\gamma\left(\mathfrak{H},E\right),t}D_EF,
\qquad F\in\mathbb{D}^{1,p}\left(E\right).
\end{equation}
Applying Equation~\eqref{eq:vector-poincare} to $P_{E,t}F$, and using that the semigroup preserves expectations, we obtain the exponential decay estimate
\begin{equation}
\label{eq:sobolev-semigroup-decay}
\norm{P_{E,t}F-\E\left(F\right)}_{L^p\left(\Omega;E\right)}
\leq C_pe^{-t}\norm{D_EF}_{L^p\left(\Omega;\gamma\left(\mathfrak{H},E\right)\right)}.
\end{equation}
While Equation~\eqref{eq:sobolev-semigroup-decay} holds for
arbitrary Banach spaces, a corresponding exponential decay
estimate in the $L^p\left(\Omega;E\right)$-norm for all centered
random variables is false in general. Indeed, suppose that the
underlying Gaussian space is infinite-dimensional and take
$E=c_0$. For
$F_N=\left(\operatorname{sgn}(G_1),\ldots,
\operatorname{sgn}(G_N),0,\ldots\right)$, where the $G_k$ are
independent standard Gaussian random variables, we have
$\E\left(F_N\right)=0$ and
$\norm{F_N}_{L^p\left(\Omega;c_0\right)}=1$, whereas Mehler's
formula gives
\begin{equation*}
\lim_{N\to\infty}
\norm{P_{c_0,t}F_N}_{L^p\left(\Omega;c_0\right)}=1
\end{equation*}
for every $t>0$. Thus, uniform exponential decay on the whole
centered subspace fails, while Equation~\eqref{eq:sobolev-semigroup-decay}
avoids this obstruction by controlling the decay through the
Malliavin derivative.

We therefore define the pseudo-inverse on $\mathbb{D}^{1,p}\left(E\right)$ by
\begin{equation}
\label{eq:banach-pseudo-inverse}
L_E^{-1}F=-\int_0^\infty P_{E,t}\left(F-\E\left(F\right)\right)dt.
\end{equation}
By Equations~\eqref{eq:sobolev-semigroup-decay} and~\eqref{eq:semigroup-derivative-commutation}, the integral converges in $\mathbb{D}^{1,p}\left(E\right)$, and
\begin{equation}
\label{eq:pseudo-inverse-gradient-representation}
D_EL_E^{-1}F=-\int_0^\infty e^{-t}P_{\gamma\left(\mathfrak{H},E\right),t}D_EFdt.
\end{equation}
In particular,
\begin{align}
\label{eq:pseudo-inverse-sobolev-bounds}
\norm{L_E^{-1}F}_{L^p\left(\Omega;E\right)}
\leq C_p\norm{D_EF}_{L^p\left(\Omega;\gamma\left(\mathfrak{H},E\right)\right)}
\end{align}
and
\begin{align*}
\norm{D_EL_E^{-1}F}_{L^p\left(\Omega;\gamma\left(\mathfrak{H},E\right)\right)}
\leq\norm{D_EF}_{L^p\left(\Omega;\gamma\left(\mathfrak{H},E\right)\right)}.
\end{align*}

\begin{proposition}
\label{prop:1}
Let $E$ be a Banach space and $p\in\left(1,\infty\right)$. For every $F\in\mathbb{D}^{1,p}\left(E\right)$, the element $L_E^{-1}F$ defined by Equation~\eqref{eq:banach-pseudo-inverse} belongs to $\operatorname{dom}_p\left(L_E\right)\cap\mathbb{D}^{1,p}\left(E\right)$ and satisfies
\begin{equation*}
L_EL_E^{-1}F=F-\E\left(F\right).
\end{equation*}
Moreover, $L_E^{-1}$ is a bounded linear operator from $\mathbb{D}^{1,p}\left(E\right)$ into itself and vanishes on constant random elements.
\end{proposition}

\begin{proof}
For $T>0$, put
\begin{equation*}
G_T=-\int_0^T P_{E,t}\left(F-\E\left(F\right)\right)dt.
\end{equation*}
The standard generator identity for a strongly continuous semigroup (see~\cite[Ch. II, Lemma 1.3]{engel-nagel:2006:short-course-operator}) gives $G_T\in\operatorname{dom}_p\left(L_E\right)$ and
\begin{equation*}
L_EG_T=F-P_{E,T}F.
\end{equation*}
By Equation~\eqref{eq:sobolev-semigroup-decay}, we have $P_{E,T}F\to\E\left(F\right)$ in $L^p\left(\Omega;E\right)$ as $T\to\infty$, while the defining integral yields $G_T\to L_E^{-1}F$ in $\mathbb{D}^{1,p}\left(E\right)$. Since $L_E$ is closed, we obtain the required domain inclusion and identity. The boundedness and the assertion concerning constants follow from \eqref{eq:pseudo-inverse-sobolev-bounds}.
\end{proof}

\begin{lemma}
\label{lem:2}
Let $E$ be a Banach space and $p\in\left(1,\infty\right)$. For every $F\in\mathbb{D}^{1,p}\left(E\right)$,
\begin{equation*}
\norm{D_EL_E^{-1}F}_{L^p\left(\Omega;\gamma\left(\mathfrak{H},E\right)\right)}
\leq
\norm{D_EF}_{L^p\left(\Omega;\gamma\left(\mathfrak{H},E\right)\right)}.
\end{equation*}
\end{lemma}

\begin{proof}
This follows from Equation~\eqref{eq:pseudo-inverse-gradient-representation}, the contraction property of $P_{\gamma\left(\mathfrak{H},E\right),t}$, and the identity $\int_0^\infty e^{-t}\mathrm{d}t=1$.
\end{proof}

\begin{remark}
  The construction of
$L_E^{-1}$ on $\mathbb{D}^{1,p}\left(E\right)$ is sufficient for all
the integration-by-parts, covariance and Gaussian approximation
results below, where the arguments of the pseudo-inverse are Malliavin
differentiable. We do not assert an extension of $L_E^{-1}$ as a
bounded operator on the whole of $L^p\left(\Omega;E\right)$ for
arbitrary $E$. The scalar case (which is a spectral-gap estimate) does not in general tensorize in this manner.
\end{remark}
\endgroup

\begingroup
The vector-valued operators are compatible with their scalar counterparts. In particular, for $\ell\in E^{\ast}$ and $F\in\mathbb{D}^{1,p}\left(E\right)$, application of $\ell$ under the Bochner integral in \eqref{eq:banach-pseudo-inverse} shows that
\begin{equation*}
\ell\left(L_E^{-1}F\right)=L_{\R}^{-1}\left(\ell\left(F\right)\right).
\end{equation*}
The same compatibility holds for the Malliavin derivative, the Ornstein--Uhlenbeck semigroup and its generator on their respective domains. Indeed, these identities hold for smooth cylindrical random elements and extend by the corresponding closure and continuity properties. For this reason we usually omit the subscripts on the operators in the sequel. In particular, for an $E$-valued random element in a fixed Wiener chaos of order $q\geq1$, one has
\begin{equation*}
L^{-1}F=-\frac1qF,
\end{equation*}
by the action of the semigroup on that chaos.
\endgroup

As another convention, whenever we use a Malliavin operator or consider a random element in its domain, we tacitly assume the existence of an underlying isonormal Gaussian process $W$ on a real, separable Hilbert space $\mathfrak{H}$, with respect to which the operator is defined.

\subsection{Banach-valued Wiener--It\^{o} chaos}
\label{sec:wiener-ito-chaos}

Let $W \colon \mathfrak{H} \to L^2(\Omega,\mathcal{F},P)$ be an isonormal Gaussian process and $\mathcal{H}_p(\mathbb{R})$ the $p$-th Wiener chaos ($p \in \mathbb{N}_0$, with $\mathcal{H}_0(\mathbb{R}) = \mathbb{R}$). Throughout this section, we assume that $\mathcal{F}$ is generated by $W$, so that the orthogonal Wiener chaos decomposition
\begin{equation*}
  L^2(\Omega,\mathcal{F},P) = \bigoplus_{p=0}^{\infty} \mathcal{H}_p(\mathbb{R})
\end{equation*}
holds.

For a Banach space $E$, the $E$-valued $p$-th Wiener chaos $\mathcal{H}_p(E)$ is defined as the $L^2(\Omega;E)$-completion of $\mathcal{H}_p(\mathbb{R}) \otimes E$, the $E$-valued random elements which are linear combinations of $E$-vectors using coefficients in $\mathcal{H}_p(\mathbb{R})$.
It follows from the Kahane-Khintchine inequality that, as is the case for $\mathcal{H}_p(\mathbb{R})$,  all $L^s(\Omega;E)$-norms ($s \in [1,\infty)$) are equivalent on $\mathcal{H}_p(E)$.

As $\mathcal{F}_{sym}^{p}(\mathfrak{H},E) \simeq \mathfrak{H}^{\odot p} \otimes E$, we can define the tensor operator
\begin{equation*}
  I_{p}^{\mathbb{R}} \otimes \operatorname{Id}_E \colon \mathcal{F}_{sym}^p(\mathfrak{H},E) \to \mathcal{H}_p(\mathbb{R}) \otimes E,
\end{equation*}
where $I_p^{\mathbb{R}}  \colon \mathfrak{H}^{\odot p} \to \mathcal{H}_p(\mathbb{R})$ is the Wiener-It\^{o} integral acting on real valued kernels.

It is one of the main findings of~\cite{maas:2010:malliavin-calculus-decoupling} (see Thm. 3.2 therein), that for any $p \in \mathbb{N}_{0}$ and $s \in \left[1,\infty\right)$, the above tensor operator $I_p \otimes \operatorname{Id}_E$ extends to a bounded linear operator
\begin{equation*}
  I_{p}^{E} \colon \gamma^p_{sym}(\mathfrak{H},E) \to L^s(\Omega;E),
\end{equation*}
called the $E$-valued multiple Wiener-It\^{o} integral of order $p$, whose range is $\mathcal{H}_p(E)$. The isometry property of $I_p^{\mathbb{R}}$ is only retained in the form of an equivalence of norms: For any $s \in [1,\infty)$ and any $p \in \mathbb{N}_0$ there exist positive constants $\mathfrak{c}_{p,s}$ and $\mathfrak{C}_{p,s}$ such that for all $T \in \gamma^p_{sym}(\mathfrak{H},E)$ one has
\begin{equation}
  \label{eq:35}
  \mathfrak{c}_{p,s}
  \left\lVert T \right\rVert_{\gamma^p(\mathfrak{H},E)}
  \leq
  \left\lVert I_{p}^E(T) \right\rVert_{L^s(\Omega;E)}
  \leq
  \mathfrak{C}_{p,s}
  \left\lVert T \right\rVert_{\gamma^p(\mathfrak{H},E)},
\end{equation}
Inspecting the proof of~\cite[Theorem 3.2]{maas:2010:malliavin-calculus-decoupling} and in view of~ \cite[Theorem 2.2]{arcones-gine:1993:decoupling-series-expansions}, one can take
\begin{equation}
 \label{eq:36}
  \mathfrak{c}_{p,s} = 2^{\frac{p-1}{s}} \sqrt{\frac{p!}{p^{p}}} \kappa_{s,2}^{-1} \qquad \text{and} \qquad \mathfrak{C}_{p,s} =  4^{\frac{3p+1}{s}} \sqrt{\frac{p^p}{p!}} \kappa_{s,2},
\end{equation}
where $\kappa_{s,t}$ denotes the Kahane-Khintchine constant for the $L^s-L^{t}$-norm comparison of Gaussian (or Rademacher) sums (see for example~\cite[Theorem 6.2.6]{hytonen-neerven-veraar-ea:2017:analysis-banach-spaces} for a precise definition). In the special case where $E$ is a Hilbert space and $s=2$,~\eqref{eq:35} of course reduces to the classical Wiener-It\^{o} isometry, i.e. the inequalities~\eqref{eq:35} become equalities, which hold for $\mathfrak{c}_{p,2} = \mathfrak{C}_{p,2} = \sqrt{p!}$.

Observe that, as for the other Malliavin operators (see the last paragraph in the previous section), the multiple integral commutes with linear forms in the following sense: if $f \in \gamma^p(\mathfrak{H},E)$ and $\ell \in E^{\ast}$, then $\ell \circ f  \in \gamma^{p}(\mathfrak{H},\mathbb{R}) \simeq \mathfrak{H}^{\widehat{\odot}_{\sigma} p}$ and
\begin{equation}
  \label{eq:25}
  I_{p}^{\mathbb{R}}(\ell \circ f) = \ell \left( I_p^{E}(f) \right).
\end{equation}
Because of this compatibility result, we will most of the time refrain from adding the codomain of the kernel as a superscript to the operator $I_{p}$ in later sections. Taking the $L^2(\Omega)$-norm in~\eqref{eq:25} and applying the Wiener-It\^{o} isometry for real-valued kernels gives
\begin{equation*}
  \left\lVert \ell \left( I_p^{E}(f) \right) \right\rVert_{L^2(\Omega)}
  =
  \sqrt{p!}
  \left\lVert \ell \circ f \right\rVert_{\mathfrak{H}^{\widehat{\otimes}_{\sigma} p}}.
\end{equation*}

One can verify directly that kernels $f \in \mathfrak{H}^{\widehat{\odot}_{\sigma} p} \widehat{\otimes}_{\pi} E$ of the form
\begin{equation*}
  f = \sum_{j=1}^{\infty} f_j \otimes x_j
\end{equation*}
with $f_j \in \mathfrak{H}^{\widehat{\odot}_{\sigma} p}$ and $x_j \in E$ for $j \in \mathbb{N}$ and such that
\begin{equation*}
  \sum_{j=1}^{\infty} \left\lVert f_j \right\rVert_{\mathfrak{H}^{\widehat{\odot}_{\sigma} p}} \left\lVert x_j \right\rVert_E < \infty
\end{equation*}
are in $\gamma^p_{sym}(\mathfrak{H},E)$ (it suffices to approximate in $\mathcal{H}^{\odot p} \otimes E$ and crudely bound the $\gamma^p(\mathfrak{H},E)$-norm by shifting the $L^2$-norm to the summands).

By the Wiener-It\^{o} isometry, the corresponding integrals are nuclear series of the form
\begin{equation*}
   I_p^{E}(f) = \sum_{j=1}^{\infty} I_p^{\mathbb{R}}(f_j) \otimes x_j
\end{equation*}
with
\begin{equation*}
  \sum_{j=1}^{\infty} \left\lVert f_j \right\rVert_{\mathfrak{H}^{\widehat{\odot}_{\sigma} p}} \left\lVert x_j \right\rVert_E < \infty.
\end{equation*}
Such integrals, in particular finite-rank versions where all sums are finite, will play a role in the applications treated in later sections.

When generalizing the multiple integral for $\mathfrak{H} = L^{2}(M) = L^2(M,\mathcal{M},\mu)$, with $\mu$ a $\sigma$-finite measure without atoms on a measurable space $(M,\mathcal{M})$, some care has to be applied as one can no longer use the isometric identification
\begin{equation*}
  L^2(M,\mathcal{M},\mu)^{\widehat{\otimes}_{\sigma} p} \simeq L^{2}(M^{p}, \mathcal{M}^{\otimes p}, \mu^{\otimes p}).
\end{equation*}
Indeed, as mentioned earlier, $\gamma^p(\mathfrak{H},E)$ is in general not isomorphic to $\gamma(\mathfrak{H}^{ \widehat{\otimes}_{\sigma} p}, E)$. However, as noted in~\cite[Section 2.2]{maas:2010:malliavin-calculus-decoupling}, the space $\gamma^p(\mathfrak{H},E)$ contains a dense subset of representable operators $f$ for which there exists a Bochner measurable function $\varphi_{f} \colon M^p \to E$ which is weakly $L^2$, i.e. $(\varphi,x^{\ast})_{E,E^{\ast}} \in L^2(\mu^{\otimes p})$ for all $x^{\ast} \in E^{\ast}$, such that
\begin{equation*}
  (f(h_{1},h_2,\dots,h_p),x^{\ast})_{E,E^{\ast}} = \int_{M^p}^{} h_1(t_1) \dots h_p(t_p) \left( \phi_f(t_1,\dots,t_p), x^{\ast} \right)_{E,E^{\ast}} d \mu^{\otimes p} \left( t_1,\dots,t_p \right).
\end{equation*}
For such representable operators, the multiple integral can be lifted to the vector-valued case (see~\cite[Sec. 4]{maas:2010:malliavin-calculus-decoupling}).

Let us point out that the tensor operators $J_p \otimes \operatorname{Id}_E$, where $J_p$ is the bounded projection onto the $p$-th Wiener chaos $\mathcal{H}_p(\mathbb{R})$, do in general not extend to bounded projections onto $\mathcal{H}_p(E)$. The needed property here is $K$-convexity of $E$, possessed for example by UMD spaces and in particular Hilbert spaces. See~\cite[Rem. 3.4]{maas:2010:malliavin-calculus-decoupling} for references and details.

The projection $J_0$ onto the constants $\mathcal{H}_0(E)$ however is always bounded, as it coincides with the expectation operator (see the construction of $L^{-1}$ in the previous section).

\section{The carré du champ operator and integration by parts formul{\ae}}
\label{sec:carre-du-champ}

In this section, we construct a tensor-valued carré du champ operator and then use it to prove integration by parts formul{\ae}, which can be interpreted as chain rules for (cross)-covariance tensors and operators. Applying the smart path method, these chain rules yield a generic carré du champ-bound for differences of the form
\begin{equation*}
  \left| \mathbb{E} \left( f(X) \right) - \mathbb{E} \left( f(Y) \right) \right|,
\end{equation*}
with sufficiently regular Banach-valued random elements $X$ and $Y$ and test functions $f$, which will serve as a prototype to bound integral probability metrics in the later sections.

We begin by introducing a tensor contraction, also known as vector valued trace (see~ \cite[Ch. 29]{defant-floret:1993:tensor-norms-operator} for a more general definition and the associated calculus).

Let $\mathfrak{H}$ be a pre-Hilbert space and $E$, $F$ be normed spaces. The contraction $C_{\mathfrak{H}}\colon (\mathfrak{H} \otimes E) \times (\mathfrak{H} \otimes F) \to E \otimes F$ is defined for elementary tensors as
\begin{align*}
  C_{\mathfrak{H}} (h \otimes x, h' \otimes y) = \left\langle h,h' \right\rangle_{\mathfrak{H}} x \otimes y
\end{align*}
and then extended by bilinearity. If the pre-Hilbert space is of the form $\mathfrak{H}^{\otimes r}$, we write
$C_\mathfrak{H}^{r} = C_{\mathfrak{H}^{\otimes r}}$ (so that $C_{\mathfrak{H}}^{1} = C_{\mathfrak{H}}$), and in addition define
\begin{align*}
  C_{\mathfrak{H}}^{0} \colon (\mathfrak{H} \otimes E) \times (\mathfrak{H} \otimes F) &\to \mathfrak{H} \otimes \mathfrak{H} \otimes E \otimes F \\
  (h \otimes x, h' \otimes y) &\mapsto h \otimes h' \otimes x \otimes y.
\end{align*}

Canonically identifying the algebraic tensor product $\mathfrak{H}^{\otimes p} \otimes E$ with the finite-rank operators $\mathcal{F}^{p}(\mathfrak{H};E)$, the contraction is also defined as an operator
\begin{align*}
  C_{\mathfrak{H}}^r \colon \mathcal{F}^{p}(\mathfrak{H};E) \times \mathcal{F}^{q}(\mathfrak{H};F) \to \mathcal{F}^{p+q-2r}(\mathfrak{H};E \otimes F)
\end{align*}
for any $p,q \in \mathbb{N}$ and $0 \leq r \leq p \land q$, and of course we could further identify the algebraic tensor product $E \otimes F$ with the finite rank operators $\mathcal{F}(F^{\ast},E)$. For a normed space $S$, we denote by $\mathcal{F}_{\gamma}(\mathfrak{H},S)$ the space $\mathcal{F}(\mathfrak{H},S)$ equipped with the gamma-radonifying norm $\gamma$ and then inductively define $\mathcal{F}^p_{\gamma}(\mathfrak{H},S)$. The following lemma shows that the contraction has a continuous extension to spaces of $\gamma$-radonifying operators (see Section~\ref{sec:gamma-radon-oper} for the definition).

\begin{lemma}
  \label{lem:1}
  For $p,q \in \mathbb{N}$, a real separable Hilbert space $\mathfrak{H}$ and Banach spaces $E$, $F$, let $f \in \mathcal{F}^p(\mathfrak{H},E)$ and $g \in \mathcal{F}^{q}(\mathfrak{H},F)$. Then, for any $0 \leq r \leq p \land q$, one has
  \begin{equation}
    \label{eq:26}
\left\lVert C_{\mathfrak{H}}^{r} (f,g) \right\rVert_{\mathcal{F}^{p+q-2r}_{\gamma}(\mathfrak{H},E \otimes_{\pi} F)} \leq
  \left\lVert f \right\rVert_{\mathcal{F}^p_{\gamma}(\mathfrak{H},E)} \left\lVert g \right\rVert_{\mathcal{F}^{q}_{\gamma}(\mathfrak{H},F)}.
\end{equation}
Consequently, the contraction operator has a bounded tensor-valued extension
  \begin{equation*}
  \widehat{C}_{\mathfrak{H},\pi}^r \colon \gamma^p(\mathfrak{H},E) \times \gamma^q(\mathfrak{H},F) \to \gamma^{p+q-2r}(\mathfrak{H},E \widehat{\otimes}_{\pi} F)
\end{equation*}
and also a bounded operator-valued extension
\begin{equation*}
  \widehat{C}_{\mathfrak{H}}^r \colon \gamma^p(\mathfrak{H},E) \times \gamma^q(\mathfrak{H},F) \to \gamma^{p+q-2r}(\mathfrak{H},\mathfrak{N}(F^{\ast},E))
\end{equation*}
\end{lemma}

\begin{remark}
  \label{rmk:5}
  \hfill
  \begin{enumerate}[(i)]
  \item Note that in the case $r=1$, we have $\widehat{C}_{\mathfrak{H}}(f,g) = f g^{\ast}$, where $g^{\ast}$ denotes the adjoint of $g$.
\item Lemma~\ref{lem:1} also yields bounded versions of $C_{\mathfrak{H}}^{r}$ mapping into $\gamma^{p+q-2r}(\mathfrak{H},E \widehat{\otimes}_{\alpha} F)$ for any tensor norm $\alpha$, as $\alpha \leq \pi$. A similar statement can be made for the associated operator ideals.
  \end{enumerate}

\end{remark}

\begin{proof}[Proof of Lemma~\ref{lem:1}]
  Once~\eqref{eq:26} is established, the existence of the bounded tensor-valued extension follows from the fact that for any normed space $S$, the finite-rank operators $\mathcal{F}^p(\mathfrak{H},S)$ are dense in $\gamma^p(\mathfrak{H},\widehat{S})$, where $\widehat{S}$ denotes the completion of $S$. To obtain the operator-valued extension, it suffices to note that for a tensor $u  \in E \otimes F$ and the associated finite rank operator $T_u \in \mathcal{F}(F^{\ast},E)$, one has
\begin{equation*}
  \left\lVert T_u \right\rVert_{\nu} \leq \left\lVert u \right\rVert_{\pi}.
\end{equation*}
Consequently, for $T \in \mathcal{F}^p_{\gamma}(\mathfrak{H},E \otimes_{\pi} F)$, denote by $\widetilde{T} \in \mathcal{F}^p_{\gamma}(\mathfrak{H}, \mathcal{F}_{\nu}(F^{\ast},E))$ its image under the canonical map. Monotonicity of the expectation yields
\begin{equation*}
  \left\lVert \widetilde{T} \right\rVert_{\mathcal{F}^p_{\gamma}(\mathfrak{H},\mathcal{F}_{\nu}(F^{\ast},E))}
  \leq
  \left\lVert T \right\rVert_{\mathcal{F}^p_{\gamma}(\mathfrak{H},E \otimes_{\pi}F)},
\end{equation*}
where $\mathcal{F}_{\nu}(F^{\ast},E)$ denotes the finite rank operators equipped with the nuclear norm, which are dense in $\mathfrak{N}(F^{\ast},E)$. Only the bounded map into this nuclear-operator space is used. Surjectivity is not needed. Together with~\eqref{eq:26}, we get
\begin{equation*}
  \left\lVert C_{\mathfrak{H}}^{r} (f,g) \right\rVert_{\gamma^{p+q-2r}(\mathfrak{H},\mathcal{F}^{(\nu)}(F^{\ast},E))} \leq
  \left\lVert f \right\rVert_{\mathcal{F}^p_{\gamma}(\mathfrak{H},E)} \left\lVert g \right\rVert_{\mathcal{F}^q_{\gamma}(\mathfrak{H},F)}
\end{equation*}
for $f \in \mathcal{F}^p_{\gamma}(\mathfrak{H},E)$ and $g \in \mathcal{F}^{q}_{\gamma}(\mathfrak{H},F)$ and can then pass to the completion.

To prove~\eqref{eq:26}, we will reuse the multi index notation introduced in Section~\ref{sec:gamma-radon-oper}.

Fix an orthonormal basis $\left\{ \mathfrak{h}_k \colon k \in \mathbb{N} \right\}$ of $\mathfrak{H}$ and let $f \in \mathcal{F}^{p}(\mathfrak{H},E)$, $g \in \mathcal{F}^q(\mathfrak{H},F)$ with representations analogous to~\eqref{eq:33} in Section~\ref{sec:gamma-radon-oper} of the form
\begin{equation*}
  f = \sum_{\mathbf{i} \in [n]^{p}}^{} \mathfrak{h}_{\mathbf{i}} \otimes x_{\mathbf{i}}
    \qquad \text{and} \qquad
 g = \sum_{\mathbf{j} \in [n]^{q}}^{} \mathfrak{h}_{\mathbf{j}} \otimes y_{\mathbf{j}}
\end{equation*}
with $x_{\mathbf{i}} \in E$ and $y_{\mathbf{j}} \in F$ for all $\mathbf{i} \in [n]^p$ and $\mathbf{j} \in [n]^{q}$, where $[n]$ denotes the set of the first $n$ integers and $[n]^{p}$ the corresponding $p$-tuples thereof. The notation $\mathfrak{h}_{\mathbf{j}} \otimes y_{\mathbf{j}}$ for a rank-one operator retains the same meaning as in~\eqref{eq:33}, so does $W_{\mathbf{i}}$ for the product of $m$ independent copies of the isonormal Gaussian process $W$ evaluated in $\mathfrak{h}_j$. Let $W'$, $W''$ be independent copies of $W$ as well and denote by $\mathbb{E}'$ and $\mathbb{E}''$ expectation with respect to $W'$ and $W''$, respectively, including all its respective independent copies if we use the multi-index notation $W'_{\mathbf{i}}$ or $W''_{\mathbf{i}}$), whereas the expectation $\mathbb{E}$ is always taken with respect to all random elements inside its argument. Finally, for two multi-indices $\mathbf{k} \in \mathbb{N}^{m}$ and $\mathbf{l} \in \mathbb{N}^{n}$, we denote their concatenation $(k_1,\dots,k_{m},l_1,\dots,l_n) \in \mathbb{N}^{m+n}$ by $(\mathbf{k},\mathbf{l})$.

Fix $1 \leq r \leq p \land q$. By definition of the contraction operator and the $\gamma$-radonifying norm, we obtain
\begin{align*}
   &\left\lVert C_{\mathfrak{H}}^{r}(f,g) \right\rVert^{2}_{\gamma^{p+q-2r}(\mathfrak{H}, E \otimes_{\pi} F)}
  \\ &\qquad= \notag
  \mathbb{E} \left(
    \left\lVert
      \sum_{\substack{\mathbf{u} \in [n]^{r}\\ \mathbf{s} \in [n]^{p-r}, \mathbf{t} \in [n]^{q-r}}}^{}
      \left( W_{\mathbf{s}}  \, x_{(\mathbf{u},\mathbf{s})} \right) \otimes \left(   W'_{\mathbf{t}} \, y_{(\mathbf{u},\mathbf{t})} \right)
    \right\rVert_{\pi}^2
  \right)
  \\ &\qquad= \notag
  \mathbb{E} \left(
    \left\lVert
      \mathbb{E}'' \left(
      \sum_{\substack{\mathbf{u}, \mathbf{v} \in [n]^{r}\\ \mathbf{s} \in [n]^{p-r}, \mathbf{t} \in [n]^{q-r}}}^{}
      \left( W''_{\mathbf{u}} W_{\mathbf{s}}  \, x_{(\mathbf{u},\mathbf{s})} \right) \otimes \left(   W''_{\mathbf{v}} W'_{\mathbf{t}} \, y_{(\mathbf{v},\mathbf{t})} \right)
       \right)
    \right\rVert_{\pi}^2
  \right)
  \\ &\qquad= \notag
       \mathbb{E} \left(
    \left\lVert
      \mathbb{E}'' \left(
       \left(
       \sum_{\mathbf{u}  \in [n]^r, \mathbf{s} \in [n]^{p-r}}^{}
  W''_{\mathbf{u}} W_{\mathbf{s}}  \, x_{(\mathbf{u},\mathbf{s})} \right) \otimes \left(
  \sum_{\mathbf{v} \in [n]^{r}, \mathbf{t} \in [n]^{q-r}}^{}
  W''_{\mathbf{v}} W'_{\mathbf{t}} \, y_{(\mathbf{v},\mathbf{t})} \right)
       \right)
    \right\rVert_{\pi}^2
        \right),
\end{align*}
where we have used Gaussian orthogonality. By monotonicity of the Bochner integral and the cross norm property of the projective norm ($\left\lVert x \otimes y \right\rVert_{\pi} = \left\lVert x \right\rVert_E \cdot \left\lVert y \right\rVert_{F}$), we can continue to write
\begin{align*}
  &\qquad\leq \notag
  \mathbb{E} \left(
  \mathbb{E}'' \left(
  \left\lVert
       \left(
       \sum_{\mathbf{u}  \in [n]^r, \mathbf{s} \in [n]^{p-r}}^{}
  W''_{\mathbf{u}} W_{\mathbf{s}}  \, x_{(\mathbf{u},\mathbf{s})} \right) \otimes \left(
  \sum_{\mathbf{v} \in [n]^{r}, \mathbf{t} \in [n]^{q-r}}^{}
  W''_{\mathbf{v}} W'_{\mathbf{t}} \, y_{(\mathbf{v},\mathbf{t})} \right)
  \right\rVert_{\pi}
       \right)^2
       \right)
  \\ &\qquad=
\mathbb{E} \left(
  \mathbb{E}'' \left(
  \left\lVert
       \sum_{\mathbf{u}  \in [n]^r, \mathbf{s} \in [n]^{p-r}}^{}
       W''_{\mathbf{u}} W_{\mathbf{s}}  \, x_{(\mathbf{u},\mathbf{s})}
       \right\rVert_{E}
       \,
       \left\lVert
       \sum_{\mathbf{v} \in [n]^{r}, \mathbf{t} \in [n]^{q-r}}^{}
  W''_{\mathbf{v}} W'_{\mathbf{t}} \, y_{(\mathbf{v},\mathbf{t})}
  \right\rVert_{F}
       \right)^2
       \right),
\end{align*}
and it remains to apply Cauchy-Schwarz to bound the above by $\left\lVert f \right\rVert_{\gamma^{p}(\mathfrak{H},E)}^{2} \, \left\lVert g \right\rVert_{\gamma^q(\mathfrak{H},F)}^2$.

The proof for the case $r=0$ follows in a similar but simpler way.
\end{proof}

We can now introduce the carré du champ operator.

\begin{definition}
  \label{def:2}
  Let $E$ and $F$ be Banach spaces. The tensor valued carré du champ operator
  $\Gamma_{\pi} \colon \mathbb{D}^{1,2}(E) \times \mathbb{D}^{1,2}(F) \to L^1(\Omega;E \widehat{\otimes}_{\pi} F)$ is
  defined by
  \begin{equation*}
  \Gamma_{\pi}(X,Y) = \widehat{C}_{\mathfrak{H},\pi}(DX,DY),
\end{equation*}
and the operator valued carré du champ operator $\Gamma \colon \mathbb{D}^{1,2}(E) \times \mathbb{D}^{1,2}(F) \to L^1\left(\Omega;\mathfrak{N}\left(F^{\ast},E\right)\right)$ as
\begin{equation*}
  \Gamma(X,Y) = \widehat{C}_{\mathfrak{H}}(DX,DY) = DX (DY)^{\ast},
\end{equation*}
where $\widehat{C}_{\mathfrak{H},\pi}$ and $\widehat{C}_{\mathfrak{H}}$ are the bounded contraction operators defined in Lemma~\ref{lem:1}.
\end{definition}

Lemma~\ref{lem:1} immediately yields the boundedness of both carré du champ operators, as
\begin{multline*}
  \mathbb{E} \left( \left\lVert \Gamma(X,Y) \right\rVert_{\nu} \right)
  \leq
  \mathbb{E} \left( \left\lVert \Gamma_{\pi}(X,Y) \right\rVert_{\pi} \right)
  =
  \mathbb{E} \left(
    \left\lVert \widehat{C}_{\mathfrak{H},\pi}(DX,DY) \right\rVert_{\pi}
  \right)
  \\ \leq
  \mathbb{E} \left(
\left\lVert DX \right\rVert_{\gamma(\mathfrak{H},E)} \left\lVert DY \right\rVert_{\gamma(\mathfrak{H},F)}
\right)
\leq
\left\lVert DX \right\rVert_{1,2} \left\lVert DY \right\rVert_{1,2},
\end{multline*}
where $\left\lVert \cdot \right\rVert_{1,2}$ denotes the norm of $\mathbb{D}^{1,2}(E)$.
Arguing analogously to Remark~\ref{rmk:5}, the tensor valued carré du champ operator is also bounded as a random bilinear mapping taking values in $L^1(\Omega; E \widehat{\otimes}_{\alpha} F)$, where $\alpha$ is any tensor norm.

The chain rule for the Malliavin derivative (Theorem~\ref{thm:4}) immediately yields the following chain rule for the carré du champ.

\begin{lemma}
  \label{lem:5}
  Let $E_1$, $E_2$, $F_1$, $F_2$ be Banach spaces, $X \in \mathbb{D}^{1,2}(E_1)$, $Y \in \mathbb{D}^{1,2}(F_1)$, $f \in \mathcal{C}^1_b(E_1,E_2)$ and $g \in \mathcal{C}^1_b(F_1,F_2)$.
  \begin{enumerate}[(i)]
  \item If $E_2$ has the UMD property, then
  \begin{equation*}
  \Gamma \left( f(X), Y \right) = f'(X) \Gamma(X,Y)
  \end{equation*}
  and
  \begin{equation*}
  \Gamma_{\pi} \left( f(X),Y \right) = \left( f'(X) \otimes \operatorname{Id}_{F_1}  \right) \Gamma_{\pi} \left( X,Y \right)
\end{equation*}.
\item If $F_2$ has the UMD property, then
\begin{equation*}
\Gamma\left(X,g\left(Y\right)\right)=\Gamma\left(X,Y\right)g'\left(Y\right)^{\ast}
\end{equation*}
and
\begin{equation*}
  \Gamma_{\pi} \left( X,g(Y) \right) = \left( \operatorname{Id}_{E_1} \otimes g'(Y) \right) \Gamma_{\pi} \left( X,Y \right).
\end{equation*}
  \end{enumerate}
\end{lemma}

The next lemma will be crucial for obtaining the integration by parts formula.

\begin{lemma}
  \label{lem:4}
  Let $E$ and $F$ be Banach spaces, $X \in \mathbb{D}^{1,2}(E)$ and $Y  \in \mathbb{D}^{1,2}(F)$. Then,
  \begin{equation}
    \label{eq:28}
    Q_{X,Y} = \mathbb{E} \left( \Gamma_{\pi} \left( X, -L^{-1}Y \right) \right)
    =
    \mathbb{E} \left( \Gamma_{\pi} \left( -L^{-1}X,Y \right) \right)
  \end{equation}
  and
  \begin{equation}
    \label{eq:29}
    R_{X,Y} = \mathbb{E} \left( \Gamma \left( X,-L^{-1} Y \right) \right)
    =
    \mathbb{E} \left( \Gamma \left( -L^{-1}X,Y \right) \right).
\end{equation}
\end{lemma}

\begin{proof}
  By Lemma~\ref{lem:2} and Lemma~\ref{lem:1}, the carré du champ remains continuous as a bilinear mapping from $\mathbb{D}^{1,2}(E) \times \mathbb{D}^{1,2}(F)$ to $L^1(\Omega;E  \widehat{\otimes}_{\pi} F)$ (or $L^1(\Omega;\mathfrak{N}(F^{\ast},E))$) when we apply $-L^{-1}$ to one of its arguments. Approximating in $\mathbb{D}^{1,2}(E)$ and $\mathbb{D}^{1,2}(F)$, it suffices to show the identities~\eqref{eq:28} for $X$ and $Y$ in the dense subsets $\mathcal{S}(\mathbb{R}) \otimes E$ and $\mathcal{S}(\mathbb{R}) \otimes F$, respectively. For such $X$ and $Y$, write
\begin{equation*}
  X = \sum_{j=1}^n X_j \otimes x_j \qquad \text{and} \qquad Y = \sum_{k=1}^{n} Y_k \otimes y_k
\end{equation*}
with $X_{j},Y_k \in \mathbb{D}^{1,2}(\mathbb{R})$, $x_j \in E$ and $y_k\in F$ for $1\leq j,k\leq n$ and some $n \in \mathbb{N}$. Then,
\begin{equation}
  \label{eq:10}
  Q_{X,Y} =
  \mathbb{E} \left(
    \left( X - \mathbb{E}(X) \right) \otimes Y
\right)
=
\sum_{j,k=1}^n \mathbb{E}
\left( \left( X_j - \mathbb{E}(X_j) \right) Y_k  \right)
x_j \otimes y_k
=
\sum_{j,k=1}^n R_{X_j,Y_k} \, x_j \otimes y_k.
\end{equation}
Let $\mathfrak{H}$ be the real, separable Hilbert space on which the underlying isonormal Gaussian process is defined and pick an orthonormal basis $\left\{ \mathfrak{h}_n \colon n \in \mathbb{N} \right\}$. As both $X_j - \mathbb{E}(X_j)$ and $Y_k$ are real valued, the well-known scalar-valued integration by parts formula yields
\begin{equation*}
  R_{X_j,Y_k}
  =
  \mathbb{E} \left(
    \left( X_j - \mathbb{E}(X_j) \right) \left( Y_k - \mathbb{E}(Y_k) \right)
  \right)
  =
  \E\left(X_j LL^{-1}Y_k\right)
  =
  \mathbb{E} \left( \Gamma(X_{j}, - L^{-1} Y_k )  \right).
\end{equation*}
Plugged into~\eqref{eq:10}, this yields
\begin{align*}
Q_{X,Y}=\sum_{j,k=1}^{n}\E\left(\Gamma\left(X_j,-L^{-1}Y_k\right)\right)x_j\otimes y_k=\E\left(\Gamma_{\pi}\left(X,-L^{-1}Y\right)\right).
\end{align*}
By symmetry of the scalar covariance identity, the same argument gives
\begin{equation*}
Q_{X,Y}=\E\left(\Gamma_{\pi}\left(-L^{-1}X,Y\right)\right).
\end{equation*}
The identity~\eqref{eq:29} for the operator-valued carré du champ follows by identifying the algebraic tensors with the respective finite-rank operators.
\end{proof}

We are now ready to prove that the integration by parts formula continues to hold in the Banach-valued setting. Recall the discussion on topological tensor products in Section~\ref{sec:topol-tens-prod}.

\begin{proposition}[Malliavin integration by parts for Banach-valued random elements]
  \label{MalliavinIBPtensorform}
  For Banach spaces $E$ and $F$, let $X \in \mathbb{D}^{1,2}(E)$, $Y \in \mathbb{D}^{1,2}(F)$ and $f \in \mathcal{C}^1_b(E,F^{\ast})$. Then,
  \begin{equation*}
    \operatorname{tr}_{\pi} \left( Q_{f(X),Y} \right)
=
\mathbb{E}\left(\left( f'(X), \Gamma_{\pi}(X,-L^{-1}Y) \right)_{\pi} \right),
\end{equation*}
where $f'(x)$ is identified with an element of $\mathcal{L}(E,F;\mathbb{R})$ and $\left( \cdot, \cdot \right)_{\pi}$ is the pairing between the projective tensor product and its dual.
For $u\in F^{\ast}\widehat{\otimes}_{\pi}F$, let $J_F\left(u\right)\in\mathfrak{N}\left(F\right)$ be the canonical nuclear operator on $F$. If $F$ has the approximation property, the operator-valued form is
\begin{equation*}
\operatorname{tr}_F\left(J_F\left(Q_{f\left(X\right),Y}\right)\right)
=\operatorname{tr}_F\left(\E\left(J_F\left(\left(f'\left(X\right)\otimes\Id_F\right)\Gamma_{\pi}\left(X,-L^{-1}Y\right)\right)\right)\right).
\end{equation*}
Here the operators being traced act on $F$. Their adjoints act on $F^{\ast}$ and are, respectively, $R_{f\left(X\right),Y}$ and $f'\left(X\right)\Gamma\left(X,-L^{-1}Y\right)$.

\end{proposition}

\begin{proof}
We first treat the case where $X$ and $Y$ are of the form
  \begin{equation*}
  X = \sum_{j=1}^n X_j \otimes x_j \qquad \text{and} \qquad Y = \sum_{k=1}^{n} Y_k \otimes y_k
\end{equation*}
with $X_{j},Y_k \in \mathcal{S}(\mathbb{R})$, $x_j \in E$ and $y_k\in F$ for $1\leq j,k\leq n$ and some $n \in \mathbb{N}$. Let $\left\{ \mathfrak{h}_k \colon k \in \mathbb{N} \right\}$ be an orthonormal basis of $\mathfrak{H}$. Then,
\begin{align} \notag
  \mathbb{E} \left( \left( f(X),Y - \mathbb{E}(Y) \right)_{F^{\ast},F} \right)
  &=
    \sum_{k=1}^n \mathbb{E} \left( \left( Y_k - \mathbb{E}(Y_k) \right) \left( f(X), y_k \right)_{F^{\ast},F}  \right)
  \\ &= \label{eq:24}
   \sum_{k=1}^n \mathbb{E} \left( \left( Y_k - \mathbb{E}(Y_k) \right) \varphi_k(X) \right),
\end{align}
where $\varphi_k  \in \mathcal{C}_b^1(E,\mathbb{R})$ is defined as $\varphi_k(x) =  \left( f(x), y_k \right)_{F^{\ast},F}$ for $k=1,\dots,n$. The Fréchet derivative of $\varphi_{k}$ is given by
\begin{equation*}
  \varphi'_k(x)(u) = \left( f'(x) u ,y_k \right)_{F^{\ast},F}
\end{equation*}
for $x,u \in E$, so that by the chain rule for the Malliavin derivative (Theorem~\ref{thm:4}),
\begin{equation*}
  D \varphi_k(X) = \varphi_k'(X) DX = \left( f'(X) DX, y_k \right)_{F^{\ast},F}.
\end{equation*}
Furthermore, the well-known finite-dimensional integration by parts formula gives
\begin{align*}
  \mathbb{E} \left( \left( Y_k - \mathbb{E}(Y_k) \right) \varphi_k(X) \right)
  &=
  \mathbb{E} \left(  \left( D \varphi_k(X), -DL^{-1}Y_k \right)_{\mathfrak{H}} \right)
  \\ &=
       \mathbb{E} \left(
       \sum_l^{}
       \left( f'(X) DX \, \mathfrak{h}_{l}, y_k \right)_{F^{\ast},F}  \left( -DL^{-1}Y_k \mathfrak{h}_l \right)
        \right)
  \\ &=
       \mathbb{E} \left(
       \sum_l^{}
       \left( f'(X) DX \, \mathfrak{h}_{l}, \left( -DL^{-1}Y_k \mathfrak{h}_l \right) y_k \right)_{F^{\ast},F}
        \right),
  \end{align*}
  where only finitely many summands are different from zero as the finitely many $X_j$ and $Y_k$ individually and hence collectively only depend on finitely many Gaussians $W(\mathfrak{h}_{l})$ (with $W$ denoting the underlying isonormal Gaussian process). Plugged into~\eqref{eq:24}, writing $b_{f'(X)}$ for the bounded bilinear form on $F$ determined by $f'(X)$ via $b_{f'(X)}(u,v) = f'(X)(u)(v)$, we obtain

\begin{align*}
  \operatorname{tr}_{\pi} \left( Q_{f(X),Y} \right)
  &=
    \mathbb{E} \left(
       \sum_l^{}
    \left( f'(X) DX \, \mathfrak{h}_{l}, \sum_{k=1}^n \left( -DL^{-1}Y_k \mathfrak{h}_l \right) y_k \right)_{F^{\ast},F}
     \right)
  \\ &=
       \mathbb{E} \left(
\sum_l^{}
       \left( f'(X) DX \, \mathfrak{h}_{l}, -DL^{-1}Y \mathfrak{h}_l \right)_{F^{\ast},F}
        \right)
  \\ &=
       \mathbb{E} \left(
       \sum_l^{}
       b_{f'(X)} \left( DX \, \mathfrak{h}_l , -DL^{-1}Y \mathfrak{h}_l \right)
        \right)
  \\ &=
       \mathbb{E} \left(
       \left( b_{f'(X)}, \Gamma_{\pi}(X,-L^{-1}Y) \right)_{\pi}
        \right).
\end{align*}
For the operator formulation, apply the canonical map $J_F$ to the tensor identity. On a simple tensor $a\otimes y\in F^{\ast}\otimes F$, the operator $J_F\left(a\otimes y\right)$ acts on $F$ by $z\mapsto a\left(z\right)y$, and
\begin{equation*}
\operatorname{tr}_F\left(J_F\left(a\otimes y\right)\right)=a\left(y\right).
\end{equation*}
Consequently, the operator identity follows from the tensor identity for finite-rank $X$ and $Y$ by linearity.

For the general case, we approximate $X \in \mathbb{D}^{1,2}(E)$ and $Y \in \mathbb{D}^{1,2}(F)$ by sequences $(X_n)_{n  \in \mathbb{N}} \subseteq \mathcal{F} \left( \mathcal{S}(\mathbb{R}), E \right)$ and $(Y_n)_{n \in \mathbb{N}} \subseteq \mathcal{F} \left( \mathcal{S}(\mathbb{R}), F \right)$, such that $\left\lVert X - X_n \right\rVert_{1,2} \to 0$ and $\left\lVert Y - Y_n \right\rVert_{1,2} \to 0$. Possibly passing to subsequences, we assume that  $(X_{n})_{n \in \mathbb{N}}$ and $(Y_n)_{n \in \mathbb{N}}$ converge almost surely. Then,
\begin{equation*}
  \operatorname{tr}_{\pi} \left( Q_{f(X_n),Y_n} \right) = \mathbb{E}
  \left(
     \left( b_{f'(X_n)}, \Gamma_{\pi}(X_n,-L^{-1}Y_n) \right)_{\pi}
  \right).
\end{equation*}

By the Cauchy--Schwarz inequality and boundedness of $f$ and $f'$,
\begin{align*}
\abs{\operatorname{tr}_{\pi}\left(Q_{f\left(X\right),Y}\right)-\operatorname{tr}_{\pi}\left(Q_{f\left(X_n\right),Y_n}\right)}
&\leq \norm{Y-\E\left(Y\right)}_{L^2\left(\Omega;F\right)}\norm{f'}_{\infty}\norm{X-X_n}_{L^2\left(\Omega;E\right)}\\
&\quad+2\norm{f}_{\infty}\norm{Y-Y_n}_{L^1\left(\Omega;F\right)},
\end{align*}
which tends to zero by the assumed Sobolev convergence.

Furthermore
\begin{align}
  \notag
  \mathbb{E} &\left|
  \left( f'(X), \Gamma_{\pi}(X,-L^{-1}Y) \right)_{\pi}
  -
                     \left( f'(X_{n}), \Gamma_{\pi}(X_n,-L^{-1} Y_n) \right)_{\pi}
                      \right|
  \\ &\qquad \leq \notag
              \mathbb{E} \left|
  \left( f'(X) - f'(X_{n}), \Gamma_{\pi}(X,-L^{-1} Y) \right)_{\pi}
       \right|
              \\ &\qquad\qquad \qquad + \notag
    \mathbb{E} \left|
                       \left( f'(X_n), \Gamma_{\pi}(X,-L^{-1} Y) - \Gamma_{\pi}(X_n,-L^{-1} Y_{n}) \right)_{\pi}
                       \right|
  \\ &\qquad \leq \notag
                 \mathbb{E} \left(
            \left\lVert f'(X) - f'(X_n) \right\rVert_{\mathcal{L}(E,F;\mathbb{R})}
            \,
       \left\lVert \Gamma_{\pi}(X,-L^{-1}Y) \right\rVert_{\pi}
            \right)
       \\ &\qquad\qquad \qquad + \notag
  \left\lVert f' \right\rVert_{\infty}
       \mathbb{E} \left(
       \left\lVert \Gamma_{\pi}(X,-L^{-1}Y) - \Gamma_{\pi}(X_n,-L^{-1} Y_n) \right\rVert_{\pi}
       \right)
\end{align}
and both terms on the right hand side vanish in the limit due to dominated convergence, as $\left\lVert \Gamma_{\pi}(X_n,-L^{-1}Y_n)
            -
            \Gamma_{\pi}(X,-L^{-1}Y) \right\rVert_{\pi}$ converges to zero by the boundedness of $\Gamma_{\pi}$ and $\left\lVert f'(X) - f'(X_n) \right\rVert_{\mathcal{L}(E,F;\mathbb{R})}$ converges to zero by continuity of $f'$ and almost sure convergence of $\left(X_n\right)_{n\in\N}$.

If $F$ has the approximation property, the nuclear trace
  $\operatorname{tr}_F$ is well defined and bounded by the nuclear
  norm (see~\cite[Chapter~29]{defant-floret:1993:tensor-norms-operator}). Applying the bounded map $J_F$ to the tensor identity and commuting the resulting continuous linear functional with expectation proves the operator formulation.
\end{proof}

We now show that jointly Radon Gaussians can always be represented as Radonifications of a common underlying isonormal Gaussian process. Their cross-covariance operators then coincide almost surely with the corresponding carré du champ operators. This result is in complete analogy with the scalar-valued case.

\begin{proposition}
  \label{prop:3}
  Let $Z_1$ and $Z_2$ be jointly Radon Gaussians, taking values in Banach spaces $E_1$ and $E_2$, respectively. Then there exists a real, separable Hilbert space $\mathfrak{H}$, an isonormal Gaussian process $W \colon \mathfrak{H} \to L^2(\Omega)$ and $\gamma$-radonifying operators $T_{Z_i} \colon \mathfrak{H} \to E_i$, such that the law of $Z_i$ coincides with the Radonification of $W$ by $T_{Z_i}$ and
\begin{equation*}
  R_{Z_i,Z_j} = \Gamma \left( Z_i,-L^{-1}Z_j \right) \qquad \text{and} \qquad Q_{Z_i,Z_j} = \Gamma_{\pi} \left( Z_i,-L^{-1}Z_j \right)
\end{equation*}
for all $1 \leq i,j \leq 2$, where all identities hold almost surely. In particular, the carré du champ of Radon Gaussians is almost surely deterministic.
\end{proposition}

\begin{proof}
\begingroup
Let $Z=Z_1\oplus Z_2$, and subtract its mean. The standard first-chaos
representation of a Radon Gaussian measure
(see~\cite[Chapter~3]{bogachev:1998:gaussian-measures}
or~\cite{neerven:2010:gamma-radonifying-operators-survey}) realizes
$Z$ as the radonification of an isonormal process $W$ over a separable
Hilbert space $\mathfrak H$ through a $\gamma$-radonifying operator
$T\colon\mathfrak H\to E_1\oplus E_2$. Let $T_i=P_iT$, where $P_i$ is
the coordinate projection. Then
$Z_i-\E\left(Z_i\right)=I_1\left(T_i\right)$, $DZ_i=T_i$ and
$-L^{-1}Z_i=Z_i-\E\left(Z_i\right)$. In particular, by the definitions
of the two contractions, we have $\Gamma_{\pi}\left(Z_i,-L^{-1}Z_j\right)
=\widehat C_{\mathfrak H,\pi}\left(T_i,T_j\right)$ and $\Gamma\left(Z_i,-L^{-1}Z_j\right)
=\widehat C_{\mathfrak H}\left(T_i,T_j\right)$.

Fix an orthonormal basis $\left(e_k\right)_{k\geq1}$ of $\mathfrak H$, let $P_N$ denote the orthogonal projection onto their first $N$-dimensional span and put $T_i^{\left(N\right)}=T_iP_N$. For the finite-rank Gaussian elements $Z_i^{\left(N\right)}=\sum_{k=1}^N W\left(e_k\right)T_ie_k$, the covariance identity is
\begin{equation*}
Q_{Z_i^{\left(N\right)},Z_j^{\left(N\right)}}=
\sum_{k=1}^N T_ie_k\otimes T_je_k
=\widehat C_{\mathfrak H,\pi}\left(T_i^{\left(N\right)},T_j^{\left(N\right)}\right).
\end{equation*}
Since $T_i^{\left(N\right)}\to T_i$ in $\gamma\left(\mathfrak H,E_i\right)$, the Gaussian series converge in $L^2\left(\Omega;E_i\right)$. The covariance tensors consequently converge in $E_i\widehat{\otimes}_{\pi}E_j$ by the Cauchy--Schwarz inequality, while the contractions converge there by their continuity. This proves the tensor identity $\Gamma_{\pi}\left(Z_i,-L^{-1}Z_j\right)=Q_{Z_i,Z_j}$, and applying the canonical tensor-to-operator map gives the operator identity.
\endgroup
\end{proof}

Combining Propositions~\ref{MalliavinIBPtensorform} and~\ref{prop:3} yields a version of the well-known Gaussian integration by parts formula. For classical Gaussian integration by parts on Banach spaces, see~\cite[Chapter~5]{bogachev:1998:gaussian-measures}. The statement below records its projective-tensor formulation.

\begin{corollary}[Gaussian integration by parts]
  \label{GaussianIBPfromMalIBPcoro}
  Let $X$ and $Y$ be jointly Radon Gaussian, taking values in Banach spaces $E$ and $F$, respectively, so that $(X,Y)$ is Radon and $f \in \mathcal{C}_b^1(E,F^{\ast})$. Then,
  \begin{equation*}
\operatorname{tr}_{\pi} \left( Q_{f(X),Y} \right)
=
\left( \mathbb{E} \left( f'(X) \right), Q_{X,Y} \right)_{\pi} ,
\end{equation*}
where $f'(x)$ is identified with an element of $\mathcal{L}(E,F;\mathbb{R})$ and $\left( \cdot, \cdot \right)_{\pi}$ is the pairing between the projective tensor product and its dual.
If $F$ has the approximation property, the equivalent operator identity is
\begin{equation*}
\operatorname{tr}_F\left(J_F\left(Q_{f\left(X\right),Y}\right)\right)
=\operatorname{tr}_F\left(J_F\left(\left(\E\left(f'\left(X\right)\right)\otimes\Id_F\right)Q_{X,Y}\right)\right).
\end{equation*}
\end{corollary}

\begin{remark}
  \label{rmk:6}
For $X\in\mathbb{D}^{1,2}\left(F^{\ast}\right)$ and $Y\in\mathbb{D}^{1,2}\left(F\right)$, the covariance identity of Lemma~\ref{lem:4} may be paired directly with the bounded evaluation bilinear form on $F^{\ast}\times F$. This gives
\begin{align*}
  \operatorname{tr}_{\pi} \left( Q_{X,Y} \right)
  &=
  \mathbb{E} \left( \left( \operatorname{Id}_{F^{\ast}}, \Gamma_{\pi}(X,-L^{-1}Y) \right)_{\pi} \right)
  \\ &=
  \mathbb{E} \left(
    \operatorname{tr}_{\pi} \left( \Gamma_{\pi} \left( X, -L^{-1}Y \right) \right)
       \right)
  \\ &=
       \mathbb{E} \left(
       \left( DX, -DL^{-1}Y \right)_{F^{\ast},F}
       \right).
\end{align*}
The pairing inside the last expectation is the trace duality pairing between the random~$\gamma$-radonifying operators $DX  \in  \gamma(\mathfrak{H},F^{\ast})$ and $-DL^{-1}Y  \in  \gamma(\mathfrak{H},F)$. Applying Lemma~\ref{lem:1} then gives
\begin{equation*}
  \mathbb{E} \left(
       \left( DX, -DL^{-1}Y \right)_{F^{\ast},F}
       \right)
     \leq
     \mathbb{E} \left(
       \left\lVert DX \right\rVert_{\gamma(\mathfrak{H},F^{\ast})}
       \left\lVert -DL^{-1}Y \right\rVert_{\gamma(\mathfrak{H},F)}
     \right),
\end{equation*}
and if we take $X$ and $Y$ to be Gaussian, then by Proposition~\ref{prop:3} the Malliavin derivatives become deterministic and the pseudo inverse of the generator disappears. As any two $\gamma$-radonifying operators $S \in \gamma(\mathfrak{H},F^{\ast})$ and $T \in \gamma(\mathfrak{H},F)$ can be written as the Malliavin derivatives of two Radon Gaussians $X$ and $Y$ such that $(X,Y)$ is jointly Gaussian (see the proof of Proposition~\ref{prop:3}), we recover the well-known continuity proof of the trace duality pairing for $\gamma$-radonifying operators (see Theorem 10.2 in~ \cite{neerven:2010:gamma-radonifying-operators-survey}).

For a bounded linear operator $T\colon E\to E^{\ast}$, Lemma~\ref{lem:4}, the projective duality pairing and Lemma~\ref{lem:2} yield
\begin{align*}
\abs{\operatorname{tr}_{\pi}\left(Q_{TX,X}\right)}
&=\abs{\E\left(\left(T,\Gamma_{\pi}\left(X,-L^{-1}X\right)\right)_{\pi}\right)}\\
&\leq\norm{T}_{\mathcal{L}\left(E,E^{\ast}\right)}
\E\left(\norm{DX}_{\gamma\left(\mathfrak{H},E\right)}\norm{-DL^{-1}X}_{\gamma\left(\mathfrak{H},E\right)}\right)\\
&\leq\norm{T}_{\mathcal{L}\left(E,E^{\ast}\right)}
\E\left(\norm{DX}_{\gamma\left(\mathfrak{H},E\right)}^2\right).
\end{align*}
This is a Banach-valued covariance estimate, with the pseudo-inverse constructed in Proposition~\ref{prop:1} and controlled by Lemma~\ref{lem:2}.
  In particular, taking $E=K$, where $K$ is some Hilbert space, $f$ the Riesz isometry and identifying $K$ with its dual gives
\begin{align*}
  \operatorname{tr}_{K} \left( R_{X} \right)
  &= \mathbb{E} \left(
    \left\lVert X - \mathbb{E} \left( X \right)  \right\rVert_K^2
  \right)
  \\ &\leq
  \mathbb{E} \left(
       \left\lVert DX \right\rVert_{\mathfrak{H} \widehat{\otimes}_{\sigma} K}
       \left\lVert -DL^{-1}X \right\rVert_{\mathfrak{H} \widehat{\otimes}_{\sigma} K}
       \right)
  \\ &\leq
       \sqrt{
       \mathbb{E} \left(
       \left\lVert DX \right\rVert_{\mathfrak{H} \widehat{\otimes}_{\sigma} K}^2
       \right)
       \mathbb{E} \left(
       \left\lVert -DL^{-1}X \right\rVert_{\mathfrak{H} \widehat{\otimes}_{\sigma} K}^2
       \right)
       }
  \\ &\leq
       \mathbb{E} \left(
       \left\lVert DX \right\rVert_{\mathfrak{H} \widehat{\otimes}_{\sigma} K}^2
       \right),
\end{align*}
which is the well-known vector valued Poincaré inequality (see for example~ \cite[Remark 5.5.6]{bogachev:1998:gaussian-measures}). Recall that $\mathfrak{H} \widehat{\otimes}_{\sigma} K$ can be identified with the Hilbert-Schmidt operators.
\end{remark}

Now that an integration by parts formula is available, one can apply the well-known smart path method. Recall that the carré du champ operator can be defined as a random tensor-valued mapping or a random nuclear operator. We use the same symbol to denote both versions, and it will be clear from the context whether a tensor or operator is meant.

\begin{proposition}
  \label{smartpathproposition}
  For a Banach space $E$, let $X \in \mathbb{D}^{1,2}(E)$ be centered and $Z$ be an
independent, centered and $E$-valued Radon Gaussian random element with covariance tensor $Q_Z$ and covariance operator $R_{Z}$. Define, for $t \in [0,1]$,
\begin{equation*}
U_t = \sqrt{t} X+ \sqrt{1-t} Z.
\end{equation*}
 Then, for every $f\in\mathcal{C}^2\left(E;\R\right)$ such that $\norm{f'}_{\infty}+\norm{f''}_{\infty}<\infty$, one has
\begin{equation}
  \label{eq:17}
\mathbb E (f(X))-\mathbb E( f(Z))=\frac12 \int_0^1 \mathbb E \left(
  \left( f''(U_t), \Gamma_{\pi}(X,-L^{-1}X) - Q_Z \right)_{\pi} \right) dt
\end{equation}
and consequently
\begin{equation}
 \label{eq:16}
 \left| \mathbb E (f(X))-\mathbb E( f(Z)) \right|
 \leq \frac{1}{2}
 \left\lVert f'' \right\rVert_{\infty}
 \E\left(\norm{\Gamma_{\pi}\left(X,-L^{-1}X\right)-Q_Z}_{\pi}\right).
\end{equation}
If $E$ has the approximation property, the same identity admits the following nuclear-operator formulation:
\begin{equation*}
\E\left(f\left(X\right)\right)-\E\left(f\left(Z\right)\right)
=\frac12\int_0^1\E\left(\operatorname{tr}_E\left(J_E\left(\left(f''\left(U_t\right)\otimes\Id_E\right)\left(\Gamma_{\pi}\left(X,-L^{-1}X\right)-Q_Z\right)\right)\right)\right)dt.
\end{equation*}
and
\begin{equation*}
 \left| \mathbb E (f(X))-\mathbb E( f(Z)) \right|
 \leq \frac{1}{2}
 \left\lVert f'' \right\rVert_{\infty}
 \E\left(\norm{\Gamma\left(X,-L^{-1}X\right)-R_Z}_{\nu}\right).
\end{equation*}

\end{proposition}
\begin{proof}
  We prove the identities for $\Gamma_{\pi}$.
  \begingroup
  For the nuclear-operator version, apply $J_E$ to the tensor
  \begin{equation*}
  \left(f''\left(U_t\right)\otimes\Id_E\right)
  \left(\Gamma_{\pi}\left(X,-L^{-1}X\right)-Q_Z\right)
  \in E^{\ast}\widehat{\otimes}_{\pi}E.
  \end{equation*}
  On elementary tensors, $\operatorname{tr}_E\circ J_E$ equals the
  projective pairing with $f''\left(U_t\right)$. The identity extends
  by continuity. The approximation property of $E$ gives the isometric
  identification of the projective tensor discrepancy with its
  nuclear-operator image, justifying the nuclear-norm estimate (see~\cite[Chapter~29]{defant-floret:1993:tensor-norms-operator}).
  \endgroup
  Define $\Phi(t)=\mathbb E \left( f(U_t)\right)$, so that
\begin{equation*}
\mathbb{E} \left( f(X)\right) - \mathbb{E} \left( f(Z)\right) = \int_0^1 \Phi'(t)dt.
\end{equation*}
Since $\norm{f'}_{\infty}<\infty$, the mean-value inequality gives
$\abs{f\left(x\right)}\leq\abs{f\left(0\right)}+\norm{f'}_{\infty}\norm{x}_E$.
Hence $\Phi$ is well defined, and $f'$ is bounded while $X,Z\in L^2\left(\Omega;E\right)$. For $0<t<1$, differentiation under the expectation is justified by the bound
$\norm{f'}_{\infty}\left(\norm{X}_E/\sqrt{t}+\norm{Z}_E/\sqrt{1-t}\right)/2$.
This bound is integrable in $t$ and expectation, so $\Phi$ is absolutely continuous on $\left[0,1\right]$ and the fundamental theorem of calculus applies. Moreover, $f'\in\mathcal C_b^1\left(E,E^{\ast}\right)$, which is the class required by Proposition~\ref{MalliavinIBPtensorform}. We thus obtain
\begin{align}
  \notag
  \Phi'(t) &= \mathbb E \left(f'(U_t)\left( \frac{X}{2 \sqrt{t}} - \frac{Z}{2 \sqrt{1-t}} \right) \right)
  \\ &= \notag
       \frac{1}{2 \sqrt{t}} \mathbb E \left( (X, f'(U_t))_{E,E^{\ast}} \right)
       -  \frac{1}{2 \sqrt{1-t}} \mathbb E \left((Z, f'(U_t))_{E,E^{\ast}} \right)
  \\ &= \notag
       \frac{1}{2 \sqrt{t}}
       \mathbb{E}
       \left(
       \operatorname{tr}_{\pi} \left( Q_{f'(\sqrt{t}X + z),X}
       \right)_{\mid z = \sqrt{1-t}Z }
       \right)
       \\ & \qquad \qquad  -  \label{phiprime}
            \frac{1}{2 \sqrt{1-t}}
            \mathbb{E} \left(
            \operatorname{tr}_{\pi} \left( Q_{f'(x + \sqrt{1-t}Z),Z} \right)_{\mid x = \sqrt{t} X}
            \right)
\end{align}
Integrating both terms by parts, i.e. applying Proposition~\ref{MalliavinIBPtensorform}, yields
\begin{equation*}
  \frac{1}{2 \sqrt{t}}
       \mathbb{E}
       \left(
       \operatorname{tr}_{\pi} \left( Q_{f'(\sqrt{t}X + z),X}
       \right)_{\mid z = \sqrt{1-t}Z }
       \right)
   =
  \frac{1}{2} \mathbb{E}\left( \left( f''(U_t), \Gamma_{\pi}(X,-L^{-1}X) \right)_{\pi}  \right)
\end{equation*}
and
\begin{equation*}
           \frac{1}{2 \sqrt{1-t}}
            \mathbb{E} \left(
            \operatorname{tr}_{\pi} \left( Q_{f'(x + \sqrt{1-t}Z),Z} \right)_{\mid x = \sqrt{t} X}
            \right)
  =
  \frac{1}{2} \mathbb{E}\left( \left(  f''(U_t), Q_Z \right)_{\pi}  \right).
\end{equation*}
Plugged into~\eqref{phiprime}, we obtain
\begin{align*}
  \Phi'(t) =
  \frac{1}{2}
  \mathbb{E}\left( \left(  f''(U_t) , \Gamma_{\pi}(X,-L_E^{-1}X) - Q_Z \right)_{\pi}  \right),
\end{align*}
which proves~\eqref{eq:17}. Furthermore,
\begin{align*}
  \left| \left( f''(U_t), \Gamma_{\pi}(X,-L^{-1}X) - Q_{Z} \right)_{\pi} \right|
  &\leq  \left\lVert f'' (U_t) \right\rVert_{\mathcal{L}(E,E;\mathbb{R})} \,
    \left\lVert \Gamma_{\pi}(X,-L^{-1}_E X) - Q_Z \right\rVert_{\pi}
  \\ &\leq
       \left\lVert f'' \right\rVert_{\infty} \,
       \left\lVert
    \Gamma_{\pi}(X,-L^{-1}_E X) - Q_Z \right\rVert_{\pi}.
\end{align*}
Together with~\eqref{eq:17}, this yields~\eqref{eq:16}.
\end{proof}

\section{Bounds in the bounded Lipschitz metric via regularization}
\label{sec:Skorokhod}
In this section, we combine the bound~\eqref{eq:16} obtained via integration by parts with the regularization procedure developed in~\cite{barbour-ross-zheng:2024:steins-method-smoothing} in order to derive bounds for the bounded Lipschitz metric on the Skorokhod space (or any closed subspace thereof). The regularization procedure is designed to establish a link between this metric and a smooth integral probability metric initially introduced in~\cite{barbour:1990:steins-method-diffusion}, which uses Fréchet differentiable functions and is therefore in general only able to guarantee convergence of finite-dimensional distributions. This smooth distance can be controlled by the difference between the projective norm of the carré du champ and the Gaussian covariance tensor via~\eqref{eq:16}. The regularization introduces additional terms ensuring tightness in such a way, that under mild assumptions global control by the aforementioned projective norm can be achieved via parameter optimization.

Recall that the bounded Lipschitz metric $d_{BL}$ on the class of Borel probability measures on a Banach space $E$ is defined as

\begin{equation*}
  d_{BL}(\mu,\nu) = \sup_{f  \in \operatorname{BL}(E)} \left| \int_E^{} f(x) \mu(dx) - \int_{E}^{} f(x) \nu(dx) \right|,
\end{equation*}

where $\operatorname{BL}(E)$ denotes the set of all bounded Lipschitz functions $f \colon E \to \mathbb{R}$ with Lipschitz constant at most one, such that $\left\lVert f \right\rVert_{\infty} \leq 1$.

For two $E$-valued Borel random elements, we define $d_{BL}(X,Y) = d_{BL}(P_{X},P_{Y})$, where $P_{X}$ and $P_{Y}$ denote the laws of $X$ and $Y$.

Throughout this section, we fix $T>0$, $d \in \mathbb{N}$ and denote by $\mathcal{D}$ the Skorokhod space $D([0,T];\mathbb{R}^{d})$ equipped with the uniform metric.

As in~\cite{barbour:1990:steins-method-diffusion}, we define $M^{0} \subseteq \mathcal{C}_b^2(\mathcal{D})$ to consist of all twice continuously Fréchet differentiable functions $f \colon \mathcal{D} \to \mathbb{R}$ such that $f$ and its first two derivatives are uniformly bounded and $f''$ is Lipschitz continuous, i.e. for which the norm $\left\lVert f \right\rVert_{M^0}$ defined by
\begin{equation*}
  \left\lVert f \right\rVert_{M^0}
  =
  \left\lVert f \right\rVert_{\mathcal{C}_b^2} +  \left\lVert f'' \right\rVert_{\text{Lip}}
  =
  \left\lVert f \right\rVert_{\infty}
  +
  \left\lVert f' \right\rVert_{\infty} + \left\lVert f'' \right\rVert_{\infty} + \left\lVert f'' \right\rVert_{\text{Lip}}
\end{equation*}
is finite.

The $\varepsilon$-regularization introduced  in~\cite{barbour-ross-zheng:2024:steins-method-smoothing}, is defined as follows.

\begin{definition}
  \label{def:epsreg}
  For $\varepsilon>0$, a Banach space $E$ and a Bochner integrable function $f \colon [0,T] \to E$, we define its $\varepsilon$-regularized version $f_{\varepsilon}$ by
\begin{equation*}
f_{\varepsilon}\left(t\right)=\frac{1}{2\varepsilon}\int_{t-\varepsilon}^{t+\varepsilon}f\left(\left[u\right]_{\left[0,T\right]}\right)\mathrm{d}u,
\end{equation*}
where $[ \cdot ]_{[0,T]} = \operatorname{min} \left\{ \operatorname{max} \left\{ 0, \cdot \right\}, T \right\}$ is the truncation operator, forcing its argument to stay in the interval $[0,T]$.
\end{definition}

The following lemma is straightforward to see.

\begin{lemma}
  \label{lem:8}
  For a Banach space $E$, $p \geq 1$ and $T>0$, let $X = (X(t))_{t \in [0,T]}$ be a jointly measurable stochastic process such that $X(t) \in L^p(\Omega;E)$ for all $t \in [0,T]$. Then the following is true:
  \begin{enumerate}[(i)]
  \item If\, $\lim_{\varepsilon \to 0} \mathbb{E} \left( \left\lVert X_{\varepsilon} - X \right\rVert_{\infty}^p \right) = 0$, then $X$ has almost surely continuous sample paths.
  \item Conversely, if $X$ has almost surely continuous sample paths and $\mathbb{E} \left( \left\lVert X \right\rVert_{\infty}^p \right) < \infty$, then
$\lim_{\varepsilon \to 0} \mathbb{E} \left( \left\lVert X_{\varepsilon} - X \right\rVert_{\infty}^p \right) = 0$.
\end{enumerate}
\end{lemma}

Now we can formulate the mentioned link between the smooth integral probability metric based on $M^0$ and the bounded Lipschitz distance on $\mathcal{D}$. The following theorem is a special case of~\cite[Thm. 1.1 and Cor. 1.3]{barbour-ross-zheng:2024:steins-method-smoothing} adjusted for our purposes.

\begin{theorem}
\label{thm:OnePointTwenty}
Let $X$ and $Z$ be $\mathcal{D}$-valued random elements such that $Z$ has
almost surely continuous sample paths. For $\varepsilon>0$, let $X_{\varepsilon}$ and $Z_{\varepsilon}$ be their $\varepsilon$-regularizations and suppose there exists $\kappa > 0 $ such that, for any $h \in M^0$, we have
\begin{equation}
\label{eq:OneThirteen}
\abs{\mathbb{E}(h(X))-\mathbb{E}(h(Z))} \leq
\kappa \norm{h}_{M^0}.
\end{equation}
Then, for any $\varepsilon >0$ and $\delta>0$, one has
\begin{equation*}
d_{BL} \left(X,Z\right)
\leq \mathbb{E}\left(\norm{X_\varepsilon-X}_{\infty}\right)
+
\mathbb{E}\left(\norm{Z_\varepsilon-Z}_{\infty}\right)
+2 \left(\sqrt{T} s_d + \sqrt{d} \right) \delta + \frac{4(T+2)}{\varepsilon^2 \delta^2} \, \kappa,
\end{equation*}
where $s_d = \mathbb{E} \left( \sup_{t \in [0,1]} \left\lVert B_t \right\rVert_{\mathbb{R}^d} \right)$ is the supremum of the norm of a $d$-dimensional standard Brownian motion.
\end{theorem}

In the above special case, one can immediately optimize in $\delta$ and then use our bound from Proposition~\ref{smartpathproposition}, which satisfies~\eqref{eq:OneThirteen} with
\begin{equation*}
  \kappa = \frac{1}{2} \norm{\Gamma_{\pi}\left(X,-L^{-1}X\right)-Q_Z}_{L^1\left(\Omega;\mathcal{D}\widehat{\otimes}_{\pi}\mathcal{D}\right)}
\end{equation*}
as $\left\lVert h'' \right\rVert_{\infty} \leq \left\lVert h \right\rVert_{M^0}$. This yields the following result.

\begin{theorem}
  \label{thm:2}
  Let $X$ and $Z$ be centered $\mathcal{D}$-valued Radon random elements such that $Z$ is Radon Gaussian with almost surely continuous sample paths and $X \in \mathbb{D}^{1,2}(\mathcal{D})$. Then, for any $\varepsilon>0$, one has
\begin{equation*}
  d_{BL}(X,Z) \leq
  \mathbb{E}\left(\norm{X_\varepsilon-X}_{\infty}\right)
+
\mathbb{E}\left(\norm{Z_\varepsilon-Z}_{\infty}\right)
+ \frac{C_{T,d}}{\varepsilon^{2/3}} \left\lVert \Gamma_{\pi}(X,-L_E^{-1}X) - Q_Z \right\rVert_{L^1(\Omega;\mathcal{D} \widehat{\otimes}_{\pi} \mathcal{D})}^{1/3},
\end{equation*}
where
\begin{equation*}
  C_{T,d} =
  3\left(2\left(T+2\right)\left(\sqrt{T}s_d+\sqrt{d}\right)^2\right)^{1/3}.
\end{equation*}
\end{theorem}

\begin{remark}
  \label{rmk:3}
  As already mentioned, Theorem~\ref{thm:OnePointTwenty}
  is a special case of~\cite[Theorem 1.1 and Corollary 1.3]{barbour-ross-zheng:2024:steins-method-smoothing},
  where a bound for the L\'{e}vy--Prokhorov distance
  $d_{LP}$ is also derived. We do not pursue such a bound
  here, since optimizing the regularization parameters is
  more involved without explicit estimates for the
  regularization errors. We recall that, under our convention
  $\norm{h}_{\infty}\leq1$ and
  $\operatorname{Lip}\left(h\right)\leq1$, the standard
  comparison between the two distances gives
  (see~\cite{dudley:1968:distances-probability-measures})
  \begin{equation}
    \label{eq:19}
    d_{BL}\leq 3d_{LP}.
  \end{equation}
\end{remark}

\begin{remark}
  \label{rmk:2}
  In the setting of Theorem~\ref{thm:2}, if we assume that there exists a non-negative function $\varphi  \in C([0,\infty))$ vanishing at zero such that
  \begin{equation*}
    \operatorname{max} \left\{ \mathbb{E}\left(\norm{X_\varepsilon-X}_{\infty}\right),  \mathbb{E}\left(\norm{Z_\varepsilon-Z}_{\infty}\right) \right\} \leq \varphi(\varepsilon)
  \end{equation*}
  for all $\varepsilon>0$, this implies the existence of a continuous function $\psi$ vanishing at zero, such that
\begin{equation}
  d_{BL}(X,Z) \leq \psi \left(
    \left\lVert \Gamma_{\pi}(X,-L^{-1}X) - Q_Z \right\rVert_{L^1(\Omega;\mathcal{D} \widehat{\otimes}_{\pi} \mathcal{D})} \right).
\end{equation}
Analogously, taking a sequence $(X_n)$ instead of $X$ and assuming
\begin{equation}
    \label{eq:47}
    \operatorname{max} \left\{ \mathbb{E}\left(\norm{X_{n,\varepsilon}-X_n}_{\infty}\right),  \mathbb{E}\left(\norm{Z_\varepsilon-Z}_{\infty}\right) \right\} \leq \varphi(\varepsilon) + a_n
  \end{equation}
  for all $\varepsilon>0$ and $n \in \mathbb{N}$ such that $a_n \geq 0$ and $\lim_{n \to \infty} a_n = 0$ yields the existence of a continuous function $\psi$ vanishing in zero satisfying
\begin{equation*}
  d_{BL}(X_n,Z) \leq a_{n} + \psi \left(
    \left\lVert \Gamma_{\pi}(X_n,-L_E^{-1}X_n) - Q_Z \right\rVert_{L^1(\Omega;\mathcal{D} \widehat{\otimes}_{\pi} \mathcal{D})} \right),
\end{equation*}
In both scenarios, this can be seen by choosing $\varepsilon=\kappa^{\alpha}$ for some $\alpha \in (0,1/2)$. A similar statement is true in the more general case when the left hand side of~\eqref{eq:47} is bounded by suitable functions $\varphi_n(\varepsilon)$.
\end{remark}

As a somewhat more concrete example, we carry out the $\varepsilon$-optimization for the case where the $\varepsilon$-regularizations have a Hölder-continuous modulus of continuity.

\begin{proposition}
\label{MainabstractBoundWasserstein}
In the setting of Theorem~\ref{thm:2}, assume that there exist positive constants $M$ and $\alpha$ such that
\begin{equation*}
\max\left\{\E\left(\norm{X_\varepsilon-X}_{\infty}\right),\E\left(\norm{Z_\varepsilon-Z}_{\infty}\right)\right\}\leq M\varepsilon^{\alpha}
\end{equation*}
for every $\varepsilon>0$. Put
\begin{equation*}
\Delta=\E\left(\norm{\Gamma_{\pi}\left(X,-L^{-1}X\right)-Q_Z}_{\pi}\right).
\end{equation*}
Then
\begin{equation*}
d_{\mathrm{BL}}\left(X,Z\right)\leq C_{\alpha,T,M,d}\Delta^{\alpha/\left(3\alpha+2\right)},
\end{equation*}
where one may take $C_{\alpha,T,M,d}=2M+C_{T,d}$.
Indeed, for $\Delta>0$, choose $\varepsilon=\Delta^{1/\left(3\alpha+2\right)}$ in Theorem~\ref{thm:2}. The case $\Delta=0$ follows by letting $\varepsilon\downarrow0$.
\end{proposition}

Of course, if the joint topological support $E$ of the Radon random elements $X$ and $Z$ is a closed subspace of $\mathcal{D}$, most notably the space $\mathcal{C}([0,T];\mathbb{R}^{d})$ of continuous functions, our bounds for $d_{BL}(X,Z)$ remain valid with the same numerical bounds when the metrics are restricted to $E$. The corresponding statement for the Lévy--Prokhorov metric follows by restricting measures to their common closed support. For $d_{BL}$, note that on the one hand, for any $f \in BL(\mathcal{D})$ its restriction $f_{\restriction E}$ is in $BL(E)$, so that $d_{\operatorname{BL}(\mathcal{D})}(X,Z) \leq d_{\operatorname{BL}(E)}(X,Z)$. On the other hand, for $g \in BL(E)$, its McShane-Whitney extension $\widetilde{g}$ (see \cite[Theorem 1]{mcshane:1934:extension-range-functions} or \cite{whitney:1934:analytic-extensions-differentiable}) to $\mathcal{D}$ yields a function in $BL(\mathcal{D})$ as it preserves the values of both, the Lipschitz- and the uniform norm. This shows $d_{BL(E)}(X,Z) \leq d_{BL(\mathcal{D})}(X,Z)$.

Finally, we recall the well-known facts that the uniform topology on $\mathcal{D}$ is stronger than all of the four standard Skorokhod topologies, and that uniform convergence in $C([0,T];\mathbb{R}^{d})$ is equivalent to convergence in the $J_1$-topology.

\section{Contraction bounds on Wiener chaos}
\label{sec:chaos}

In this section, we show how the carré du champ norms appearing in the bounds in Section~\ref{sec:Skorokhod} can be controlled in terms of kernel contractions, when the random element $X$ is an element of a Wiener chaos (or a finite sum thereof). Throughout, $E$ and $F$ will be real Banach spaces and $\mathfrak{H}$ a real, separable Hilbert space.

Recall from Section~\ref{sec:gamma-radon-oper} the space $\gamma^p(\mathfrak{H},E)$ of iterated $\gamma$-radonifying $E$-valued mappings and the subspace $\gamma_{sym}^{p}(\mathfrak{H},E)$ of its symmetrization, and from Section~\ref{sec:wiener-ito-chaos} the definition of the multiple integral $I_p(f)$ for $f \in \gamma_{sym}^p(\mathfrak{H},E)$.

We will show that, as in finite-dimension, the carré du champ of a multiple integral $I_p(f)$ decomposes into a finite sum of multiple integrals whose kernels are given by contractions of $f$. This will then allow to control the nuclear norm by these contractions.

For kernels $f \in \gamma^{p}(\mathfrak{H},E)$ and $g \in \gamma^p(\mathfrak{H},E)$, we denote the symmetrizations of the tensor- and operator-valued contractions (see Lemma~\ref{lem:1} for their definition) by
\begin{equation*}
  \widehat{C}_{\mathfrak{H},sym}^{r}(f,g) = \operatorname{Sym} \widehat{C}_{\mathfrak{H}}(f,g)
\end{equation*}
and
\begin{equation*}
  \widehat{C}_{\mathfrak{H},sym,\pi}^{r}(f,g) = \operatorname{Sym} \widehat{C}_{\mathfrak{H},\pi}(f,g)
\end{equation*}
respectively, where $\operatorname{Sym}$ is the symmetrization operator acting on the $q+p-2r$ remaining $\mathfrak{H}$-variables (see Subsection~\ref{sec:gamma-radon-oper}).

As in the scalar-case, the linear symmetrization operator is continuous, i.e. the norm of a symmetrized contraction is bounded by its non-symmetrized norm (see~\cite[Thm. 4.2]{maas:2010:malliavin-calculus-decoupling}), so that together with Lemma~\ref{lem:1} we have
\begin{align*}
  \notag
  \left\lVert \widehat{C}_{\mathfrak{H},\pi,sym}^{r}(f,g) \right\rVert_{\gamma^{(p+q-2r)}(\mathfrak{H}, E  \widehat{\otimes}_{\pi} F)}
  &\leq
  \left\lVert \widehat{C}_{\mathfrak{H},\pi}^{r}(f,g) \right\rVert_{\gamma^{(p+q-2r)}(\mathfrak{H}, E  \widehat{\otimes}_{\pi} F)}
  \\ &\leq
  \left\lVert f \right\rVert_{\gamma^p(\mathfrak{H},E)} \left\lVert g \right\rVert_{\gamma^q(\mathfrak{H},F)}.
\end{align*}
and the same inequalities hold for the operator-valued counterpart.

When no ambiguity can arise, we will follow the de facto standard notational convention of the Malliavin community and write $f \otimes_{r} g$ for $\widehat{C}_{\mathfrak{H}}^r(f,g)$ and $f  \widetilde{\otimes}_r g$ for its symmetrization. The tensor valued counterparts will be denoted by $f \otimes_{r,\pi} g$ and $f \widetilde{\otimes}_{r,\pi} g$, respectively. Note that the familiar scalar-valued contractions are obtained by choosing $E=F=\mathbb{R}$, as $\mathbb{R}  \widehat{\otimes}_{\pi} \mathbb{R} \simeq \mathfrak{N}(\mathbb{R}, \mathbb{R}) \simeq \mathbb{R}$ and $\gamma^p(\mathfrak{H},\mathbb{R}) \simeq \mathfrak{H}^{\widehat{\otimes}_{\sigma} p}$, and analogously $\gamma^p_{sym}(\mathfrak{H},\mathbb{R}) \simeq \mathfrak{H}^{\widehat{\odot}_{\sigma} p}$

As a caveat, we would like to mention again that for $\mathfrak{H} = L^{2}(M) = L^2(M,\mathcal{M},\mu)$, with $\mu$ a $\sigma$-finite measure without atoms on a measurable space $(M,\mathcal{M})$, one has to work with representable operators (see Section~\ref{sec:wiener-ito-chaos}). For the corresponding representing functions however, the contraction is obtained in the same way as in the scalar-valued case, i.e. by identifying and then integrating out variables.

We now extend the product formula (see for example~\cite[Prop. 1.1.3]{nualart:2006:malliavin-calculus-related} or~ \cite[Thm. 2.7.10]{nourdin-peccati:2012:normal-approximations-malliavin}) to the vector-valued setting. As in the previous sections, given $u \in E$ and $v \in F$, we denote by $u \otimes v$ both the simple tensor in $E \otimes F$ and the rank one operator in $\mathcal{F}(G^{\ast},E)$. Which definition is meant can always be inferred from the context.

\begin{proposition}
  \label{thm:5}
  For $m \in (1,\infty)$, $p,q \in \mathbb{N}$, $f \in \gamma^p_{sym}(\mathfrak{H},E)$ and $g \in \gamma^q_{sym}(\mathfrak{H},F)$, one has
\begin{equation}
    \label{eq:32}
  I_p(f) \otimes I_q(g) = \sum_{r=0}^{p \land q} r! \binom{p}{r} \binom{q}{r} I_{p+q-2r} \left(
    f \widetilde{\otimes}_{r,\pi} g
  \right)
\end{equation}
in $L^m(\Omega; E  \widehat{\otimes}_{\pi} F)$ and, interpreting the left hand side as a random rank-one operator,
\begin{equation*}
    I_p(f) \otimes I_q(g) = \sum_{r=0}^{p \land q} r! \binom{p}{r} \binom{q}{r} I_{p+q-2r} \left(
    f \widetilde{\otimes}_{r} g
  \right)
\end{equation*}
in $L^m(\Omega; \mathfrak{N}(F^{\ast},E))$.
\end{proposition}

\begin{proof}
  We first show the claim for the tensor-valued contractions.
  Let $\left\{ \mathfrak{h}_k \colon k \in \mathbb{N} \right\}$ be an orthonormal basis of $\mathfrak{H}$.
  For finite rank kernels $f \in \mathcal{F}_{sym}^{p}(\mathfrak{H},E) \simeq \mathfrak{H}^{\odot p} \otimes E$ and $g \in \mathcal{F}_{sym}^q(\mathfrak{H},F) \simeq \mathfrak{H}^{\odot q} \otimes F$, we have representations as in~\eqref{eq:33} of the form
  \begin{equation*}
 f = \sum_{\mathbf{j} \in [n]^{p}} \mathfrak{h}_{\mathbf{j}} \otimes x_{\mathbf{j}} \qquad \text{and} \qquad
 g = \sum_{\mathbf{k} \in [n]^{q}} \mathfrak{h}_{\mathbf{k}} \otimes y_{\mathbf{k}},
\end{equation*}
where $n \in \mathbb{N}$, $x_{\mathbf{j}} \in E$ and $y_{\mathbf{k}} \in F$ for $\mathbf{j} \in [n]^{p}$ and $\mathbf{k} \in [n]^{q}$, respectively. For clarity, and just in this proof, we will add the codomain of the kernel the operator $I_p$ acts on as a superscript, writing for example $I_{p}^{E}(f)$ instead of $I_p(f)$. Analogously, we add the codomains of the kernels to the contraction operation, writing for example $f \widetilde{\otimes}_{r,\pi}^{E,F} g$ instead of $f \widetilde{\otimes}_{r,\pi} g$.

Then, by the product formula for real-valued kernels (see for example~ \cite[Prop. 1.1.3]{nualart:2006:malliavin-calculus-related} or~\cite[Thm. 2.7.10]{nourdin-peccati:2012:normal-approximations-malliavin}),
\begin{align*}
  I_{p}^{E}(f) \otimes I_{q}^{F}(g)
  &=
    \sum_{\mathbf{j} \in [n]^p, \mathbf{k} \in [n]^{q}}^{}
    I_{p}^{\mathbb{R}}(\mathfrak{h}_{\mathbf{j}}) I_{q}^{\mathbb{R}}(\mathfrak{h}_{\mathbf{k}}) x_{\mathbf{j}} \otimes y_{\mathbf{k}}
  \\ &=
       \sum_{r=0}^{p \land q} r! \binom{p}{r} \binom{q}{r}
       \sum_{\mathbf{j} \in [n]^p, \mathbf{k} \in [n]^{q}}^{}
       I_{p+q-2r}^{\mathbb{R}} \left( \mathfrak{h}_{\mathbf{j}}
       \widetilde{\otimes}_r^{\mathbb{R},\mathbb{R}}
       \mathfrak{h}_{\mathbf{k}}
       \right)
       x_{\mathbf{j}} \otimes y_{\mathbf{k}}
  \\ &=
       \sum_{r=0}^{p \land q} r! \binom{p}{r} \binom{q}{r}
       I_{p+q-2r}^{E \widehat{\otimes}_{\pi} F} \left(
       \sum_{\mathbf{j} \in [n]^p, \mathbf{k} \in [n]^{q}}^{}
        \left( \mathfrak{h}_{\mathbf{j}}
       \widetilde{\otimes}_r^{\mathbb{R},\mathbb{R}}
       \mathfrak{h}_{\mathbf{k}} \right) \,
       x_{\mathbf{j}} \otimes y_{\mathbf{k}}
       \right)
  \\ &=
       \sum_{r=0}^{p \land q} r! \binom{p}{r} \binom{q}{r}
       I_{p+q-2r}^{E  \widehat{\otimes}_{\pi} F} \left(
              f \widetilde{\otimes}_{r,\pi}^{E,F} g
       \right).
\end{align*}

For general kernels $f \in \gamma^p(\mathfrak{H},E)$ and $g \in \gamma^q(\mathfrak{H},F)$, we approximate by finite-rank sequences $(f_n)_{n \in \mathbb{N}} \subseteq \mathcal{F}^p_{sym}(\mathfrak{H},E)$ and $(g_n)_{n \in \mathbb{N}} \subseteq \mathcal{F}^q_{sym}(\mathfrak{H},F)$. The Wiener-It\^{o} norm equivalence (see~\eqref{eq:35} in Section~\ref{sec:wiener-ito-chaos}) and the fact that changing the $L^2$-norm in the $\gamma$-radonifying norm to an $L^m$ norm for some $m \in [1,\infty)$ yields an equivalent $\gamma$-norm due to the Kahane-Khintchine inequalities, we obtain that
$I_p(f_n) \to I_p(f)$ in $L^m(\Omega;E)$ and also that $I_q(g_n) \to I_q(g)$ in $L^m(\Omega;F)$ for $m \in [1,\infty)$.
Together with boundedness of the contraction operator (see~\eqref{eq:26}) and of the multiple integral operator $I_{p}$, this also yields that
\begin{equation*}
         I_{p+q-2r}^{E  \widehat{\otimes}_{\pi} F} \left(
              f_n \widetilde{\otimes}_{r,\pi}^{E,F} g_n
            \right)
            \to
          I_{p+q-2r}^{E  \widehat{\otimes}_{\pi} F} \left(
              f \widetilde{\otimes}_{r,\pi}^{E,F} g
            \right)
\end{equation*}
in $L^m(\Omega;E  \widehat{\otimes}_{\pi} F)$.
\end{proof}

\begin{remark}
  \label{rmk:4}
  By the universal property of the projective tensor product, the “multiplication map”
\begin{equation*}
  T \colon \mathcal{H}_{p}(E) \otimes \mathcal{H}_q(F) \to L^m(\Omega;E \widehat{\otimes}_{\pi} F)
\end{equation*}
induced by~\eqref{eq:32} extends continuously to $\mathcal{H}_p\left(E\right)\widehat{\otimes}_{\pi}\mathcal{H}_q\left(F\right)$, with a bound depending on $p$, $q$ and $m$. This follows by H\"older's inequality and moment equivalence on each homogeneous Wiener chaos.
\end{remark}

Via the product formula, one can now obtain a decomposition of the carré du champ into a sum of multiple integrals.

\begin{proposition}
  \label{prop:5}
  For $p,q \in \mathbb{N}$, $f \in \gamma^p_{sym}(\mathfrak{H},E)$ and $g \in \gamma^q_{sym}(\mathfrak{H},F)$, one has
\begin{equation*}
  \Gamma_{\pi} \left(I_p(f),I_q(g)\right) =
  \sum_{r=1}^{p \land q} a_{p,q,r} I_{p+q-2r} \left( f \widetilde{\otimes}_{r,\pi} g \right)
\end{equation*}
and
\begin{equation*}
  \Gamma \left(I_p(f),I_q(g)\right) =
  \sum_{r=1}^{p \land q} a_{p,q,r} I_{p+q-2r} \left( f \widetilde{\otimes}_{r} g \right),
\end{equation*}
where
\begin{equation}
  \label{eq:27}
  a_{p,q,r} = pq (r-1)! \binom{p-1}{r-1} \binom{q-1}{r-1}.
\end{equation}
\end{proposition}

\begin{proof}
  By the same density argument as in the proof of Proposition~\ref{thm:5}, it suffices to show the claim for $f  \in \mathfrak{H}^{\odot p} \otimes E$ and $g \in \mathfrak{H}^{\odot q} \otimes F$, so that
  \begin{equation*}
  I_p(f) = \sum_{j=1}^n I_p(f_{j}) \otimes x_j \qquad \text{and} \qquad I_q(g) = \sum_{k=1}^n I_q(g_k) \otimes y_{k}
\end{equation*}
with $f_j \in \mathfrak{H}^{\odot p}$, $g_k \in \mathfrak{H}^{\odot q}$, $x_j \in E$ and $y_k \in F$ for $1 \leq j,k \leq n$.
Then, by bilinearity,
\begin{equation*}
  \Gamma_{\pi}(I_p(f), I_q(g)) = \sum_{j,k=1}^n \Gamma \left(I_p(f_j), I_q(g_k)\right) x_j \otimes y_k.
\end{equation*}
The claim now follows from the scalar-valued identity
\begin{equation*}
  \Gamma \left( I_p(f_j), I_q(g_k) \right) = \sum_{r=1}^{p \land q} a_{p,q,r} I_{p+q-2r} \left( f_j \widetilde{\otimes}_r g_k \right),
\end{equation*}
which is the classical scalar Malliavin contraction identity. See~\cite[Lemma~3.7]{nourdin-peccati-reveillac:2010:multivariate-normal-approximation} or~\cite[Proof of Lemma~6.2.1]{nourdin-peccati:2012:normal-approximations-malliavin}. The finite-rank expansion above and continuity of the projective contraction provide its Banach-valued extension. Applying the canonical tensor-to-operator map gives the operator-valued identity.
\end{proof}

We are now in a position to prove the main result of this section, namely a bound for the nuclear norm of the random covariance difference in terms of norms of contracted kernels. Note that, up to the constant $c_{p,q,m}$, these contraction bounds reduce to the classical scalar-valued bounds for $E=F=\mathbb{R}$ (which can be found in the references given above in the proof of Proposition~\ref{thm:5}).

\begin{theorem}
  \label{thm:6}
  For $m \in [1,\infty)$, $p,q \in \mathbb{N}$, $f \in \gamma^p_{sym}(\mathfrak{H},E)$ and $g \in \gamma^q_{sym}(\mathfrak{H},F)$, let $X=I_p(f)$, $Y=I_q(g)$ and denote the cross-covariance tensor and operator of $X$ and $Y$ by $Q_{X,Y}$ and $R_{X,Y}$, respectively. Then
  \begin{multline*}
  \left\lVert \Gamma_{\pi}(X,-L^{-1}Y) - Q_{X,Y} \right\rVert_{L^m(\Omega; E  \widehat{\otimes}_{\pi} F)}
  \\ \leq
  \frac{1}{q}
  \sum_{r=1}^{p\land q-\1_{\left\{p=q\right\}}} a_{p,q,r}
   \, \mathfrak{C}_{p+q-2r,m} \,
  \left\lVert f \widetilde{\otimes}_{r,\pi} g \right\rVert_{\gamma^{p+q-2r}(\mathfrak{H},E  \widehat{\otimes}_{\pi} F)},
\end{multline*}
and
\begin{multline*}
  \left\lVert \Gamma(X,-L^{-1}Y) - R_{X,Y} \right\rVert_{L^m(\Omega;\mathfrak{N}(F^{\ast},E))}
  \\ \leq
  \frac{1}{q}
  \sum_{r=1}^{p\land q-\1_{\left\{p=q\right\}}} a_{p,q,r}
  \, \mathfrak{C}_{p+q-2r,m} \,
  \left\lVert f \widetilde{\otimes}_r g \right\rVert_{\gamma^{p+q-2r}(\mathfrak{H},\mathfrak{N}(F^{\ast},E))},
\end{multline*}
where the positive constants $\mathfrak{C}_{p+q-2r,m}$ and $a_{p,q,r}$ are defined at~\eqref{eq:35} and~\eqref{eq:27}, respectively.
\end{theorem}

\begin{proof}
  By Proposition~\ref{prop:5}, one has
\begin{equation*}
  \Gamma_{\pi}(X,-L^{-1}Y) = \frac{1}{q} \Gamma_{\pi}(X,Y) =  \frac{1}{q} \sum_{r=1}^{p \land q} a_{p,q,r} \, I_{p+q-2r} \left( f \widetilde{\otimes}_{r,\pi} g \right).
\end{equation*}
If $p \neq q$, then both $Q_{X,Y}$ and $R_{X,Y}$ vanish as elements from different Wiener chaoses are uncorrelated (this transfers from the finite-dimensional case by approximation). If $p=q$, we can take expectation on both sides of the above sum and obtain by Lemma~\ref{lem:4} and the fact that all multiple integrals of positive order are centered, that
\begin{equation*}
  Q_{X,Y} = \mathbb{E} \left( \Gamma_{\pi}(X,-L^{-1}Y) \right) = \frac{a_{p,q,p}}{q} I_0 \left( f \widetilde{\otimes}_{p,\pi} g \right).
\end{equation*}
Thus the $r=p\land q$ term must be retained when $p\neq q$: it is a nonconstant chaos component, even though the cross-covariance vanishes. Therefore
\begin{equation*}
  \Gamma_{\pi}(X,-L^{-1}Y) - Q_{X,Y} = \frac{1}{q} \sum_{r=1}^{p\land q-\1_{\left\{p=q\right\}}} a_{p,q,r} I_{p+q-2r} \left( f \widetilde{\otimes}_{r,\pi} g \right),
\end{equation*}
so that by the triangle inequality and the Wiener-It\^{o} equivalence of norms (see \eqref{eq:35} in Section~\ref{sec:wiener-ito-chaos}),
\begin{align*}
  &\left\lVert \Gamma_{\pi}(X,-L^{-1}Y) - Q_{X,Y} \right\rVert_{L^m(\Omega; E  \widehat{\otimes}_{\pi} F)}
  \\ &\qquad \qquad\leq
  \frac{1}{q} \sum_{r=1}^{p\land q-\1_{\left\{p=q\right\}}} a_{p,q,r}
    \left\lVert I_{p+q-2r} \left( f \widetilde{\otimes}_{r,\pi} g \right)
    \right\rVert_{L^m(\Omega; E  \widehat{\otimes}_{\pi} F)}
  \\ &\qquad \qquad\leq
       \frac{1}{q}
        \sum_{r=1}^{p\land q-\1_{\left\{p=q\right\}}} a_{p,q,r} \, \mathfrak{C}_{p+q-2r,m} \,
       \left\lVert f \widetilde{\otimes}_{r,\pi} g \right\rVert_{\gamma^{(p+q-2r)}(\mathfrak{H}, E  \widehat{\otimes}_{\pi} F)}
\\ &\qquad \qquad\leq
       \frac{1}{q}
        \sum_{r=1}^{p\land q-\1_{\left\{p=q\right\}}} a_{p,q,r} \, \mathfrak{C}_{p+q-2r,m} \,
       \left\lVert f \otimes_{r,\pi} g \right\rVert_{\gamma^{(p+q-2r)}(\mathfrak{H}, E  \widehat{\otimes}_{\pi} F)}
\end{align*}
The operator-valued version follows analogously.
\end{proof}

Recall from Section~\ref{sec:gamma-radon-oper} that the $E$-valued Wiener chaoses contain  nuclear series. Given two such nuclear series $I_p(f)$ and $I_q(g)$ of the form
\begin{equation*}
  I_p(f) = \sum_{j=1}^{\infty} I_p(f_j) \otimes x_j \qquad \text{and} \qquad I_{q}(g) = \sum_{k=1}^{\infty} I_{q}(g_k) \otimes y_k
\end{equation*}
with $f_j \in \mathfrak{H}^{\widehat{\odot}_{\sigma} p}$, $g_k \in \mathfrak{H}^{\widehat{\odot}_{\sigma} q}$, $x_j \in E$ and $y_k \in F$ for all $j,k \in \mathbb{N}$, it is readily obtained that for $1 \leq r \leq p \land q - 1$ one has
\begin{align*}
  \norm{f\otimes_{r,\pi}g}_{\gamma^{\left(p+q-2r\right)}\left(\mathfrak{H},E\widehat{\otimes}_{\pi}F\right)}
  \leq
  \sum_{j,k=1}^{\infty} \left\lVert f_{j} \otimes_r g_k \right\rVert_{\mathfrak{H}^{\widehat{\otimes}_{\sigma} p+q-2r}} \left\lVert x_j \right\rVert_E \left\lVert y_k \right\rVert_F,
\end{align*}
where the right hand side is finite as
\begin{equation*}
  \left\lVert f_{j} \otimes_r g_k \right\rVert_{\mathfrak{H}^{\widehat{\otimes}_{\sigma} p+q-2r}}
  \leq
  \left\lVert f_j \right\rVert_{\mathfrak{H}^{\widehat{\otimes}_{\sigma} p}} \left\lVert g_k \right\rVert_{\mathfrak{H}^{\widehat{\otimes}_{\sigma} q}}.
\end{equation*}
In this special case, one can hence reduce the computation of vector valued contractions to the well-studied scalar case. This is in particular true for finite-rank kernels, in which case all above sums are finite.

We conclude this section with a general consequence of the contraction estimates for Gaussian approximation on the space of continuous functions. 

To prepare this example, we first prove that Hölder-continuous paths provide control on the $\varepsilon$-regularization.

\begin{proposition}
  \label{continuitysufficientcondtocontroltightnessofF}
  Let $W$ be an isonormal Gaussian process on a real, separable Hilbert space $\mathfrak{H}$. Fix $p\ge 1$, $T>0$ and let $f\in\gamma_{sym}^{p}\left(\mathfrak{H},\mathcal{C}\left(\left[0,T\right]\right)\right)$ with Hölder continuous range, i.e. there exists $L>0$ and $\beta \in (0,1]$ such that for all $x,y \in [0,T]$ one has
\begin{equation*}
  \left\lVert (e_y - e_x) \circ f \right\rVert_{\mathfrak{H}^{\widehat{\otimes}_{\sigma} p}} \leq L \left| y-x \right|^{\beta},
\end{equation*}
where $e_x$ denotes the point evaluation functional at $x \in [0,T]$. Then, for
all $s > \max \left\{ 2, \frac{1}{\beta} \right\}$, all $\eta \in (0,\beta - \frac{1}{s})$ and all $\varepsilon \in (0,T)$, one has
\begin{equation*}
  \mathbb{E} \left( \left\lVert \left( I_p(f) \right)_{\varepsilon} - I_p(f) \right\rVert_{\infty} \right)
  \leq C \, L \, \varepsilon^{\beta - \frac{1}{s} - \eta},
\end{equation*}
where $\left(I_p\left(f\right)\right)_{\varepsilon}$ denotes the regularization of $I_p\left(f\right)$ and $C=C_{p,\beta,s,\eta,T}$ is a positive constant independent of $f$, $L$ and $\varepsilon$.
\end{proposition}
\begin{remark}
  Note that as $s$ can be chosen arbitrarily large and $\eta>0$
  arbitrarily close to zero, we have that for every $\zeta\in(0,\beta)$, there exists $C_\zeta >0$ such that
  \begin{equation*}
    \mathbb{E} \left(  \norm{\left( I_p(f) \right)_{\varepsilon} - I_p(f)}_{\infty} \right) \leq C_{\zeta} \, L \, \varepsilon^{\beta-\zeta}.
\end{equation*}
\end{remark}
\begin{proof}
By the Wiener--It\^{o} isometry and hypercontractivity on the $p$-th Wiener chaos (see~\cite[Theorem~2.7.2]{nourdin-peccati:2012:normal-approximations-malliavin}), for $s\geq2$ and $x,y\in\left[0,T\right]$ we have
\begin{align*}
\norm{I_p\left(f\right)\left(y\right)-I_p\left(f\right)\left(x\right)}_{L^s\left(\Omega\right)}
&\leq\left(s-1\right)^{p/2}\norm{I_p\left(\left(e_y-e_x\right)\circ f\right)}_{L^2\left(\Omega\right)}\\
&\leq\left(s-1\right)^{p/2}\sqrt{p!}L\abs{y-x}^{\beta}.
\end{align*}
The moment bound ensures that $I_p\left(f\right)$ has a continuous modification. By the definition of the regularization and Jensen's inequality,
\begin{equation*}
\norm{\left(I_p\left(f\right)\right)_{\varepsilon}-I_p\left(f\right)}_{\infty}
\leq\omega\left(I_p\left(f\right),\varepsilon\right),
\end{equation*}
where $\omega\left(g,\varepsilon\right)=\sup_{\abs{y-x}\leq\varepsilon}\abs{g\left(y\right)-g\left(x\right)}$ for $g\in C\left(\left[0,T\right]\right)$.
The Garsia--Rodemich--Rumsey inequality (see~\cite{garsia-rodemich-rumsey:1970:real-variable-lemma}) implies that
\begin{equation*}
\omega\left(g,\varepsilon\right)
\leq C_{\beta,s,\eta}\varepsilon^{\beta-\eta-1/s}
\left(\int_0^T\int_0^T
\frac{\abs{g\left(y\right)-g\left(x\right)}^s}{\abs{y-x}^{1+s\left(\beta-\eta\right)}}\mathrm{d}x\mathrm{d}y\right)^{1/s}.
\end{equation*}
Combining the preceding moment estimate with Jensen's inequality and using
\begin{equation*}
\int_0^T\int_0^T\abs{y-x}^{\eta s-1}\mathrm{d}x\mathrm{d}y
=\frac{2T^{1+\eta s}}{\eta s\left(1+\eta s\right)},
\end{equation*}
we obtain
\begin{equation*}
\E\left(\omega\left(I_p\left(f\right),\varepsilon\right)\right)
\leq C_{\beta,s,\eta}\left(s-1\right)^{p/2}\sqrt{p!}L
\left(\frac{2T^{1+\eta s}}{\eta s\left(1+\eta s\right)}\right)^{1/s}
\varepsilon^{\beta-\eta-1/s},
\end{equation*}
which proves the claim.
\end{proof}
We now combine the preceding regularity estimate with the operator-valued contraction bound to control the bounded-Lipschitz distance on the continuous-path space.
\begingroup
\begin{proposition}
\label{corollaryboundedwassersteinwithcontractionbounds}
Let $E=C\left(\left[0,T\right]\right)$, $p\geq1$, and $f\in\gamma_{sym}^{p}\left(\mathfrak{H},E\right)$. Put $X=I_p\left(f\right)$, and let $Z$ be an independent centered Radon Gaussian random element on $E$. Suppose that, for some $\beta\in\left(0,1\right]$,
\begin{equation*}
\E\left(\left[X\right]_{C^{\beta}}\right)+\E\left(\left[Z\right]_{C^{\beta}}\right)<\infty,
\end{equation*}
where $\left[\cdot\right]_{C^{\beta}}$ is the H\"older seminorm. Then there exists a constant $C>0$, depending on $p$, $T$, $\beta$ and the expected H\"older seminorms, such that
\begin{equation*}
d_{\mathrm{BL}}\left(X,Z\right)
\leq C\left(\norm{R_X-R_Z}_{\nu}^{\alpha}
+\sum_{r=1}^{p-1}\norm{f\widetilde{\otimes}_r f}_{\gamma^{2\left(p-r\right)}\left(\mathfrak{H},\mathfrak{N}\left(E^{\ast},E\right)\right)}^{\alpha}\right),
\end{equation*}
where $\alpha=\beta/\left(3\beta+2\right)$.
\end{proposition}
\begin{proof}
The operator-valued version of the smart-path comparison in Proposition~\ref{smartpathproposition} gives, for every twice Fr\'{e}chet-differentiable test function $h$,
\begin{equation*}
\abs{\E\left(h\left(X\right)\right)-\E\left(h\left(Z\right)\right)}
\leq\frac12\norm{h''}_{\infty}\E\left(\norm{\Gamma\left(X,-L^{-1}X\right)-R_Z}_{\nu}\right),
\end{equation*}
Applying the smoothing result of Theorem~\ref{thm:OnePointTwenty} and optimizing as in Proposition~\ref{MainabstractBoundWasserstein}, using the H\"older assumption to bound both regularization errors by $C\varepsilon^{\beta}$, we obtain
\begin{equation*}
d_{\mathrm{BL}}\left(X,Z\right)\leq C \E\left(\norm{\Gamma\left(X,-L^{-1}X\right)-R_Z}_{\nu}\right)^{\alpha}.
\end{equation*}
Moreover, the triangle inequality, Lemma~\ref{lem:4} and the operator-valued contraction bound in Theorem~\ref{thm:6} give
\begin{align*}
\Delta_{\nu}
&\leq\norm{R_X-R_Z}_{\nu}
+\E\left(\norm{\Gamma\left(X,-L^{-1}X\right)-R_X}_{\nu}\right)\\
&\leq\norm{R_X-R_Z}_{\nu}
+\frac1p\sum_{r=1}^{p-1}a_{p,p,r}\mathfrak{C}_{2p-2r,1}
\norm{f\widetilde{\otimes}_r f}_{\gamma^{2\left(p-r\right)}\left(\mathfrak{H},\mathfrak{N}\left(E^{\ast},E\right)\right)}.
\end{align*}
Since $\alpha\in\left(0,1\right)$, the subadditivity of the map $x\mapsto x^{\alpha}$ yields the assertion.
\end{proof}
\endgroup

\begingroup
\begin{corollary}[Hilbert-space specialization]
\label{contractionboundonprojectivenormHilbertspace}
Let $K$ be a separable Hilbert space, $p\in\N$, and $f\in\mathfrak{H}^{\widehat{\odot}_{\sigma}p}\widehat\otimes_{\sigma}K$. Put $X=I_p\left(f\right)$ and let $Z$ be a centered, independent $K$-valued Radon Gaussian random element. Denote by $\rho_{\infty}$ the smooth probability distance on laws with finite first moment, defined by
\begin{equation*}
\rho_{\infty}\left(X,Z\right)=\sup_{\substack{h\in\mathcal{C}^2\left(K;\R\right)\\
\norm{h'}_{\infty}+\norm{h''}_{\infty}\leq1}}
\abs{\E\left(h\left(X\right)\right)-\E\left(h\left(Z\right)\right)}.
\end{equation*}
Then
\begin{align*}
\rho_{\infty}\left(X,Z\right)
&\leq \frac12\norm{R_X-R_Z}_{\mathcal{S}_1\left(K\right)}
+\frac1{2p}\sum_{r=1}^{p-1}a_{p,p,r}\mathfrak{C}_{2p-2r,1}
\norm{f\widetilde\otimes_r f}_{\gamma^{2p-2r}\left(\mathfrak{H},\mathcal{S}_1\left(K\right)\right)}.
\end{align*}
Here $\mathcal{S}_1\left(K\right)$ is the space of trace class operators, and $a_{p,p,r}$, $\mathfrak{C}_{2p-2r,1}$ are the constants of Theorem~\ref{thm:6}.
\end{corollary}
\begin{proof}
The test functions in the definition of $\rho_{\infty}$ need not be bounded. Their first and second derivatives are bounded, however, so the linear-growth version of Proposition~\ref{smartpathproposition} applies. The smooth comparison of Proposition~\ref{smartpathproposition}, together with the canonical isometric identification $K\widehat\otimes_{\pi}K\simeq\mathcal{S}_1\left(K\right)$, implies
\begin{equation*}
\rho_{\infty}\left(X,Z\right)
\leq \frac12\E\left(\norm{\Gamma\left(X,-L^{-1}X\right)-R_Z}_{\mathcal{S}_1\left(K\right)}\right).
\end{equation*}
The triangle inequality and the contraction estimate of Theorem~\ref{thm:6} give the claimed bound.
\end{proof}
The underlying trace-class comparison is also obtained in~\cite[Theorem~3.1]{duker-zoubouloglou:2025:fourth-moment-theorem-hilbert}. The metrization property of $\rho_{\infty}$ in separable Hilbert spaces is discussed in~\cite{gine-leon:1980:central-limit-theorem,bassetti-bourguin-campese-ea:2025:caveat-metrizing-convergence}.
\endgroup

\begingroup
\section{Quantitative functional Breuer--Major approximation}
\label{applicationSection}

The classical Breuer--Major theorem concerns normalized sums of nonlinear transformations of a stationary Gaussian sequence. See~\cite{breuer-major:1983:central-limit-theorems} and its functional extension in~\cite{nourdin-nualart:2020:functional-breuer-major}. We use the canonical partial-sum process to illustrate the projective-tensor comparison from the preceding sections on a standard stochastic-process model. The purpose of this section is to obtain a quantitative bound in the bounded-Lipschitz metric associated with the uniform norm on the entire path, not to improve the optimal assumptions or rates of Breuer--Major-type results. In particular, the classical functional theorem holds under $\rho\in\ell^p\left(\Z\right)$, whereas our direct projective estimate uses stronger covariance conditions. These are transparent sufficient conditions that permit a concise demonstration of the general method. For quantitative normal approximation at fixed times, see also~\cite{nourdin-peccati-podolskij:2011:quantitative-breuer-major}. Those scalar bounds address a different metric problem.

Let $\left(G_k\right)_{k\in\Z}$ be a centered stationary Gaussian
sequence satisfying $\E\left(G_0^2\right)=1$, and for any $k \in \Z$,
let $\rho\left(k\right)=\E\left(G_0G_k\right)$. Fix an integer $p\geq2$, and let $H_p$ denote the $p$-th Hermite polynomial. Assume that $\rho\in\ell^p\left(\Z\right)$ and that
\begin{equation}
\label{eq:bm-sigma}
\sigma_p^2=p!\sum_{k\in\Z}\rho\left(k\right)^p>0.
\end{equation}
The series defining $\sigma_p^2$ is absolutely convergent. For $1\leq
j\leq n$ and any $t\in\left[0,1\right]$, set
\begin{equation*}
h_{j,n}\left(t\right)=\min\left\{1,\max\left\{0,nt-j+1\right\}\right\},
\end{equation*}
and consider the $\mathcal{C}\left(\left[0,1\right]\right)$-valued random element
\begin{equation}
\label{eq:bm-polygonal}
F_n\left(t\right)=\frac{1}{\sigma_p\sqrt n}\sum_{j=1}^nh_{j,n}\left(t\right)H_p\left(G_j\right).
\end{equation}
Thus $F_n\left(j/n\right)=\left(\sigma_p\sqrt n\right)^{-1}\sum_{k=1}^jH_p\left(G_k\right)$, and $F_n$ is affine between successive grid points. Let $B$ be standard Brownian motion on $\left[0,1\right]$. The covariance tensor of $B$ is denoted by $Q_B$.

\begin{theorem}[Quantitative functional Breuer--Major theorem]
\label{thm:bm-functional}
Suppose, in addition to~\eqref{eq:bm-sigma}, that
\begin{equation}
\label{eq:bm-short-memory}
\sum_{k\in\Z}\abs{\rho\left(k\right)}<\infty\qquad\mbox{and}\qquad
\sum_{k\in\Z}\abs{k}\abs{\rho\left(k\right)}^p<\infty.
\end{equation}
Then, for every $\beta\in\left(0,1/2\right)$, there exists $C_{p,\rho,\beta}>0$ such that
\begin{equation}
\label{eq:bm-short-rate}
d_{\mathrm{BL}}\left(F_n,B\right)
\leq C_{p,\rho,\beta}\left(\frac{\log\left(en\right)}{\sqrt n}\right)^{\beta/\left(3\beta+2\right)},\qquad n\geq2,
\end{equation}
where the bounded-Lipschitz distance is induced by the supremum norm on $\mathcal C\left(\left[0,1\right]\right)$. Consequently, for every $\eta>0$,
\begin{equation*}
d_{\mathrm{BL}}\left(F_n,B\right)=O\left(n^{-1/14+\eta}\right).
\end{equation*}
\end{theorem}

\begin{remark}
\label{rmk:bm-scope}
The conditions in~\eqref{eq:bm-short-memory} are convenient sufficient assumptions, rather than an attempt to characterize the largest covariance class. Absolute summability permits a direct summation of the projective-tensor estimates over covariance lags, while the weighted $\ell^p$ condition controls the deterministic covariance discrepancy. The proof also provides estimates applicable to certain non-absolutely summable covariances, as illustrated by Corollary~\ref{cor:bm-fgn}.

The exponent in~\eqref{eq:bm-short-rate} reflects the use of the full projective norm followed by the regularization of Lipschitz test functions. We make no optimality claim: the purpose of the application is to demonstrate the general Banach-valued comparison theorem on the canonical functional Breuer--Major process without turning the paper into a study of sharp, model-specific rates or minimal covariance assumptions.
\end{remark}

The argument proceeds by a direct projective-tensor estimate for the random covariance discrepancy. Its key deterministic ingredient is a logarithmic bound for tensors formed from nested step functions, which we prove next. This keeps the application within the framework of Sections~\ref{sec:carre-du-champ}--\ref{sec:chaos}.

\begin{lemma}[Nested-step projective estimate]
\label{lem:bm-staircase}
Let $m\geq1$ and, for $1\leq j\leq m$, define
$\tau_j=\left(\1_{\left\{a\geq j\right\}}\right)_{a=0}^m\in\ell_\infty^{m+1}$,
where $\ell_\infty^{m+1}$ denotes the space of real sequences of length $m+1$ with the supremum norm. For arbitrary real numbers $z_1,\ldots,z_m$, we have
\begin{equation}
\label{eq:bm-staircase}
\norm{\sum_{j=1}^mz_j\tau_j\otimes\tau_j}_{\pi}
\leq C\log\left(e\left(m+1\right)\right)
\max_{0\leq a\leq m}\abs{\sum_{j=1}^az_j},
\end{equation}
where $C$ is an absolute constant. Consequently, for increasing sequences of indices $1\leq i_1<\cdots<i_m\leq n$ and $1\leq j_1<\cdots<j_m\leq n$,
\begin{equation}
\label{eq:bm-ramp-staircase}
\norm{\sum_{r=1}^mz_rh_{i_r,n}\otimes h_{j_r,n}}_{\pi}
\leq C\log\left(e\left(n+1\right)\right)
\max_{0\leq a\leq m}\abs{\sum_{r=1}^az_r}.
\end{equation}
\end{lemma}

\begin{proof}
We begin by explaining how the finite-dimensional tensors used below
are represented by matrices. Let $\left(e_a\right)_{a=0}^m$ be the
canonical basis of $\ell_\infty^{m+1}$. Every array of coefficients
$\left(A_{ab}\right)_{0\leq a,b\leq m}$ determines a tensor
\begin{equation*}
A=\sum_{a,b=0}^m A_{ab}e_a\otimes e_b
\in\ell_\infty^{m+1}\widehat\otimes_{\pi}\ell_\infty^{m+1}.
\end{equation*}
The completion causes no difficulty here because the two spaces
are finite-dimensional. We shall identify this tensor with its matrix
of coefficients. This is an identification of vector spaces, not an
assertion that the projective norm is the maximum of the matrix entries.

We first estimate the tensor corresponding to the weak
upper-triangular matrix
\begin{equation*}
V=\sum_{0\leq a\leq b\leq m}e_a\otimes e_b,
\qquad
V_{ab}=\1_{\left\{a\leq b\right\}}.
\end{equation*}
Its projective norm will be controlled by a dyadic decomposition.
This is the only step in which the logarithmic factor appears.

Choose $J\in\N$ such that $N=2^J$ is the smallest power of two
satisfying $N\geq m+1$. We work first with the indices
$0,\ldots,N-1$. For $0\leq l\leq J-1$, partition these indices
into consecutive dyadic intervals $I$ of length $2^{l+1}$.
Write $I_{\mathrm L}$ and $I_{\mathrm R}$ for the left and right
halves of each such interval, and set
\begin{equation*}
T_l=\sum_{\substack{I\text{ dyadic}\\\abs{I}=2^{l+1}}}
\1_{I_{\mathrm L}}\otimes\1_{I_{\mathrm R}}
\in\ell_\infty^N\otimes\ell_\infty^N.
\end{equation*}
For every pair $a<b$, there is a unique smallest dyadic interval
containing both $a$ and $b$. In this interval, $a$ belongs to the
left half and $b$ to the right half. Consequently, the matrix
with entries $\1_{\left\{a<b\right\}}$ is precisely
$\sum_{l=0}^{J-1}T_l$. Thus it is enough to bound the
projective norm of each $T_l$ independently of $l$ and $N$.

Fix a level $l$, and let the variables $\varepsilon_I$, indexed
by its dyadic intervals, be independent Rademacher variables.
Independence and
$\E\left(\varepsilon_I\varepsilon_{I'}\right)
=\1_{\left\{I=I'\right\}}$ imply
\begin{equation*}
T_l=\E\left(\left(\sum_I\varepsilon_I\1_{I_{\mathrm L}}\right)
\otimes\left(\sum_I\varepsilon_I\1_{I_{\mathrm R}}\right)\right).
\end{equation*}
At any fixed level, the sets $I_{\mathrm L}$ are pairwise
disjoint, and the same holds for the sets $I_{\mathrm R}$.
Each of the two random vectors in the preceding expectation
therefore has $\ell_\infty^N$-norm at most one. By the
definition of the projective norm and its convexity,
\begin{equation*}
\norm{T_l}_{\pi}
\leq\E\left(\norm{\sum_I\varepsilon_I\1_{I_{\mathrm L}}}_{\infty}
\norm{\sum_I\varepsilon_I\1_{I_{\mathrm R}}}_{\infty}\right)
\leq1.
\end{equation*}
The diagonal tensor
$D=\sum_{a=0}^{N-1}e_a\otimes e_a$ has projective norm
at most one as well. Indeed, if the $\varepsilon_a$ are
independent Rademacher variables, then
\begin{equation*}
D=\E\left(\left(\sum_{a=0}^{N-1}\varepsilon_ae_a\right)
\otimes\left(\sum_{a=0}^{N-1}\varepsilon_ae_a\right)\right),
\end{equation*}
and both vectors have supremum norm one.
We have consequently proved that the weak upper-triangular
tensor on $\ell_\infty^N$ has projective norm at most $J+1$.
Finally, the coordinate restriction
$\ell_\infty^N\to\ell_\infty^{m+1}$ is a contraction.
By the metric mapping property of the projective tensor norm, we get
\begin{equation}
\label{eq:bm-staircase-triangular-proof}
\norm{V}_{\pi}\leq J+1
\leq C\log\left(e\left(m+1\right)\right).
\end{equation}
The same estimate holds for the transpose of $V$, and
$\norm{V^{\mathsf T}-D}_{\pi}\leq C\log\left(e\left(m+1\right)\right)$,
where $D$ now denotes the diagonal tensor in dimension $m+1$.

We next relate this triangular estimate to the tensor in
\eqref{eq:bm-staircase}. Put $S_0=0$ and
$S_a=\sum_{j=1}^az_j$ for $1\leq a\leq m$, and let
$M=\max_{0\leq a\leq m}\abs{S_a}$. By the definition of
$\tau_j$, the $\left(a,b\right)$ entry of
$\sum_{j=1}^mz_j\tau_j\otimes\tau_j$ equals
\begin{equation*}
\sum_{j=1}^mz_j\1_{\left\{a\geq j\right\}}
\1_{\left\{b\geq j\right\}}
=\sum_{j=1}^{\min\left\{a,b\right\}}z_j
=S_{\min\left\{a,b\right\}}.
\end{equation*}
Let $M_S\colon\ell_\infty^{m+1}\to\ell_\infty^{m+1}$
be the diagonal multiplication operator
$\left(M_Su\right)\left(a\right)=S_au\left(a\right)$.
Its operator norm is $\norm{M_S}=M$.
The matrix just computed splits into the two tensors
\begin{equation*}
\left(M_S\otimes\operatorname{Id}\right)V
+\left(\operatorname{Id}\otimes M_S\right)
\left(V^{\mathsf T}-D\right).
\end{equation*}
Indeed, the first term has entries
$S_a\1_{\left\{a\leq b\right\}}$, whereas the second
has entries $S_b\1_{\left\{a>b\right\}}$.
The metric mapping property of the projective tensor norm,
together with~\eqref{eq:bm-staircase-triangular-proof}, therefore
yields
\begin{align*}
\norm{\sum_{j=1}^mz_j\tau_j\otimes\tau_j}_{\pi}
\leq \norm{M_S}\left(\norm{V}_{\pi}
+\norm{V^{\mathsf T}-D}_{\pi}\right)\leq C\log\left(e\left(m+1\right)\right)M,
\end{align*}
which is~\eqref{eq:bm-staircase}.

It remains to pass from the step functions $\tau_r$ to the
polygonal functions $h_{i_r,n}$. Define
$U\colon\ell_\infty^{m+1}\to\mathcal C\left(\left[0,1\right]\right)$
by
\begin{equation*}
Uu=u_0\1+\sum_{r=1}^m
\left(u_r-u_{r-1}\right)h_{i_r,n}.
\end{equation*}
The ordering $i_1<\cdots<i_m$ implies
$1\geq h_{i_1,n}\left(t\right)\geq\cdots
\geq h_{i_m,n}\left(t\right)\geq0$ for every
$t\in\left[0,1\right]$. Rearranging the definition of $U$
shows explicitly that
\begin{align*}
\left(Uu\right)\left(t\right)
=\left(1-h_{i_1,n}\left(t\right)\right)u_0+\sum_{r=1}^{m-1}
\left(h_{i_r,n}\left(t\right)-h_{i_{r+1},n}\left(t\right)\right)u_r
+h_{i_m,n}\left(t\right)u_m.
\end{align*}
All coefficients on the right are nonnegative and their sum
is one. Hence $\abs{\left(Uu\right)\left(t\right)}
\leq\norm{u}_{\infty}$ for every $t$, which proves
$\norm{U}\leq1$. Moreover, since
$\tau_r\left(a\right)=\1_{\left\{a\geq r\right\}}$,
substitution into the definition gives $U\tau_r=h_{i_r,n}$.
Construct analogously a contraction $W$ using the increasing
indices $j_1,\ldots,j_m$. It satisfies
$W\tau_r=h_{j_r,n}$. Therefore
\begin{equation*}
\sum_{r=1}^mz_rh_{i_r,n}\otimes h_{j_r,n}
=\left(U\otimes W\right)
\left(\sum_{r=1}^mz_r\tau_r\otimes\tau_r\right).
\end{equation*}
Applying the metric mapping property of the projective tensor norm once more and using
$m\leq n$ proves~\eqref{eq:bm-ramp-staircase}.
\end{proof}

We next establish a maximal inequality for products of Hermite
polynomials, with a constant independent of the lag $k$.
The proof also shows how the assumption
$\rho\in\ell^2(\mathbb Z)$ is used.

\begin{lemma}[Uniform lag maximal inequality]
\label{lem:bm-maximal}
Suppose that $\rho\in\ell^2\left(\Z\right)$. For $k\in\Z$, put
\begin{equation*}
\xi_{i,k}=H_{p-1}\left(G_i\right)H_{p-1}\left(G_{i-k}\right)
-\left(p-1\right)!\rho\left(k\right)^{p-1}.
\end{equation*}
There exists $C_{p,\rho}<\infty$ such that, for all $k\in\Z$, all integers $a$, and all $m\geq1$,
\begin{equation}
\label{eq:bm-maximal}
\E\left(\max_{0\leq r\leq m}\abs{\sum_{i=a+1}^{a+r}\xi_{i,k}}\right)
\leq C_{p,\rho}\sqrt m.
\end{equation}
\end{lemma}

\begin{proof}
Fix $k\in\Z$ and write $q=p-1\geq1$. We first estimate the
covariance of the stationary centered sequence
$\left(\xi_{i,k}\right)_{i\in\Z}$, uniformly in $k$.
This will allow us to control the sums of $\xi_{i,k}$ over
arbitrary intervals and subsequently their maximum.

Recall that, for any two jointly standard Gaussian random
variables $X$ and $Y$, we have
\begin{equation*}
\E\left(H_q\left(X\right)H_q\left(Y\right)\right)
=q!\E\left(XY\right)^q.
\end{equation*}
Consequently, for $i,j\in\Z$,
\begin{equation*}
\begin{aligned}
\E\left(\xi_{i,k}\xi_{j,k}\right)
=\E\left[
H_q\left(G_i\right)H_q\left(G_{i-k}\right)
H_q\left(G_j\right)H_q\left(G_{j-k}\right)
\right]-\left(q!\rho\left(k\right)^q\right)^2.
\end{aligned}
\end{equation*}

We estimate the expectation on the right-hand side using
the diagram formula for products of Hermite polynomials
(see~\cite[Chapter~7]{nourdin-peccati:2012:normal-approximations-malliavin}).
To apply this formula, it is useful to distinguish the
covariances within each of the two pairs
$\left(G_i,G_{i-k}\right)$ and
$\left(G_j,G_{j-k}\right)$ from the covariances
between these pairs.

Set $h=j-i$. By stationarity and the symmetry of $\rho$,
the covariances within the two pairs are
\begin{equation*}
\E\left(G_iG_{i-k}\right)
=\E\left(G_jG_{j-k}\right)
=\rho\left(k\right),
\end{equation*}
whereas the four covariances between the pairs are
\begin{equation}
\label{eq:bm-cross-covariances}
\begin{aligned}
&\E\left(G_iG_j\right)=\rho\left(h\right),
&\qquad &\E\left(G_iG_{j-k}\right)=\rho\left(h-k\right),\\
&\E\left(G_{i-k}G_j\right)=\rho\left(h+k\right),
&\qquad &\E\left(G_{i-k}G_{j-k}\right)=\rho\left(h\right).
\end{aligned}
\end{equation}

Let us recall how the diagram formula applies in this
situation. Each of the four Hermite factors $H_q\left(G_i\right)$,
$H_q\left(G_{i-k}\right)$, $H_q\left(G_j\right)$ and
$H_q\left(G_{j-k}\right)$ is represented by $q$ distinguishable slots.
A diagram is obtained by pairing all $4q$ slots,
with the restriction that two slots belonging to the
same Hermite factor cannot be paired.

Every pair of slots gives an edge, whose weight is the
covariance between the corresponding Gaussian variables.
Since there are $2q$ edges in each diagram, its weight is
the product of $2q$ covariance factors. The diagram
formula expresses the expectation of the product of the
four Hermite polynomials as the sum of these weights over
all admissible diagrams. The combinatorial coefficients
are accounted for by the different possible pairings
of the distinguishable slots.

Before treating general $q$, consider the case $q=1$.
There is then only one slot per Hermite factor, and
exactly three possible pairings of the four Gaussian
variables. They are represented below, where the dotted
line separates $\left(G_i,G_{i-k}\right)$ from
$\left(G_j,G_{j-k}\right)$.

\begin{center}
\begin{tikzpicture}[x=0.70cm,y=0.78cm]
\begin{scope}[xshift=0cm]
\node at (2.55,1.55)
{\scriptsize no cross-group edges};
\foreach \x/\v in {
  0/i,
  1.55/{i-k},
  3.55/j,
  5.1/{j-k}
}{
  \fill (\x,0) circle (1.7pt)
  node[below=4pt] {$G_{\v}$};
}
\draw[densely dotted] (2.55,-0.35)--(2.55,1.05);
\draw (0,0)
  to[out=90,in=90,looseness=1.35] (1.55,0);
\draw (3.55,0)
  to[out=90,in=90,looseness=1.35] (5.1,0);
\node at (2.55,-1.4)
{$\rho\left(k\right)^2$};
\end{scope}

\begin{scope}[xshift=4.5cm]
\node at (2.55,1.55)
{\scriptsize two cross-group edges};
\foreach \x/\v in {
  0/i,
  1.55/{i-k},
  3.55/j,
  5.1/{j-k}
}{
  \fill (\x,0) circle (1.7pt)
  node[below=4pt] {$G_{\v}$};
}
\draw[densely dotted] (2.55,-0.35)--(2.55,1.05);
\draw (0,0)
  to[out=90,in=90,looseness=0.8] (3.55,0);
\draw (1.55,0)
  to[out=90,in=90,looseness=0.8] (5.1,0);
\node at (2.55,-1.4)
{$\rho\left(h\right)^2$};
\end{scope}

\begin{scope}[xshift=9cm]
\node at (2.55,1.55)
{\scriptsize two cross-group edges};
\foreach \x/\v in {
  0/i,
  1.55/{i-k},
  3.55/j,
  5.1/{j-k}
}{
  \fill (\x,0) circle (1.7pt)
  node[below=4pt] {$G_{\v}$};
}
\draw[densely dotted] (2.55,-0.35)--(2.55,1.05);
\draw (0,0)
  to[out=90,in=90,looseness=0.9] (5.1,0);
\draw (1.55,0)
  to[out=90,in=90,looseness=0.9] (3.55,0);
\node at (2.55,-1.4)
{$\rho\left(h-k\right)\rho\left(h+k\right)$};
\end{scope}
\end{tikzpicture}
\end{center}
The first pairing connects $G_i$ with $G_{i-k}$
and $G_j$ with $G_{j-k}$. Its weight is therefore
$\rho\left(k\right)^2$. The other two pairings connect
the two pairs of Gaussian variables and have weights
$\rho\left(h\right)^2$ and
$\rho\left(h-k\right)\rho\left(h+k\right)$,
respectively. We thus recover Wick's formula
\begin{equation*}
\begin{aligned}
\E\left(G_iG_{i-k}G_jG_{j-k}\right)
={}&\rho\left(k\right)^2
+\rho\left(h\right)^2+\rho\left(h-k\right)\rho\left(h+k\right).
\end{aligned}
\end{equation*}
Since $\xi_{i,k}$ and $\xi_{j,k}$ are centered,
the contribution of the first pairing disappears
from their covariance.

The same cancellation occurs for general $q\geq1$.
Indeed, consider a diagram in which no edge connects
a slot associated with $\left(G_i,G_{i-k}\right)$
to a slot associated with $\left(G_j,G_{j-k}\right)$.
Since slots belonging to the same Hermite factor
cannot be paired, every slot associated with $G_i$
must be paired with a slot associated with $G_{i-k}$.
Likewise, every slot associated with $G_j$ must be
paired with a slot associated with $G_{j-k}$.

There are $q!$ ways of pairing the $q$ slots of
$G_i$ with those of $G_{i-k}$, and independently
$q!$ ways of pairing the slots of $G_j$ with those
of $G_{j-k}$. Consequently, there are
$\left(q!\right)^2$ diagrams of this type.
Each contains $2q$ edges of weight $\rho\left(k\right)$,
so their total contribution is
\begin{equation*}
\left(q!\right)^2\rho\left(k\right)^{2q}
=\left(q!\rho\left(k\right)^q\right)^2.
\end{equation*}
This is precisely the term subtracted in the expression
for $\E\left(\xi_{i,k}\xi_{j,k}\right)$.
We are therefore left with the diagrams containing
at least one edge between the two pairs
$\left(G_i,G_{i-k}\right)$ and
$\left(G_j,G_{j-k}\right)$.

We claim that each such diagram actually contains
at least two edges between these pairs.
To see this, let $r$ be the number of these edges
in a given diagram. The first pair of Hermite factors
contains $2q$ slots, of which exactly $r$ are
connected to slots belonging to the second pair.
The remaining $2q-r$ slots must be paired among
themselves. Since each edge uses two slots, this
number must be even. Thus $2q-r=2s$ for some integer $s\geq0$. As $2q$ is even,
$r$ must also be even. Since we only consider
diagrams with $r>0$, we necessarily have $r\geq2$.

We can now bound the contribution of each remaining
diagram. Set
\begin{equation*}
R_{h,k}
=\abs{\rho\left(h\right)}
+\abs{\rho\left(h+k\right)}
+\abs{\rho\left(h-k\right)}.
\end{equation*}
By~\eqref{eq:bm-cross-covariances}, every edge
connecting the two pairs has a weight whose absolute
value is bounded by $R_{h,k}$. Since each remaining
diagram has at least two such edges, we can select
two of them and bound the product of their weights
by $R_{h,k}^2$.

All the other covariance factors have absolute value
at most one, because the Gaussian variables are
standardized and therefore
$\abs{\rho\left(u\right)}\leq1$ for every $u\in\Z$.
It follows that the absolute value of the weight
of each remaining diagram is at most $R_{h,k}^2$.

Finally, the number of admissible diagrams depends
only on $q$. More precisely, it is bounded by the
number $\left(4q-1\right)!!$ of all pairings of
$4q$ distinguishable slots. Summing their weights
therefore yields
\begin{equation}
\label{eq:bm-maximal-covariance-proof}
\abs{\E\left(\xi_{i,k}\xi_{j,k}\right)}
\leq C_p
\left(
\abs{\rho\left(h\right)}
+\abs{\rho\left(h+k\right)}
+\abs{\rho\left(h-k\right)}
\right)^2.
\end{equation}
Notice that $C_p$ depends only on $p$ and is
independent of $i$, $j$ and $k$.

Now, using the elementary inequality
$\left(x+y+z\right)^2
\leq3\left(x^2+y^2+z^2\right)$,
we can write
\begin{align*}
\sum_{h\in\Z}
\abs{\E\left(\xi_{0,k}\xi_{h,k}\right)}
\leq
3C_p\sum_{h\in\Z}
\left(
\abs{\rho\left(h\right)}^2
+\abs{\rho\left(h+k\right)}^2
+\abs{\rho\left(h-k\right)}^2
\right)=9C_p\sum_{h\in\Z}
\abs{\rho\left(h\right)}^2.
\end{align*}
In the last equality, we used the fact that summing
over all integers is invariant under a shift of the
index. In particular, the resulting bound is uniform
in $k$ as we have assumed $\rho\in\ell^2\left(\Z\right)$.

Let $I$ be an arbitrary finite interval of integers.
Expanding the square and using stationarity, we get
\begin{align*}
\E\left(
\abs{\sum_{i\in I}\xi_{i,k}}^2
\right)
=\sum_{i,j\in I}
\E\left(\xi_{i,k}\xi_{j,k}\right)\leq
\sum_{i\in I}\sum_{h\in\Z}
\abs{\E\left(\xi_{0,k}\xi_{h,k}\right)}\leq C_{p,\rho}\abs{I}.
\end{align*}
Thus, the $L^2$-norm of the sum over any interval
$I$ is bounded by $C_{p,\rho}\sqrt{\abs{I}}$,
with a constant independent of both $k$ and
the position of $I$.

We next obtain the corresponding $L^4$ estimate.
Since $\xi_{i,k}$ is a centered polynomial of
degree at most $2q$ in Gaussian variables,
the sum $\sum_{i\in I}\xi_{i,k}$ belongs to the
direct sum of the Wiener chaoses of orders
$1,\ldots,2q$.
By hypercontractivity
(see~\cite[Theorem~2.7.2]{nourdin-peccati:2012:normal-approximations-malliavin}),
together with the triangle and Cauchy--Schwarz
inequalities, its $L^4$-norm is controlled by
its $L^2$-norm, with a constant depending only
on the maximal chaos order $2q$.
Consequently,
\begin{equation}
\label{eq:bm-interval-l4-proof}
\norm{\sum_{i\in I}\xi_{i,k}}_{L^4\left(\Omega\right)}
\leq C_p
\norm{\sum_{i\in I}\xi_{i,k}}_{L^2\left(\Omega\right)}
\leq C_{p,\rho}\sqrt{\abs{I}}.
\end{equation}

It remains to estimate the maximum of the partial
sums. We do this by decomposing each partial sum
into sums over dyadic intervals.

Let $m\geq1$ and let $N=2^J$ be the smallest power
of two such that $N\geq m$. Then $m\leq N<2m$.
For each $0\leq l\leq J$, partition the interval
$\left\{a+1,\ldots,a+N\right\}$ into consecutive
blocks of length $2^l$, starting from its left
endpoint. Denote the resulting collection of
$N/2^l$ blocks by $\mathcal B_l$.

Every interval
$\left\{a+1,\ldots,a+r\right\}$, with
$0\leq r\leq m$, can be expressed as a disjoint
union of at most one block from each $\mathcal B_l$.
Indeed, writing $r$ in base two gives a decomposition
into consecutive blocks whose lengths are distinct
powers of two. These blocks belong to the
corresponding dyadic partitions. It follows that
\begin{equation*}
\max_{0\leq r\leq m}
\abs{\sum_{i=a+1}^{a+r}\xi_{i,k}}
\leq
\sum_{l=0}^J
\max_{I\in\mathcal B_l}
\abs{\sum_{i\in I}\xi_{i,k}}.
\end{equation*}

For each fixed level $l$, we use
\begin{equation*}
\norm{
\max_{I\in\mathcal B_l}\abs{Y_I}
}_{L^4\left(\Omega\right)}
\leq
\left(
\sum_{I\in\mathcal B_l}
\norm{Y_I}_{L^4\left(\Omega\right)}^4
\right)^{1/4}.
\end{equation*}
Since $\mathcal B_l$ contains $N/2^l$ blocks,
each of length $2^l$, Equation~\eqref{eq:bm-interval-l4-proof}
implies
\begin{equation*}
\norm{
\max_{I\in\mathcal B_l}
\abs{\sum_{i\in I}\xi_{i,k}}
}_{L^4\left(\Omega\right)}
\leq
\left(
\sum_{I\in\mathcal B_l}
C_{p,\rho}^4\abs{I}^2
\right)^{1/4}=
C_{p,\rho}
\left(\frac{N}{2^l}\right)^{1/4}
2^{l/2}=
C_{p,\rho}N^{1/4}2^{l/4}.
\end{equation*}
Applying the triangle inequality in $L^4$ and
summing over the dyadic levels, we obtain
\begin{align*}
\norm{
\max_{0\leq r\leq m}
\abs{\sum_{i=a+1}^{a+r}\xi_{i,k}}
}_{L^4\left(\Omega\right)}\leq
C_{p,\rho}N^{1/4}
\sum_{l=0}^J2^{l/4}\leq
C_{p,\rho}N^{1/4}2^{J/4}=
C_{p,\rho}\sqrt N
\leq C_{p,\rho}\sqrt m.
\end{align*}
Here we used the fact that the geometric sum is
bounded by a constant multiple of $2^{J/4}$,
together with $N=2^J<2m$.

Finally, the $L^1$-norm is bounded by the
$L^4$-norm on a probability space. We conclude that
\begin{equation*}
\E\left[
\max_{0\leq r\leq m}
\abs{\sum_{i=a+1}^{a+r}\xi_{i,k}}
\right]
\leq C_{p,\rho}\sqrt m,
\end{equation*}
which proves~\eqref{eq:bm-maximal}.
\end{proof}
We can now turn to the proof of the main result of this section,
namely Theorem~\ref{thm:bm-functional}
\begin{proof}[Proof of Theorem~\ref{thm:bm-functional}]
Since $\rho\in\ell^1\left(\Z\right)$, the hypotheses of Lemma~\ref{lem:bm-maximal} hold. For $n\geq2$, put
\begin{align}
\label{eq:bm-rate-quantity}
\mathcal{R}_{p,n}=\frac{\log\left(en\right)}{n}\sum_{\abs{k}<n}\abs{\rho\left(k\right)}\sqrt{n-\abs{k}}+\sum_{k\in\Z}\abs{\rho\left(k\right)}^p\min\left\{1,\frac{\abs{k}\log\left(en\right)}n\right\}
+\frac{\log\left(en\right)}n.
\end{align}

Represent the sequence as $G_j=W\left(e_j\right)$ for vectors
$e_j\in\mathfrak{H}$ satisfying $\left\langle
  e_i,e_j\right\rangle_{\mathfrak{H}}=\rho\left(i-j\right)$. By~\eqref{eq:bm-polygonal}, $F_n=I_p\left(f_n\right)$ for the finite-rank symmetric kernel
\begin{equation*}
f_n=\frac1{\sigma_p\sqrt n}\sum_{j=1}^ne_j^{\otimes p}\otimes h_{j,n}.
\end{equation*}
In particular, $F_n\in\mathbb{D}^{1,2}\left(\mathcal{C}\left(\left[0,1\right]\right)\right)$, and $-L^{-1}F_n=F_n/p$. Hence
\begin{equation}
\label{eq:bm-gamma}
\Gamma_{\pi}\left(F_n,-L^{-1}F_n\right)
=\frac p{\sigma_p^2n}\sum_{i,j=1}^n\rho\left(i-j\right)
H_{p-1}\left(G_i\right)H_{p-1}\left(G_j\right)h_{i,n}\otimes h_{j,n}.
\end{equation}
The expectation of~\eqref{eq:bm-gamma} equals $Q_{F_n}$. Regrouping the centered expression by the lag $k=i-j$ and applying Lemmas~\ref{lem:bm-staircase} and~\ref{lem:bm-maximal}, we obtain
\begin{equation}
\label{eq:bm-random-discrepancy}
\E\left(\norm{\Gamma_{\pi}\left(F_n,-L^{-1}F_n\right)-Q_{F_n}}_{\pi}\right)
\leq\frac{C_{p,\rho}\log\left(en\right)}n
\sum_{\abs{k}<n}\abs{\rho\left(k\right)}\sqrt{n-\abs{k}}.
\end{equation}

We next bound the difference between the covariance tensors. Put
\begin{equation*}
U_{n,k}=\sum_{\substack{1\leq i\leq n\\1\leq i-k\leq n}}
h_{i,n}\otimes h_{i-k,n},
\end{equation*}
with $U_{n,k}=0$ if $\abs{k}\geq n$. Since $\E\left(H_p\left(G_i\right)H_p\left(G_j\right)\right)=p!\rho\left(i-j\right)^p$, we have
\begin{equation*}
Q_{F_n}=\frac{p!}{\sigma_p^2n}\sum_{k\in\Z}\rho\left(k\right)^pU_{n,k}\qquad\mbox{and}\qquad
\frac{p!}{\sigma_p^2}\sum_{k\in\Z}\rho\left(k\right)^p=1.
\end{equation*}
We claim that
\begin{equation}
\label{eq:bm-shift-bound}
\norm{U_{n,k}-U_{n,0}}_{\pi}
\leq Cn\min\left\{1,\frac{\abs{k}\log\left(en\right)}n\right\}.
\end{equation}
Both tensors are obtained by polygonally interpolating their grid matrices, and interpolation is a contraction from $\ell_\infty^{n+1}$ to $\mathcal{C}\left(\left[0,1\right]\right)$. For $0\leq k<n$ and grid points $a/n,b/n$, the matrices corresponding to $U_{n,k}$ and $U_{n,0}$ are
\begin{equation*}
K_k\left(a,b\right)=\min\left\{\max\left\{a-k,0\right\},b\right\}\qquad\mbox{and}\qquad
K_0\left(a,b\right)=\min\left\{a,b\right\}.
\end{equation*}
Their difference is
\begin{equation*}
K_k\left(a,b\right)-K_0\left(a,b\right)
=-\sum_{r=1}^k\1_{\left\{a\geq r\right\}}\1_{\left\{a-r+1\leq b\right\}}.
\end{equation*}
For $1\leq r\leq k$, define $S_r\colon\ell_\infty^{n+1}\to\ell_\infty^{n+1}$ by $\left(S_ru\right)\left(a\right)=\1_{\left\{a\geq r\right\}}u\left(a-r+1\right)$. Then $\norm{S_r}\leq1$, and the $r$-th matrix in the preceding sum is obtained by applying $S_r$ to the first tensor factor of the weak upper-triangular matrix $V_{uv}=\1_{\left\{u\leq v\right\}}$. The metric mapping property of the projective tensor norm and the dyadic bound for $V$ in the proof of Lemma~\ref{lem:bm-staircase} therefore give a projective norm bound of $C\log\left(en\right)$ for each summand. The case $k<0$ follows by transposition, yielding the bound $C\abs{k}\log\left(en\right)$ for $\abs{k}<n$.

To obtain the alternative bound $Cn$, write $U_{n,k}$ as
\begin{equation*}
U_{n,k}=\E_{\varepsilon}\left(\left(\sum_{i\in I_k}\varepsilon_i h_{i,n}\right)\otimes
\left(\sum_{i\in I_k}\varepsilon_i h_{i-k,n}\right)\right),
\end{equation*}
where $I_k=\left\{i:1\leq i\leq n,\ 1\leq i-k\leq n\right\}$ and the $\varepsilon_i$ are independent Rademacher variables. Since the functions $h_{i_j,n}$ are pointwise ordered, each supremum norm in the integrand is bounded by the maximum of the associated Rademacher partial sums. Doob's inequality and the Cauchy--Schwarz inequality give $\norm{U_{n,k}}_{\pi}\leq4\abs{I_k}\leq4n$. The same estimate holds for $U_{n,0}$, proving~\eqref{eq:bm-shift-bound}, including $\abs{k}\geq n$.

Let $B_n$ be the polygonal interpolation of Brownian motion
through the same $n+1$ grid points. The process $R_n=B-B_n$
is independent of $B_n$, and its restrictions to the $n$
subintervals of the partition are independent rescaled
Brownian bridges. Consequently,
\begin{equation*}
Q_{B_n}=\frac1nU_{n,0}
\qquad\mbox{and}\qquad
\norm{Q_B-Q_{B_n}}_{\pi}
\leq\E\left(\norm{R_n}_{\infty}^2\right).
\end{equation*}
To estimate the latter quantity, we use the classical
Kolmogorov formula for the supremum of a standard
Brownian bridge $\beta$
(see~\cite[Section~9, Equation~(9.40)]{billingsley:1999:convergence-probability-measures}),
which implies, for any $u>0$,
\begin{equation*}
\P\left(
\sup_{0\leq t\leq1}\abs{\beta\left(t\right)}>u
\right)
\leq 2e^{-2u^2}.
\end{equation*}
Since the restriction of $R_n$ to each subinterval
$\left[j/n,\left(j+1\right)/n\right]$ has the same
distribution as a standard Brownian bridge multiplied
by $n^{-1/2}$, we have
\begin{equation*}
\P\left(\norm{R_n}_{\infty}>u\right)
\leq\sum_{j=0}^{n-1}
\P\left(
\sup_{j/n\leq t\leq\left(j+1\right)/n}
\abs{R_n\left(t\right)}>u
\right)\leq 2ne^{-2nu^2}.
\end{equation*}
Integrating this tail estimate yields
\begin{equation*}
\E\left(\norm{R_n}_{\infty}^2\right)
\leq\frac{\log\left(2n\right)+1}{2n}
\leq C\frac{\log\left(en\right)}n.
\end{equation*}
Using~\eqref{eq:bm-shift-bound} and the absolute convergence
of the covariance series, we conclude that
\begin{equation}
\label{eq:bm-cov-discrepancy}
\norm{Q_{F_n}-Q_B}_{\pi}
\leq C_{p,\rho}\left(
\sum_{k\in\Z}\abs{\rho\left(k\right)}^p
\min\left\{1,\frac{\abs{k}\log\left(en\right)}n\right\}
+\frac{\log\left(en\right)}n
\right).
\end{equation}
Combining~\eqref{eq:bm-random-discrepancy}
and~\eqref{eq:bm-cov-discrepancy}, we finally obtain
\begin{equation}
\label{eq:bm-total-discrepancy}
\E\left(
\norm{\Gamma_{\pi}\left(F_n,-L^{-1}F_n\right)-Q_B}_{\pi}
\right)
\leq C_{p,\rho}\mathcal{R}_{p,n}.
\end{equation}
It remains to verify the regularity hypothesis of Proposition~\ref{MainabstractBoundWasserstein}. If $0\leq s\leq t\leq1$, the coefficients $a_j=h_{j,n}\left(t\right)-h_{j,n}\left(s\right)$ satisfy $0\leq a_j\leq1$ and
\begin{equation*}
\sum_{j=1}^na_j^2\leq\sum_{j=1}^na_j=n\left(t-s\right).
\end{equation*}
Young's convolution inequality and $\sum_k\abs{\rho\left(k\right)}^p<\infty$ imply
\begin{equation*}
\E\left(\abs{F_n\left(t\right)-F_n\left(s\right)}^2\right)
\leq C_{p,\rho}\abs{t-s}.
\end{equation*}
By fixed-chaos hypercontractivity~\cite[Theorem~2.7.2]{nourdin-peccati:2012:normal-approximations-malliavin} and the Garsia--Rodemich--Rumsey inequality~\cite{garsia-rodemich-rumsey:1970:real-variable-lemma}, for every $\beta<1/2$,
\begin{equation*}
\sup_{n\geq2}\E\left(\left[F_n\right]_{C^{\beta}}\right)<\infty
\qquad\mbox{and}\qquad \E\left(\left[B\right]_{C^{\beta}}\right)<\infty.
\end{equation*}
Consequently, the regularization errors from Section~\ref{sec:Skorokhod} are bounded by $C_{p,\rho,\beta}\varepsilon^{\beta}$, uniformly in $n$. Under~\eqref{eq:bm-short-memory}, the first term in~\eqref{eq:bm-rate-quantity} satisfies
\begin{equation*}
\frac{\log\left(en\right)}{n}
\sum_{\abs{k}<n}\abs{\rho\left(k\right)}\sqrt{n-\abs{k}}
\leq\frac{\log\left(en\right)}{\sqrt n}
\sum_{k\in\Z}\abs{\rho\left(k\right)}.
\end{equation*}
The second term is bounded by
\begin{equation*}
\frac{\log\left(en\right)}n
\sum_{k\in\Z}\abs{k}\abs{\rho\left(k\right)}^p.
\end{equation*}
Thus $\mathcal R_{p,n}\leq C_{p,\rho}\log\left(en\right)/\sqrt
  n$. Proposition~\ref{MainabstractBoundWasserstein}, applied
  to~\eqref{eq:bm-total-discrepancy}, yields~\eqref{eq:bm-short-rate}
  as $\beta/\left(3\beta+2\right)\uparrow1/7$ whenever
  $\beta\uparrow1/2$, and the logarithmic factor can be absorbed into
  $n^\eta$ for every $\eta>0$. This concludes the proof.
\end{proof}

\begin{corollary}[Fractional Gaussian noise]
\label{cor:bm-fgn}
Let $G_k=B^H_{k+1}-B^H_k$ be fractional Gaussian noise with Hurst index $H\in\left(0,3/4\right)$, and assume~\eqref{eq:bm-sigma}. Then, for every $\eta>0$,
\begin{equation*}
d_{\mathrm{BL}}\left(F_n,B\right)=
\begin{cases}
O\left(n^{-1/14+\eta}\right),&0<H\leq1/2,\\
O\left(n^{-\left(3/2-2H\right)/7+\eta}\right),&1/2<H<3/4.
\end{cases}
\end{equation*}
\end{corollary}

\begin{proof}
For $H\neq1/2$, the covariance function satisfies
\begin{equation*}
\rho_H\left(k\right)=\frac12\left(\abs{k+1}^{2H}+\abs{k-1}^{2H}-2\abs{k}^{2H}\right)
=O\left(\left(1+\abs{k}\right)^{2H-2}\right).
\end{equation*}
For $H<1/2$, both series in~\eqref{eq:bm-short-memory} converge. For $H=1/2$, the variables are independent. For $1/2<H<3/4$, put $a=2-2H\in\left(1/2,1\right)$. The first term in~\eqref{eq:bm-rate-quantity} is $O\left(\log\left(en\right)n^{1/2-a}\right)$. Put $q=pa>1$ and $N=n/\log\left(en\right)$. Splitting the covariance-tail sum in~\eqref{eq:bm-rate-quantity} at $N$, we obtain
\begin{equation*}
\sum_{k\in\Z}\abs{\rho_H\left(k\right)}^p\min\left\{1,\frac{\abs{k}\log\left(en\right)}n\right\}
\leq C_{p,H}
\begin{cases}
N^{1-q},&1<q<2,\\
\log\left(en\right)^2/n,&q=2,\\
\log\left(en\right)/n,&q>2.
\end{cases}
\end{equation*}
Since $q\geq2a$ and $a>1/2$, each of these estimates is $O\left(\log\left(en\right)n^{1/2-a}\right)$. The last term in~\eqref{eq:bm-rate-quantity} satisfies the same bound. In this range, $\rho_H\in\ell^2\left(\Z\right)$. The estimates~\eqref{eq:bm-random-discrepancy}--\eqref{eq:bm-total-discrepancy} and the uniform H\"older bound derived in the proof of Theorem~\ref{thm:bm-functional} therefore apply to this covariance as well, since their derivations require only square summability. Thus Proposition~\ref{MainabstractBoundWasserstein} yields
\begin{equation*}
d_{\mathrm{BL}}\left(F_n,B\right)
\leq C_{p,H,\beta}\mathcal R_{p,n}^{\beta/\left(3\beta+2\right)}.
\end{equation*}
Taking $\beta$ arbitrarily close to $1/2$ gives the asserted rate.
\end{proof}

The original càdlàg partial-sum process also satisfies a bound in the Skorokhod setting of Section~\ref{sec:Skorokhod}.

\begin{corollary}[Skorokhod partial-sum process]
\label{cor:bm-skoro}
Under the assumptions of Theorem~\ref{thm:bm-functional}, define
\begin{equation*}
S_n\left(t\right)=\frac1{\sigma_p\sqrt n}\sum_{j=1}^{\lfloor nt\rfloor}H_p\left(G_j\right),
\qquad t\in\left[0,1\right].
\end{equation*}
Then, for every $\beta\in\left(0,1/2\right)$, with the bounded-Lipschitz metric on $\mathcal{D}\left(\left[0,1\right]\right)$ induced by the uniform norm,
\begin{equation*}
d_{\mathrm{BL},\mathcal D}\left(S_n,B\right)
\leq C_{p,\rho,\beta}\left(\frac{\log\left(en\right)}{\sqrt n}\right)^{\beta/\left(3\beta+2\right)}
+C_{p,\rho}\frac{\left(\log\left(en\right)\right)^{p/2}}{\sqrt n}.
\end{equation*}
Consequently, $d_{\mathrm{BL},\mathcal D}\left(S_n,B\right)=O\left(n^{-1/14+\eta}\right)$ for every $\eta>0$.
\end{corollary}

\begin{proof}
By polygonal interpolation,
\begin{equation*}
\norm{S_n-F_n}_{\infty}\leq\frac1{\sigma_p\sqrt n}\max_{1\leq j\leq n}\abs{H_p\left(G_j\right)}.
\end{equation*}
Hypercontractivity gives $\norm{H_p\left(G_j\right)}_{L^q\left(\Omega\right)}\leq\sqrt{p!}\left(q-1\right)^{p/2}$ for $q\geq2$, uniformly in $j$. Taking $q=\max\left\{2,\lceil\log\left(en\right)\rceil\right\}$ and using the $L^q$ bound for a maximum yields
\begin{equation*}
\E\left(\norm{S_n-F_n}_{\infty}\right)
\leq C_p\frac{\left(\log\left(en\right)\right)^{p/2}}{\sigma_p\sqrt n}.
\end{equation*}
The restriction of a bounded 1-Lipschitz function on the Skorokhod space to $\mathcal C\left(\left[0,1\right]\right)$ is again bounded and 1-Lipschitz. Hence the triangle inequality and Theorem~\ref{thm:bm-functional} give the asserted estimate.
\end{proof}
\endgroup

\providecommand{\bysame}{\leavevmode\hbox to3em{\hrulefill}\thinspace}
\providecommand{\MR}{\relax\ifhmode\unskip\space\fi MR }
\providecommand{\MRhref}[2]{%
  \href{http://www.ams.org/mathscinet-getitem?mr=#1}{#2}
}
\providecommand{\href}[2]{#2}

\end{document}